\documentclass[a4paper,12pt,reqno]{amsart}
\usepackage{amssymb,amsfonts,amsmath,amscd,amsthm,latexsym,euscript}
\usepackage{times}

\usepackage{combelowedit} 

\usepackage{enumitem}
\usepackage{tikz}

\usepackage[bookmarks=false]{hyperref}
\hypersetup{%
    colorlinks=true,        
    linkcolor=blue,          
    citecolor=blue,         
    urlcolor=blue           
    }

\usepackage{xcolor}
\usepackage{listings} 
\definecolor{codegreen}{rgb}{0,0.6,0}
\definecolor{codegray}{rgb}{0.5,0.5,0.5}
\definecolor{codepurple}{rgb}{0.58,0,0.82}
\definecolor{backcolour}{rgb}{0.95,0.95,0.92}
\lstdefinestyle{mystyle}{
    backgroundcolor=\color{backcolour},   
    commentstyle=\color{codegreen},
    keywordstyle=\color{magenta},
    numberstyle=\tiny\color{codegray},
    stringstyle=\color{codepurple},
    basicstyle=\ttfamily\footnotesize,
    breakatwhitespace=false,         
    breaklines=true,                 
    captionpos=b,                    
    keepspaces=true,                 
    numbers=left,                    
    numbersep=5pt,                  
    showspaces=false,                
    showstringspaces=false,
    showtabs=false,                  
    tabsize=2
}
\newtheorem{theor}{Theorem}

\theoremstyle{definition}

\newtheorem{prop}[theor]{Proposition}
\newtheorem{theordef}[theor]{Theorem--Definition}
\newtheorem{lem}[theor]{Lemma}

\newtheorem{manuallemmainner}{Lemma} 
\newenvironment{manuallemma}[1]{%
  \IfBlankTF{#1}
    {\renewcommand{\themanuallemmainner}{\unskip}}
    {\renewcommand\themanuallemmainner{#1}}%
  \manuallemmainner
}{\endmanuallemmainner}

\newtheorem{cor}[theor]{Corollary}
\newtheorem{conjecture}[theor]{Conjecture}
\newtheorem{define}{Definition}

\newtheorem*{resprob}{Research problem}
\newtheorem*{comm}{Comment} 

\newtheorem{notation}{Notation}
\newtheorem{problem}{Problem}
\newtheorem{open}[problem]{Open problem}
\newtheorem{ex}{Example}
\theoremstyle{remark}
\newtheorem{rem}{Remark}
\newtheorem{note}[rem]{Note}

\theoremstyle{definition}
\theoremstyle{definition}

\newcommand{\BBR}{\mathbb{R}}\newcommand{\BBC}{\mathbb{C}}
\newcommand{\BBN}{\mathbb{N}} \newcommand{\BBNo}{\BBN_{\geqslant 0}}

\newcommand{\BBZ}{\mathbb{Z}}
\newcommand{\BBQ}{\mathbb{Q}}

\newcommand{\EuA}{{{\EuScript A}}}
\newcommand{\EuL}{{{\EuScript L}}}

\newcommand{\cA}{{{\EuScript A}}}

\newcommand{\cX}{{\EuScript X}}    

\newcommand{\Bone}{{\boldsymbol{1}}}
\newcommand{\ba}{{\boldsymbol{a}}}
\newcommand{\bb}{{\boldsymbol{b}}}
\newcommand{\bc}{{\boldsymbol{c}}}
\newcommand{\bd}{{\boldsymbol{d}}}
\newcommand{\be}{{\boldsymbol{e}}}

\newcommand{\bbf}{{\boldsymbol{f}}}
\newcommand{\bi}{{\boldsymbol{i}}}
\newcommand{\bk}{{\boldsymbol{k}}}
\newcommand{\bn}{{\boldsymbol{n}}}
\newcommand{\bm}{{\boldsymbol{m}}}

\newcommand{\br}{{\boldsymbol{r}}}

\newcommand{\bivec}{{\vec{\boldsymbol{i}}}}
\newcommand{\brvec}{{\vec{\boldsymbol{r}}}}
\newcommand{\bnvec}{{\vec{\boldsymbol{n}}}}
\newcommand{\bmvec}{{\vec{\boldsymbol{m}}}}
\newcommand{\bkvec}{{\vec{\boldsymbol{k}}}}

\newcommand{\bs}{{\boldsymbol{s}}}

\newcommand{\bx}{{\boldsymbol{x}}}

\newcommand{\gothg}{\mathfrak{g}}

\newcommand{\bigzero}{\makebox(0,0){\text{\huge0}}}

\DeclareMathOperator{\Van}{Van}

\DeclareMathOperator{\Span}{span}

\newcommand{\wedges}{\wedge\cdots\wedge}

\newcommand{\by}[1]{\textrm{{#1}}}
\newcommand{\jour}[1]{\textit{{#1}}}
\newcommand{\vol}[1]{\textbf{{#1}}}
\newcommand{\book}[1]{\textit{{#1}}}

\begin{document}

\title[Polynomial ${\EuScript A} \subseteq \Bbbk{[x^1,\ldots,x^d]}$ 
with Wronskian as
$N$\nobreakdash-ary Lie bracket: $\dim{\EuScript A} < \infty$\,?]
{Explicit class of finite\/-\/dimensional polynomial algebras\\[3pt] with Wronskians over~$\mathbb{R}^d$ as $N$\nobreakdash-ary Lie brackets: beyond $\mathfrak{sl}(2)$}

\author[M.\,G.\,\cb{K}\=eni\cb{n}\v{s}]{Markuss G. \cb{K}\=eni\cb{n}\v{s}${}^{*,\star}$}
\thanks{${}^{*}$\:\textit{Address}:\quad 
Bernoulli Institute for Mathematics, Computer Science \&\ 
Artificial Intelligence, 
University of Groningen, P.O.\,Box\:407, 9700\,AK Groningen, 
the Netherlands.
}
\thanks{${}^{\star}$\:\textit{Address for correspondence} (after 1 September 2026):
ETH Z\"{u}rich -- Department of Mathematics, 
R\"{a}mistrasse 101, CH-8092 Z\"{u}rich, Switzerland.%
}

\author[A.\,V.\,Kiselev]{Arthemy V.\ Kiselev${}^{*,\S}$}
\thanks{${}^{\S}$\:Corresponding author. \textit{E-mail}: \texttt{A.V.Kiselev\symbol{"40}rug.nl}%
}

\dedicatory{\textup{Based on the talks given 
by the first author at the Algebra seminar (Bernoulli Institute, Groningen,\\[1pt] the Netherlands)
and 
by the last author at the Prague Mathematical Physics seminar (Charles\\[1pt] University, Czech Republic) and at the Mathematics seminar (IH\'ES, Bures\/-\/sur\/-\/Yvette, France).}}



\date{10 July 2026}

\begin{abstract}
Lie algebra $\mathfrak{sl}(2)$ can be realised by vector fields on $\mathbb{R}^1\ni x$ with polynomial coefficients $1$,\ $-2x$,\ $-x^2$; their Wronskian determinants yield the Lie bracket. Likewise, the monomials $1$,\ $\ldots$,\ $x^k/k!$,\ $\ldots$, $x^N/N!$ span finite\/-\/dimensional strong homotopy (SH) Lie algebras with the Wronskians $\mathbf{1} \wedge \partial_x \wedge \ldots \wedge \partial_x^{N-1}$ as the $N$\nobreakdash-ary brackets.
Over dimension $d=2$ with $\mathbb{R}^2\ni(x,y)$ and for the complete generalised Wronskian $W^{k=1}_{d=2}=\mathbf{1}\wedge \partial_x \wedge \partial_y$ of differential order $k=1$ as the ternary bracket, the finite\/-\/dimensional polynomial SH-\/Lie algebras are spanned by $\langle 1$,\ $x$,\ $y$,\ $p\rangle$ with $p\in\{x^2$,\ $xy$,\ $y^2\}$.

We explicitly describe all finite\/-\/dimensional polynomial SH-\/Lie algebras $
\Bbbk_k[\bx]\subseteq{\EuScript A}\subseteq\Bbbk[x^1,\ldots,x^d]$ (over $\Bbbk=\mathbb{R}$ or~$\mathbb{C}$) with the complete generalised Wronskians $W_{d\geqslant 1}^{k\geqslant 1}$ of order~$k$ as $N$\nobreakdash-ary brackets: $N=\binom{d+k}{d}$. We obtain a factorisation formula for the generalised Vandermonde determinants which show up in the structure constants of the polynomial algebras~${\EuScript A}$.
\end{abstract}
\maketitle

\subsection*{Introduction}\label{SecIntroduction}
\noindent%
On the real line~$\BBR^1$, take finitely many vector fields $\vec{\cX}_j\in\mathbb{D}_{p=1}(\BBR^1)$ with smooth, in particular polynomial, coefficients $X_j(x)$, and generate the Lie algebra by taking the iterated commutators. Kirillov asked: how fast does it grow\,? Specifically, what is the dimension of the subspace spanned by all products of generating elements~$X_j$, with number of factors${}\leqslant n$. Kontsevich and Molev joined this investigation of Lie algebras of intermediate growth: see~\cite{KirillovKontsevich} 
and~\cite{KirillovKontsevichMolev}. 

Suppose now that the Lie algebra of (polynomial) vector fields is deformed, with the new brackets extended by alternating compositions of differential operators of strict order $p\in\BBN_{\geqslant 1}$, to a strongly homotopy (SH) Lie algebra, see~\cite{Dzhuma2002} and \cite{ForKac,PRG25Wronsk} as well as~\cite{LadaStasheff1993,Stasheff2019Review} with more references therein.
The multi\/-\/linear antisymmetric brackets of even arities $N=2p$ are expressed by the Wronskian determinants (attaching to the Lie case at $p=1$ and $N=2$, see~\cite{cpWron}). 
Dzhumadil'daev showed in~\cite{Dzhuma2002} that the entire table of strongly homotopy Jacobi identities remains valid for arbitrary arities of the Wronskians as brackets over~$\BBR^1$.
In~\cite{ForKac}, the second author re\/-\/proved this fact for complete generalised Wronskians over multi\/-\/dimensional base, $d=\dim\BBR^d \geqslant 1$, then relaxing in~\cite{AVK25YerQTS13WrosnkBJP} the assumption that either Wronskian be complete in the Jacobi\/-\/type identities which are bi\/-\/linear w.r.t.\ the brackets.

Kirillov's question -- how fast does the algebra grow under the iteration of Wronskian brackets\,?\,-- now applies to the $N$\nobreakdash-ary structures of differential order $k\geqslant 1$ over the base dimension $d\geqslant 1$. (We have that $N=\binom{d+k}{d}$ for the complete generalised Wronskians $W_{d\geqslant 1}^{k\geqslant 1}$.)

To install the reference point along this line of research, 
in the present paper we classify the polynomial SH-Lie algebras ${\EuScript A}\supseteq \Bbbk_k[x^1,\ldots,x^d]$ which \emph{do not grow} at all but remain finite\/-\/dimensional. Surprisingly, the internal structure of these algebras over arbitrary dimension $d\geqslant 1$ patterns upon the $N$\nobreakdash-ary generalisations 
$\bigl( \Bbbk_N[x]$, $W_{d=1}^{0,1,...,N-1}=\mathbf{1}\wedge \partial_x\wedge\ldots\wedge\partial_x^{N-1} \bigr) $, here $N=k+1$, of the Lie algebra $\mathfrak{sl}(2)$, itself realised on the line~$\BBR^1$ by using quadratic vector fields.
Let us begin by recalling that classical construction. 

\begin{ex}[realisation of $\mathfrak{sl}(2)$ by polynomial vector fields]\label{ex:realis-sl2} Let $e:=1$, $h:=-2x$, and $f:=-x^2$ be monomials on the real line $\BBR\ni x$. Take the Wronskian determinant $W^{0,1}=\Bone\wedge\partial/\partial{x}$ of two functions on $\BBR$ as the bi-linear antisymmetric bracket. Then we readily see that
\begin{align}
    W^{0,1}(h,e)&=2e,& W^{0,1}(h,f)&=-2f,& &\text{and}& W^{0,1}(e,f)&=h,
\end{align}
which are the standard relations for $\mathfrak{sl}(2)$ w.r.t\ the Chevalley--Cartan basis; a verification that the Wronskian determinant on $C^\infty(\BBR)$ satisfies the Jacobi identity is straightforward.

We thus realise $\mathfrak{sl}(2)$ by using three vector fields with polynomial coefficients over $\Bbbk=\BBC$ (or $\Bbbk=\BBR$).
\end{ex}

By increasing the differential order of the Wronskian $W^{0,1,...,k}=\Bone\wedge \partial_x\wedge\cdots\wedge \partial_x^k$ to arbitrary $k\geqslant 1$, and now letting $N=k+1$, we endow the space of polynomials $\Bbbk[x]$ in one variable $x$ with an $N$-linear (over the number field $\Bbbk$) antisymmetric bracket:
\begin{equation}
    W^{0,1,...,N-1}(p_1,...,p_N)=\det \Bigl( {\partial^{i-1} p_j}/{\partial x^{i-1}} \Bigr)_{ij}(x)
\end{equation}
for $p_1,...,p_N\in \Bbbk[x]$ and $i,j\in \{1,...,N\}$. From \cite{Dzhuma2002} we recall that the $N$\nobreakdash-ary Wronskian brackets satisfy the $N$\nobreakdash-ary Jacobi identities for strong homotopy SH-Lie algebras (see \cite{LadaStasheff1993,Stasheff2019Review}, and \cite{PRG25Wronsk} for definitions and details). In particular, we know a class of finite-dimensional $N$\nobreakdash-ary SH-Lie algebras generalising $\mathfrak{sl}(2)$ with its binary structure from Example~\ref{ex:realis-sl2}.

\begin{ex}[from \cite{ForKac}]\label{ex:one-dim-finiteddim}
    Over dimension $d=1$, let $x$ be the coordinate on $\Bbbk=\BBR$. The Wronskian $W^{0,1,...,N-1}=\Bone\wedge\partial_x\wedges\partial_x^k$ of differential order $k=N-1$ acting on the subspace of polynomials of degree at most $N$, so $\EuA=\Bbbk_N[x]\subsetneq \Bbbk[x]$, makes $\EuA$ into an $(N+1)$-dimensional SH-Lie algebra: for any degree $\ell\in \{0,1,...,N\}$ we have
    \begin{equation}\label{ex:removed-factorial}
        W_d^{0,1,...,N-1}\left(1,...,\widehat{\frac{x^\ell}{\ell!}},...,\frac{x^N}{N!}\right) = \frac{x^{N-\ell}}{(N-\ell)!},
    \end{equation}
    which makes all the non-automatically-zero structure constants equal to $\pm1$ in the basis $\{x^m/m!: m\in\{0,1,...,N\}\}$.
\end{ex}

Now, we increase the dimension to $d=2$ and we let $x,y$ be the Cartesian coordinates on $\BBR^2$. Consider the generalised Wronskian $W_{d=2}^{k=1}=\Bone\wedge\partial_x\wedge\partial_y$ of differential order $k=1$: for any smooth functions $f,g,h\in C^\infty(\BBR^2)$, put
\begin{equation}\label{eq:wronskian-k2-d1}
    W^{k=1}_{d=2}(f,g,h)(x,y)=\det \begin{pmatrix}
        f&g&h\\ f_x&g_x&h_x\\ f_y&g_y&h_y
    \end{pmatrix}(x,y).
\end{equation}
As seen from \cite{ForKac} (also \cite{AVK25YerQTS13WrosnkBJP}), this ternary skew bracket equips the space of smooth functions $C^\infty (\BBR^2)$ with the SH-Lie algebra structure. 

Until now we knew three examples of finite-dimensional polynomial SH-Lie algebras with the ternary bracket \eqref{eq:wronskian-k2-d1}; here they are:

\begin{ex}\label{ex:dimension-two-closed-old}
    Over dimension $d=2$, take the generalised Wronskian $W_{d=2}^{k=1}=\Bone\wedge\partial_x\wedge\partial_y$ of differential order $k=1$ and the ternary bracket $[\cdot,\cdot,\cdot]$ on the space of smooth functions. Now in $C^\infty(\BBR^2)\supsetneq \BBR[x,y]$, consider three distinct four-dimensional subspaces of degree-two polynomials spanned by $1,x,y$ and a specific monomial of degree two:\\
    $\bullet$\quad$\EuA(x^2)=\Span_\BBR \langle 1,x,y,x^2\rangle$; the ternary commutation relations are
    \begin{align*}
        [1,x,y]&=1,& [1,x,x^2]&=0,& [1,y,x^2]&=-2x,& [x,y,x^2]&=-x^2;
    \intertext{$\bullet$\quad$\EuA(xy)=\Span_\BBR \langle 1,x,y,xy\rangle$; the ternary brackets are (cf.\ \cite{ForKac})}
        [1,x,y]&=1,& [1,x,xy]&=x,& [1,y,xy]&=-y,& [x,y,xy]&=-xy;
    \intertext{$\bullet$\quad$\EuA(y^2)=\Span_\BBR \langle 1,x,y,y^2\rangle$; we have that}
        [1,x,y]&=1,& [1,x,y^2]&=2y,& [1,y,y^2]&=0,& [x,y,y^2]&=-y^2.
    \end{align*}
    All three polynomial SH-Lie algebras are closed under the ternary bracket and thus four-dimensional.
\end{ex}

\begin{rem}\label{rem:imperfection-old-examples}
    The SH-Lie algebras $\Bbbk_N[x]$ and $\EuA(xy)$ from Examples~\ref{ex:one-dim-finiteddim} and~\ref{ex:dimension-two-closed-old}, respectively, are \emph{perfect}, that is, the $N$\nobreakdash-ary bracket has surjective span:\footnote{In literature the notation $[\gothg,\gothg]$ is already understood to be the span, i.e.\ $[\gothg,\gothg]:=\Span_{\Bbbk}[\gothg,\gothg]$. An alternative notation for the span is $\langle[\gothg,\gothg]\rangle$.} $\Span_\BBR [\EuA,\EuA,\EuA]=\EuA$. In contrast, the other algebras in Example~\ref{ex:dimension-two-closed-old} are not perfect (as the ternary bracket is missing a monomial from its image).

    In Proposition~\ref{prop:imperfection} below it will be proven that the two algebras -- over dimension $d=1$, $\Bbbk_N[x]$ with $k=N-1$, and $\EuA(xy)$ over $d=2$ -- are the only examples of perfect finite-dimensional polynomial SH-Lie algebras with complete generalised Wronskians as the $N$\nobreakdash-ary brackets, \emph{under the extra assumption} that the underlying vector space is spanned by \emph{pure monomials}.
\end{rem}

Finally, we have the fourth class of finite-dimensional examples.

\begin{ex}\label{ex:trivial-finite-dim-algebra}
    Fix the dimension $d\geqslant 1$ and differential order $k\geqslant 1$, and take the span,
    \begin{equation}
        \Span_\Bbbk\{1,x^1,...,x^d,(x^1)^2,x^1x^2,...,(x^d)^k\},
    \end{equation}
    of all monomials up to -- including but not exceeding -- degree $k$. We recall that there are $N={d+k\choose k}$ such monomials; this obviously is the number of arguments in the complete generalised Wronskian $W_{d\geqslant 1}^{k\geqslant 1}$, as the derivations in it stand in a bijective correspondence with the monomials in our set (cf.\ Definition~\ref{def:divpowermonomial-standardmonomial} in Preliminaries).

    Up to a reshuffle of all pairwise distinct arguments, the only non-zero bracket then is
    \begin{equation}
        W_{d\geqslant 1}^{k\geqslant 1}\Bigl(1,x^1,...,x^d,\frac{(x^1)^2}{2!},\frac{x^1x^2}{1!1!},...,\frac{(x^d)^k}{k!}\Bigr) = +1.
    \end{equation}
    This $N$-dimensional polynomial SH-Lie algebra is trivially closed, and it is not perfect for all differential orders $k\geqslant 1$ and dimensions $d\geqslant 1$.
\end{ex}

In Theorem~\ref{thm:general-vandermonde} we shall prove the following Lemma as a by-product.

\begin{lem}
    Over any base dimension $d\geqslant 1$ and any differential order $k\geqslant 1$, the complete generalised Wronskian determinant of monomials from $\Bbbk[\bx]$ is again a monomial.

    \noindent$\bullet$\quad By multilinearity then, the Wronskian determinant of polynomials is again a polynomial in $\Bbbk[\bx]$.
\end{lem}

Acting as the $N$\nobreakdash-ary bracket, the complete generalised Wronskian determinant $W_{d\geqslant1}^{k\geqslant 1}=\Bone\wedge \partial_\bx\wedges \partial^k_{\bx...\bx}$ endows the space $\Bbbk[\bx]$ of polynomials over $\Bbbk^d$ with a structure of strong homotopy Lie algebra (see \cite{ForKac,PRG25Wronsk} and references therein, also \cite{LadaStasheff1993}).

\begin{define}\label{def:dimension-of-polynomial-subalgebra}
    Let $\EuA\subseteq \Bbbk[\bx]=\Bbbk[x^1,...,x^d]$ be a subalgebra of the $N$\nobreakdash-ary SH-Lie algebra of polynomials $\Bbbk[\bx]$ with the Wronskian determinant of differential order $k$ over dimension $d$ as $N$\nobreakdash-ary bracket, $N={d+k\choose k}$; i.e.\ a vector subspace closed under the Wronskian $N$\nobreakdash-ary bracket $[\cdot,...,\cdot]=W_d^k(\cdot,...,\cdot)$ on $\Bbbk[\bx]$. The \emph{dimension} of $\EuA$ is the dimension of the subspace $\EuA$ viewed as a vector space over $\Bbbk$. In particular, $\dim_\Bbbk(\Bbbk[\bx])=+\infty$.
\end{define}

\begin{define}\label{def:perfect-lie-algebra}
    A Lie algebra $\mathfrak{g}=(V,[\cdot,\cdot])$ over $\Bbbk$ is called \emph{perfect} if $\Span_\Bbbk[\gothg,\gothg]=\gothg$.

    Now, in our case of $N$\nobreakdash-ary structures, specifically, $N={d+k\choose k}$, we say that the polynomial $N$\nobreakdash-ary SH-Lie algebra $\EuA$ with the Wronskian $W_{d\geqslant 1}^{k\geqslant 1}$ as the bracket is \emph{perfect} if the bracket spans the whole algebra: $\Span_\Bbbk W_d^k (\EuA,...,\EuA)=\EuA$. 
\end{define}

Any simple or semi-simple Lie algebra is immediately perfect. Finite-dimensional (semi-) simple complex Lie algebras are of great interest due to their classification by Dynkin diagrams and construction of their commutator tables from root systems. This leads to our main research problem.

\begin{resprob}\label{eq:research-problem}
    For every dimension $d\geqslant 1$ and every differential order $k\geqslant 1$ do there exist and, if so, what are all (perfect) \emph{finite-dimensional} polynomial $N$\nobreakdash-ary SH-Lie, $N={d+k\choose k}$, subalgebras $\EuA\subseteq \Bbbk[x^1,...,x^d]$ with the (complete) generalised Wronskian determinant $W_{d}^k = \Bone\wedge \partial_\bx\wedges \partial^k_{\bx...\bx}$ as the $N$\nobreakdash-ary Lie bracket?
\end{resprob}


{\small
\noindent
\parbox{77mm}{
    \textbf{\normalsize List of notation.} $\bullet$~ $\Bbbk=\BBR$ or $\BBC$, base field;\\
    $\bullet$~ $d=\dim_\BBR \BBR^d = \dim_\Bbbk \Bbbk^d$;\\
    $\bullet$~ $\bx=(x^1,...,x^d)$ Cartesian coordinates on $\BBR^d$;\\
    $\bullet$~ $\Bbbk[x^1,...,x^d]$, space of polynomials over $\BBR^d$;\\
    $\bullet$~ $\Bbbk_k[x^1,...,x^d]$ polynomials of degree $\leqslant k$;\\
    $\bullet$~ $N=\tbinom{d+k}{d}=\tbinom{d+k}{k}$, arity -- Definition~\ref{def:wronskian};\\
    $\bullet$~ perfect SH-Lie algebra $\EuA=\Span_\Bbbk [\EuA,...,\EuA]$;\\ 
    $\bullet$~ power $d$-tuple $\brvec=(r^1,...,r^d)\in\BBQ^d$;\\
    $\bullet$~ $\bx^\brvec=(x^1)^{r^1}\cdots (x^d)^{r^d}$ arbitrary monomial;\\
    $\bullet$~ $\partial^{|\brvec|}/\partial \bx^{\brvec} = \partial^{r^1}/\partial (x^1)^{r^1}\circ \cdots \circ \partial^{r^d}/\partial (x^d)^{r^d}$ differential operator -- cf.\ Eq.~\eqref{eq:differential-operator-acting-on-monomial};\\
}\hspace{3mm}\parbox{77mm}{
    $\bullet$~ $\bx^\brvec \leftrightarrows \partial^{|\brvec|}/\partial \bx^{\brvec_j}$ -- differential operators vs divided power monomials -- Definition~\ref{def:divpowermonomial-standardmonomial};\\
    $\bullet$~ operator in the $j$th row of the Wronskian: $\partial^{|\brvec_j|}/\partial \bx^{\brvec_j}$ -- Definition~\ref{def:divpowermonomial-standardmonomial};\\
    $\bullet$~ standard monomial $\bx^{\brvec_j}$ -- Definition~\ref{def:divpowermonomial-standardmonomial};\\
    $\bullet$~ $\det \Van_{d=1}^{0,1,...,N-1}$ ordinary Vandermonde det.;\\
    $\bullet$~ $\det \Van_{d\geqslant 1}^{k\geqslant 1}$ generalised 
    -- Definition~\ref{def:gen-vandermonde};\\
    $\bullet$~ $\det \Delta_k^d$, quasi-triangular form 
    -- Definition~\ref{def:quasi-triangular-form};\\
    $\bullet$~ $\bm^{\brvec}=(m_1)^{r^1}\cdots (m_d)^{r^d}$;\\
    $\bullet$~ $\bkvec=\bk\in \BBQ^d$ power $d$-tuple;\\
    $\bullet$~ $kN/(d+1)$ degree shift -- Lemma~\ref{lem:combinatorial};\\
}

\noindent\parbox[t]{77mm}{
    $\bullet$~ deficiency condition -- Proposition~\ref{prop:sufficient-vandermonde-zero}\,\ref{prop8-iii:prop-sufficient-vandermonde:deficiency};\\
    $\bullet$~ $s=k/(d+1)\cdot N/(N-1)$ -- Remark~\ref{rem:multivar-gen-of-witt-algebra};\\
}\hspace{3mm}\parbox[t]{77mm}{
    $\bullet$~ consistent algebra, $\Bbbk_k[\bx]\subseteq \EuA\subseteq \Bbbk[\bx]$, 
    trivial, lonely, chubby, tall, lanky algebras -- Definition~\ref{def:consistent-lonely-chubby-tall-lanky}.
    }
}

\centerline{\rule{3.5in}{0.7pt}}

\subsection*{Preliminaries}\phantom{.}

\medskip
\noindent%
Let us recall some notions and fix notation.

\begin{notation}\label{not:multi-index}
    Let $\Bbbk=\BBR$ (or $\BBC$) be the ground field; let $\bx=(x^1,...,x^d)$ be the (Cartesian) coordinates on $\BBR^d$, where $d$ is the base dimension. By convention, put $\partial_{x^i}=\partial/\partial x^i$ and $\partial_\bx=\partial_{x^1}\wedges \partial_{x^d}$; also, put $\partial^2_{\bx\bx}=\partial^2_{x^1x^1}\wedges\partial^2_{x^ix^j}\wedges\partial^2_{x^dx^d}$ and so on until $\partial^k_{\bx...\bx}=\partial^k_{x^1....x^1}\wedges \partial^k_{x^d...x^d}$. Denote the identity operator by $\Bone$.

    \noindent$\bullet$\quad Let $\vec\br=(r^1,...,r^d)$ be a \emph{$d$-tuple} (also called multi-index) in $\BBN_{\geqslant 0}^d$ (occasionally in $\BBZ^d$ or $\BBQ^d$). We denote by $\bx^\brvec$ the monomial $\bx^\brvec = (x^1)^{r^1}\cdots (x^d)^{r^d}$; and denote by $\partial^{|\brvec|}/\partial \bx^\brvec$ the differential operator
    \begin{equation}\label{eq:def-of-differential-operator}
        \frac{\partial^{|\brvec|}}{\partial \bx^\brvec} := \frac{\partial^{r^1}}{\partial (x^1)^{r^1}} \cdots \frac{\partial^{r^d}}{\partial (x^d)^{r^d}}.
    \end{equation}
    For instance, the $2$-tuple $\brvec=(2,1)$ gives $\partial^3/\partial x^2\partial y$, where $x:=x^1$ and $y:=x^2$. By definition, put $\brvec!=r^1!\cdots r^d!$. Wherever convenient, we shall lower the indices of the power $d$-tuple $\brvec$ without change of meaning; for instance, setting $\brvec=\bnvec=(n_1,...,n_d)$, the monomial $\bx^\brvec$ becomes $(x^1)^{n_1}\cdots (x^d)^{n_d}$.

    \noindent$\bullet$\quad Arbitrary monomials will usually be denoted by $a(\bx)$, whereas polynomials by $p(\bx)$ or $q(\bx)$; however, exceptions are possible (in \S\ref{subSecNonExistence}; cf.\ Footnote~\ref{foot:consideration-of-pure-monomials} for the reasoning).
\end{notation}

\begin{define}\label{def:wronskian}
    Consider the algebra of smooth functions functions over coordinates $\bx=(x^1,...,x^d)$; we say that $d\geqslant 1$ is the \emph{dimension}. Let the \emph{total differential order} be $k\geqslant 1$, and set $N={d+k\choose k}={d+k\choose d}$. The $N$-linear (over $\Bbbk$) totally antisymmetric, $C^\infty (\BBR^d)$-valued differential operator $W_d^k=\Bone\wedge \partial_\bx \wedges \partial^k_{\bx...\bx}$ is the \emph{complete generalised Wronskian determinant} of differential order $k$ over dimension $d$.
\end{define}

\begin{notation}\label{not:jth-row-wronskian}
    Denote by $\partial^{|\brvec_j|}/\partial \bx^{\brvec_j}$ the differential operator in the $j$th row, $j\in \{1,...,N\}$, of the Wronskian $W_d^k$. In this notation, $W_d^k=\bigwedge_{j=1}^N \partial^{|\brvec_j|}/\partial \bx^{\brvec_j}$. We say that $r_j^i$ is the \emph{differential order} in coordinate $x^i$ and $r_j=|\brvec_j|=r_j^1+\cdots+r_j^d$ is the \emph{total differential order} of the differential operator $\partial^{|\brvec_j|}/\partial \bx^{\brvec_j}$.
\end{notation}

\begin{ex}\label{ex:lowest-dim-order-wronskians}
    Take smooth functions $f,g,h\in C^\infty(\BBR^1)$ and put $W_{d=1}^{k=1}=\Bone\wedge \partial/\partial x$ and $W_{d=1}^{k=2}=\Bone\wedge \partial/\partial x\wedge \partial^2/\partial x^2$. Then the Wronskians,
    \begin{align}\label{eq:lowest-dim-order-wronskians}
        W_{d=1}^{k=1}(f,g)(x) &= \det \begin{vmatrix}
            f&g\\ f'&g'
        \end{vmatrix}(x), &
        W_{d=1}^{k=2}(f,g,h)(x) &= \det \begin{vmatrix}
            f&g&h \\ f'&g'&h'\\ f''&g''&h''
        \end{vmatrix}(x),
    \end{align}
    produce smooth functions on the real line $\BBR^1$.

    Alternatively, let the dimension be $d=2$ and the differential order be $k=1$. Then for smooth functions $f,g,h\in C^{\infty}(\BBR^2)$ we have 
    \begin{equation}\label{eq:wronskian-d2-k1}
        W_{d=2}^{k=1}(f,g,h)(\bx)=\det \begin{vmatrix}
            f&g&h\\ f_x&g_x&h_x\\ f_y&g_y&h_y
        \end{vmatrix}(\bx),
    \end{equation}
    where $f_x:=\partial f/\partial x$ for $x=x^1$ and $y=x^2$ and which is again a smooth function over $\BBR^2$.

    In Notation~\ref{not:multi-index} the above Wronskians become
    \begin{align}
        W_{d=1}^{k=1}&=\Bone\wedge \partial_x = \frac{\partial^{|\brvec_1|}}{\partial \bx^{\brvec_1}} \wedge \frac{\partial^{|\brvec_2|}}{\partial \bx^{\brvec_2}}, &
        W_{d=2}^{k=1} &= \frac{\partial^{|\brvec_1|}}{\partial \bx^{\brvec_1}} \wedge \frac{\partial^{|\brvec_2|}}{\partial \bx^{\brvec_2}} \wedge \frac{\partial^{|\brvec_3|}}{\partial \bx^{\brvec_3}},
    \end{align}
    where we observe that in all cases $\partial^{|\brvec_1|}\partial \bx^{\brvec_1}$ is the identity $\Bone$; by definition, then $r_1=|\brvec_1|=0$ and $r_1^i=0$ for all $i\in\{1,...,d\}$. The second operator is always of order 1, namely, $r_2=|\brvec_2|=1$ with $r_2^i=1$ for a single $i\in \{1,...,d\}$ and all others vanishing: $r_2^{i'}=0$ for $i'\neq i$.
\end{ex}

Let us recall a natural correspondence between monic single-term differential operators, like $\partial^3/\partial x^2\partial y$, and monomials (resp., $x^2y$) over $\BBR^d$, and functions $ (m^1,...,m^d)\mapsto (m^1)^2\cdot m^2$ from $\BBQ^d$.

\begin{define}[divided power monomial and standard monomial]\label{def:divpowermonomial-standardmonomial}
    In a \emph{fixed} system of coordinates $\bx=(x^1,...,x^d)$ on $\BBR^d$, let $D=\partial^{|\brvec|}/\partial \bx^\brvec$ be a single-term monic differential operator of total differential order $r=|\brvec|$. The \emph{divided power monomial} $a(\bx)$ corresponding to $D$ is, by definition, the (unique) monomial in $\Bbbk[\bx]$ satisfying $D(a)(\bx)\equiv 1$, the unit constant polynomial.

    \noindent$\bullet$\quad The formula $a(\bx)=\bx^\brvec/\brvec! = (x^1)^{r^1}\cdots (x^d)^{r^d}/r^1!\cdots r^d!$ obviously satisfies $D(a)(\bx)\equiv 1$ for $D=\partial^{|\brvec|}/\partial \bx^\brvec$ as above with any power $d$-tuple (multi index) $\brvec$ in $\BBN_{\geqslant 0}^d$. In particular, for $D=\partial^{|\brvec_j|}/\partial\bx^{\brvec_j}$, the differential operator of the $j$th row of the Wronskian $W_d^k=\Bone \wedge \partial_\bx\wedges \partial^k_{\bx...\bx}$ (of differential order $k\geqslant 1$ over dimension $d\geqslant 1$), the divided power monomial corresponding to $\partial^{|\brvec_j|}/\partial \bx^{\brvec_j}$ is $\bx^{\brvec_j}/\brvec_j!$.

    The \emph{standard set} of monomials w.r.t.\ the Wronskian $W_d^k$ is $\{\bx^{\brvec_1}/\brvec^1,...,\bx^{\brvec_N}/\brvec_N!\}$, and the divided power monomial $\bx^{\brvec_j}/\brvec_j!$ is called the \emph{standard monomial corresponding to the $j$th row of the Wronskian}.

    \noindent$\bullet$\quad The correspondence also sends a \emph{monomial} to its evaluation map on $\bm\in\BBQ^d$:
    \begin{equation}
    \begin{split}
        \Bbbk[\bx] &\to (\BBQ^d\to \BBQ)\\
        \bx^{\brvec} &\mapsto (\bm\mapsto \bm^{\brvec}).
    \end{split}
    \end{equation}
    Basically, $x^2y^3$ becomes the map $\BBQ^{d=2}\to \BBQ$ given by $(m_1,m_2)\mapsto (m_1)^2 (m_2)^3$.
\end{define}

Let $a(\bx)=\bx^\bnvec=(x^1)^{n_1}\cdots (x^d)^{n_d}$ be a monomial with power $d$-tuple $\bnvec\in \BBN_{\geqslant 0}^d$. The partial derivative of $a(\bx)$ with respect to $x^1$ is then $\partial a/\partial x^i= n_i\cdot a/x^i$; this extends to higher-degree differential operators $\partial^{|\brvec|}/\partial \bx^\brvec$ by
\begin{equation}\label{eq:differential-operator-acting-on-monomial}
    \frac{\partial^{|\brvec|}a}{\partial \bx^{\brvec}}(\bx) = \frac{a}{\bx^\brvec} \cdot \prod_{i=1}^d \left\{ \prod_{\ell =0}^{r^i-1} (n_i-\ell) \right\}.
\end{equation}
If, for some $i\in \{1,...,d\}$ the differential order $r^i$ exceeds the degree of $x^i$ in $a(\bx)$, i.e.\ $r^i>n_i$, then the product $\prod_{\ell=0}^{r^i-1}(n_i-\ell)$ vanishes, and therefore the `negative' degree of $a/(x^i)^{r_i}$ in $x^i$ is fictitious as the entire r.\!-h.s.\ of \eqref{eq:differential-operator-acting-on-monomial} is identically zero.

\section{Multivariate Vandermonde determinants are coefficients of multivariate Wronskians}\label{SecVander}

\noindent In this section we generalise the Vandermonde determinant to an arbitrary base dimension $d\geqslant 1$ in a way that respects the structure of the complete generalised Wronskian determinant $W_d^k=\Bone \wedge \partial_\bx\wedges \partial^k_{\bx...\bx}$ of differential order $k\geqslant 1$ over dimension $d\geqslant 1$. The complete generalised Vandermonde determinant of degree $k$ over dimension $d$ appears (see Theorem~\ref{thm:general-vandermonde} on p.~\pageref{thm:general-vandermonde}) as the coefficient of the monomial resulting from a Wronskian determinant of monomial arguments; the arguments of this Vandermonde determinant are the power $d$-tuples of the monomial arguments in the Wronskian. We use the correspondence between monic single-term differential operators, monomials, and evaluation functions on $\BBQ^d$: for example,
\begin{equation}
    \partial^5/\partial x^2\partial y^3 \leftrightarrows x^2 y^3\leftrightarrows \{(m_1,...,m_d)\mapsto (m_1)^2 (m_2)^3\},
\end{equation}
where $\bm=(m^1,...,m^d)$ is an argument of the Vandermonde determinant; see Definition~\ref{def:divpowermonomial-standardmonomial}.

The main theorem in this section is a generalisation of Theorem 13 from \cite{ForKac} (see Theorem~\ref{thm:general-vandermonde}) specifically from over dimension 1 to arbitrary dimension $d\geqslant 1$; that is, the (complete) generalised Wronskian determinant (of differential order $k\geqslant 1$ over dimension $d\geqslant 1$) of monomial arguments is again a monomial with the coefficient given by the complete generalised Vandermonde determinant of the power $d$-tuples of the monomial arguments. As part of the theorem, we give an alternative form of the generalised Vandermonde determinant, which we call quasi-triangular.

Finally, we shall give three sufficient conditions of the complete generalised Vandermonde determinant (hence also the Wronskian) to vanish (see Proposition~\ref{prop:sufficient-vandermonde-zero} on p.~\pageref{prop:sufficient-vandermonde-zero}); of these, one is non-trivial.

\medskip
Recall the definition of the (ordinary) Vandermonde determinant (over dimension $d=1$) and an old result \cite{ForKac} about the Wronskians of monomials.

\begin{theordef}[Vandermonde determinant]\label{thmdef:ordinary-vandermonde} 
    Take numbers $m_1,...,m_N$ from $\BBN_{\geqslant 0}$ (or $\BBZ$, or $\BBQ$); their Vandermonde determinant is
\begin{equation}\label{eq:vandermonde-ordinary}
    \det \Van_{d=1}^{0,1,...,N-1}(m_1,...,m_N)=\det \begin{vmatrix}
        1&\cdots&1\\
        m_1&\cdots&m_N\\
        \vdots&&\vdots\\
        m_1^{N-1}&\cdots&m_N^{N-1}
    \end{vmatrix} =
    \prod_{1\leqslant i<j\leqslant N} (m_j-m_i),
\end{equation}
    i.e.\ the determinant fully factorises to a product of linear factors in $m_1,...,m_N$.
\end{theordef}

\begin{theor}[Theorem 13 from \cite{ForKac}]\label{thm:mathRA} 
    Let $m_1,...,m_N\in \BBN_{\geqslant 0}$ (or $\BBZ$, or $\BBQ$) be constants and set $m=m_1+\cdots +m_N$; then the Wronskian determinant $W_{d=1}^{0,1,...,N-1}:=W_{d=1}^{k=N-1}$ of (generalised) monomial arguments becomes the (generalised) monomial
    \begin{equation}\label{eq:mathRA-wronskian-of-monoms}
        W_{d=1}^{0,1,...,N-1}(x^{m_1},...,x^{m_N})=\prod_{1\leqslant i<j\leqslant N} (m_j-m_i) x^{m-N(N-1)/2},
    \end{equation}
    where the coefficient is the \underline{\emph{Vandermonde}} determinant of their powers.
\end{theor}

\begin{rem}[generalisation of the Witt algebra, \cite{PRG25Wronsk}]\label{rem:oct25-witt}
    The (complex, infinite-dimensional) Witt algebra is given by differential operators with Laurent polynomial coefficients ordered by degree, $a_i=x^{i+1}$, and with the Lie bracket satisfying the commutation relations $[a_i,a_j]=(j-i)a_{i+j}$ for all $i,j\in \BBZ$. Its (strong) homotopy $L_\infty$ deformations by Wronskians increase the bracket arity; with the proper index shift $a_i=x^{i+N/2}$, by Theorem~\ref{thmdef:ordinary-vandermonde}, the $N$\nobreakdash-ary bracket $[\cdot,...,\cdot]_N$ given by the (ordinary) Wronskian determinant $W_{d=1}^{0,1,...,N-1}$, satisfies the commutation relations
    \begin{equation}\label{eq:commut-rels-Witt-deform-d=1}
        [a_{i_1},...,a_{i_N}]_N = \det \Van_{d=1}^{0,1,...,N-1}(i_1,...,i_N)\cdot a_{i_1+\cdots+i_N},
    \end{equation}
    where the structure constants, the Vandermonde determinants, are totally antisymmetric w.r.t\ their arguments.
\end{rem}

Before we give a natural multivariate generalisation of the Vandermonde determinants w.r.t\ the Wronskian determinant, let us have some motivating examples, which will illustrate all steps in the proof of Theorem~\ref{thm:general-vandermonde} below.

\begin{ex}[$k=1,d=2$]\label{ex:vander-k1d2}
    Let the dimension be $d=2$ and the differential order be $k=1$; set $N={d+k\choose k}=3$. 
    Consider $W_{k=1}^{d=2}=\mathbf{1}\wedge \partial/\partial x\wedge \partial/\partial y$ as the ternary ($N=3$) bracket $[\cdot,\cdot,\cdot]$. Let $\bx^{\bm_1}$, $\bx^{\bm_2}$, and $\bx^{\bm_3}$ be arbitrary monomials ($m_j^i\in \BBNo$ for all $i,j$) or generalised monomials ($m_j^i\in \BBZ$ or $\BBQ$). For the ease of notation let $\bmvec_1=(k_1,k_2)$, $\bmvec_2=(\ell_1,\ell_2)$, and $\bmvec_3=(n_1,n_2)$. Then their bracket becomes
    \begin{align*}
            [\bx^{\bm_1},\bx^{\bm_2},\bx^{\bm_3}] &= \det \begin{vmatrix}
                x^{k_1}y^{k_2}&x^{\ell_1}y^{\ell_2}&x^{n_1}y^{n_2} \\
                k_1 x^{k_1-1}y^{k_2}&\ell_1 x^{\ell_1-1}y^{\ell_2}&n_1 x^{n_1-1}y^{n_2} \\
                k_2 x^{k_1}y^{k_2-1}&\ell_2 x^{\ell_1}y^{\ell_2-1}&n_2 x^{n_1}y^{n_2-1}
            \end{vmatrix} \\
            &= x^{k_1-1}y^{k_2-1}\cdot x^{\ell_1-1}y^{\ell_2-1}\cdot x^{n_1-1}y^{n_2-1}\cdot \det \begin{vmatrix}
                xy& xy& xy\\
                k_1y & \ell_1 y & n_1 y \\
                k_2 x& \ell_2 x& n_2 x
            \end{vmatrix}, 
        \intertext{by gathering common powers per column,}
            &= x^{k_1+\ell_1+m_1-3}y^{k_2+\ell_2+n_2-3}\cdot xy\cdot y\cdot x\cdot \det \begin{vmatrix}
                1&1&1 \\ k_1&\ell_1&m_1 \\ k_2&\ell_2&n_2
            \end{vmatrix},
        \intertext{by gathering common monomials per row, leaving just the determinant of the number matrix}
            &= x^{k_1+\ell_1 + n_1 - 1}y^{k_2+\ell_2+n_2-1}\cdot \det \begin{vmatrix}
                1&1&1 \\ k_1&\ell_1&n_1 \\ k_2&\ell_2&n_2
            \end{vmatrix} \\
            &= \bx^{\bmvec_1+\bmvec_2+\bmvec_3 - \Bone}\cdot \det \begin{vmatrix}
                1&1&1 \\ k_1&\ell_1&n_1 \\ k_2&\ell_2&n_2
            \end{vmatrix},
    \end{align*}
    where $\Bone=(1,1)$ is the all-ones vector. We see that the result consists of a monomial in $x$ (resp. $y$) with their powers given by the sum of arguments' degrees in $x$ (resp. $y$) minus a shift, and the coefficient is a \emph{generalisation} of the Vandermonde determinant: instead of degrees' (of $x$) squares $(k_1)^2,(\ell_1)^2,(n_1)^2$ in the 3rd row, it contains the powers of $y$, namely, $k_2,\ell_2,n_2$.
\end{ex}

\begin{notation}
    In the above Example~\ref{ex:vander-k1d2}, we have the coordinates $x^1=x$ and $x^2=y$, of which we take powers to form monomials like $(x^1)^{m_j^1}\cdot (x^2)^{m_j^2}$. By convention (see Notation~\ref{not:multi-index} on p.~\pageref{not:multi-index}), $(x^1)^{m^1}\cdot (x^2)^{m^2}=\bx^{\bmvec}$ over the base dimension $d$ (here $d=2$) and $\bmvec\in\BBNo^d$ (or $\BBZ^d$, or $\BBQ^d$).

    For the (complete) generalised Wronskian $W_{d\geqslant 1}^{k\geqslant 1}$ of arity $N={d+k\choose k}$ we now take the monomials $(x^1)^{m_j^1}\cdots (x^d)^{m_j^d}=\bx^{\bmvec_j}$ with $j\in \{1,...,N\}$ as its arguments. For every $j$ and its respective power $d$-tuple $\bmvec_j=(m_j^1,...,m_j^d)$, we put by definition $\bm_j=\bmvec_j$.

    Next, consider the $n$th row in the Wronskian: in it, each argument is differentiated by the differential operator $\partial^{|\brvec_n|}/\partial \bx^{\brvec_n}$ (see Notation~\ref{not:jth-row-wronskian} on p.~\pageref{not:jth-row-wronskian}); by Definition~\ref{def:wronskian} on p.~\pageref{def:wronskian}, to this operator there corresponds a monomial $\bx^{\brvec_n}$, which we take as the (evaluation) function $\BBNo^d\to\BBNo$ (or with $\BBNo$ replaced by $\BBZ$ or $\BBQ$) given by the formula $(m^1,...,m^d)\mapsto \bm^{\brvec_j}=(m^1)^{r_j^1}\cdots (m^d)^{r_j^d}$. For instance, $\partial^3/\partial x^2\partial y\leftrightarrows x^2y\leftrightarrows \{(m^1,m^2)\mapsto (m^1)^2\, m^2\}$. In our notation, $\bm_j^{\brvec_n}$ means raising the powers of base variables $m_j^\alpha$ to the respective degrees $r_n^\alpha$ by which the arguments in the Wronskian are differentiated in the $n$th row. Here $\brvec_1=(0,...,0)$ gives $\bm_j^{\brvec_1}=1$, the unit.
\end{notation}

\begin{define}[generalised Vandermonde determinant]\label{def:gen-vandermonde}
    Let the dimension be $d\geqslant 1$ and the total degree be $k\geqslant 1$; set $N={d+k\choose k}$. Let $\partial^{|\brvec_n|}/\partial \bx^{\brvec_j}$ be the differential operator from the $n$th row in the Wronskian $W_d^k=\Bone\wedge\partial_\bx\wedges\partial^k_{\bx...\bx}$; here $\brvec_n$ is its corresponding power $d$-tuple (multi-index). The \emph{complete generalised Vandermonde determinant} of degree $k$ over dimension $d$ w.r.t\ the complete generalised Wronskian $W_d^k$ of power $d$-tuples $\bm_1,...,\bm_N\in \BBQ^d$ is the determinant
    \begin{equation}\label{eq:generalised-vandermonde-def}
        \det \Van_d^k(\bm_1,...,\bm_N)=\det \begin{pmatrix}
            \bm_1^{\brvec_1}&\cdots&\bm_N^{\brvec_1} \\
            \vdots &&\vdots \\
            \bm_1^{\brvec_N}&\cdots&\bm_N^{\brvec_N}
        \end{pmatrix},
    \end{equation}
    where in our notation $\bm_j^{\brvec_n}=(m_j^1)^{r_n^1}\cdots(m_j^d)^{r_n^d}$ and the first row consists of only units: $\bm_j^{\brvec_1}=1$. Let us remember that the formula of the complete generalised Vandermonde determinant is relative to the complete generalised Wronskian determinant.
\end{define}

\begin{ex}
    With the notation as above,
    \begin{equation}
        \det\Van_{d=1}^{k=N-1}(m_1,...,m_N)=\det\begin{pmatrix}
            1&\cdots&1\\m_1&\cdots&m_{N}\\(m_1)^2&\cdots&(m_N)^2\\ \vdots&&\vdots\\ (m_1)^{N-1}&\cdots &(m_N)^{N-1}
        \end{pmatrix},
    \end{equation}
    namely, in dimension $d=1$ we recover the ordinary Vandermonde determinant. Now for higher dimensions:
    \begin{equation}
        \det\Van_{d=2}^{k=1}(\ba,\bb,\bc)=\det\begin{pmatrix}
            1&1&1\\ a_1&b_1&c_1\\ a_2&b_2&c_2
        \end{pmatrix} = 
        (b_1c_2-b_2c_1) - (a_1c_2 - a_2c_2) + (a_1b_2 - a_2b_1)
    \end{equation}
    as before in Example~\ref{ex:vander-k1d2} with $\ba=\bm_1$, etc.; finally,
    \begin{equation}\label{eq:six-arguments-vandremonde}
        \det \Van_{d=2}^{k=2}(\ba,\bb,\bc,\bd,\be,\bbf) = \det \begin{pmatrix}
            1&1&1&1&1&1\\
            a_1&b_1&c_1&d_1&e_1&f_1\\
            a_2&b_2&c_2&d_2&e_2&f_2\\
            a_1^2&b_1^2&c_1^2&d_1^2&e_1^2&f_1^2\\
            a_1a_2&b_1b_2&c_1c_2&d_1d_2&e_1e_2&f_1f_2\\
            a_2^2&b_2^2&c_2^2&d_2^2&e_2^2&f_2^2,
        \end{pmatrix}
    \end{equation}
    which is the determinant of a $6\times 6$ matrix as $N={2+2\choose 2}=6$. As a side remark, we note that (see \cite[Example~2.8]{FBrownVandermonde2024}) the determinant in~\eqref{eq:six-arguments-vandremonde} vanishes if and only if the six points in the Euclidean plane $\ba=(a_1,a_2),\bb=(b_1,b_2),...,\bbf=(f_1,f_2)\in \BBR^2$ lie on a conic $Ax^2+Bxy+Cy^2+Dx+Ey+F$.
\end{ex}

\begin{rem}\label{rem:factorisation-of-vandermonde-earlyremark}
    Unlike the full factorisation of the ordinary Vandermonde determinant, \eqref{eq:vandermonde-ordinary} on p.~\pageref{eq:vandermonde-ordinary}, the generalised Vandermonde determinant in \eqref{eq:generalised-vandermonde-def} over dimension $d\geqslant 2$ does not immediately factorise into a product of linear terms. Yet for special choices of arguments we approach the full factorisation of the generalised Vandermonde determinant in \S\ref{SecConstruct} (see Theorem~\ref{thm:structure-loneliness} on p.~\pageref{thm:structure-loneliness} and Theorem~\ref{thm:golden-formula} on p.~\pageref{thm:golden-formula}, where we calculate the structure constants for particular classes of polynomial algebras with complete generalised Wronskian determinants as $N$\nobreakdash-ary brackets). See Problem~\ref{open:factorisations-vandermonde} and the accompanying Comment on p.~\pageref{open:factorisations-vandermonde} for further comment.
\end{rem}

\begin{prop}[translation invariance]\label{prop:translation-invariance}
    For any (power) $d$-tuples $\bm_1,...,\bm_N\in \BBQ^d$ and any (homogeneous shift) $d$-tuple $\bs\in\BBQ$ we have
    \begin{equation}
        \det\Van_d^k(\bm_1+\bs,...,\bm_N+\bs) = \det\Van_d^k(\bm_1,...,\bm_N),
    \end{equation}
    i.e.\ the complete generalised Vandermonde determinant over base dimension $d\geqslant1$ and degree $k\geqslant 1$ remains, even when $d>1$, translation invariant in the course of an homogeneous translation of all its arguments.
\end{prop}
\begin{proof}[Proof sketch]
    The proof works by subtracting previous lines of the determinant, reappearing later with a common factor depending only on the shift $\bs$.
    Indeed, expand any $(\bm_n+\bs)^{\brvec_j}=\bm_n^{\brvec_j}+(\cdots)$. For differential order 1 we cancel the garbage $(\cdots)$ by adding an $s^i$-multiple of the first row $(1,...,1)$. We cancel the garbage in all subsequent rows by adding $s^i$-multiplies of the previous rows.
\end{proof}

We work over $\Bbbk[x^1,...,x^d]$ with $N$\nobreakdash-ary bracket, $N={d+k\choose d}$, given by the complete generalised Wronskian determinant $W_d^k$ of differential order $k$ over dimension $d$. Note that the rows of any differential order $0\leqslant r\leqslant k$ form the complete set of monomials of degree $r$. For instance, over dimension $d=2$ and differential order $k=2$ in the complete Wronskian $W_{d=2}^{k=2}=\Bone\wedge\partial_x\wedge\partial_y\wedge\partial^2_{xx}\wedge\partial^2_{xy}\wedge\partial^2_{yy}$, let us take degree $r=2$: we obtain the divided power monomials $x^2/2!$, $xy/1!1!$, and $y^2/2!$. Note that these are all the basic monomials of degree 2 over dimension 2; the same is true for each degree $r$ from the range $r\in\{0,1,...,k\}$. Hence the following lemma.

\begin{lem}\label{lem:sum-of-diff-orders-indep-of-i}
    In the complete generalised Wronskian $W_{d\geqslant 2}^{k\geqslant 1}=\bigwedge_{j=1}^N \partial^{|\brvec_j|}/\partial \bx^{\brvec_j}$ over base dimension $d\geqslant 2$ and differential order $k\geqslant 1$, the sum of differential operators' orders in a coordinate $x^i$, $i\in\{1,...,d\}$, \emph{does not depend on $i$} (and equals $\sum_{j=1}^N r_j^i$ by construction).
\end{lem}

By the differential operator and monomial correspondence (basically, $\partial^3/\partial x^2\partial y \leftrightarrows x^2y$), the sum of degrees $r_j^i$ is the same in the Vandermonde determinant relative to the Wronskian as above.

We now calculate this sum of degrees $\sum_{j=1}^N r_j^i$ in a coordinate $x^i$; it is the degrees (but not coordinates) that show up in the Vandermonde determinant. By the symmetry $\sum_{j=1}^N r_j^i = \sum_{j=1}^N r_j^{i'}$ of the \emph{complete} generalised Wronskian $W_d^k$ w.r.t\ base coordinates in it (here $i,i'\in \{1,...,d\}$), the sum equals $\sum_{j=1}^N r_j^i = (1/d)\sum_{j=1}^N r_j$, where $r_j$ is the total differential order of the operator in the $j$th row in the Wronskian $W_d^k$ (see Definition~\ref{def:wronskian} and Notation~\ref{not:jth-row-wronskian} on p.~\pageref{not:jth-row-wronskian}). By indexing the blocks of rows using their common total differential order $r$ (e.g., $r=2$ in dimension $d=2$ gives three differentials $\partial^2_{xx},\partial^2_{xy}$, and $\partial^2_{yy}$), we find that there are ${d+r-1\choose r}$ monomials of degree $r$ over dimension $d$, hence equally many corresponding differential operators. Let us simplify the sum of binomial expressions within $r_j$.

\begin{lem}[combinatorial]\label{lem:combinatorial}
    For any integers $k,d\geq 1$ we have
    \begin{equation}\label{eq:total-degree-sum}
        \sum_{j=1}^N r_j = \sum_{r=0}^k r {d+r-1\choose r} = d{d+k\choose k-1} = \frac{kd}{d+1} N\in \BBN,
    \end{equation}
    where $N={d+k\choose d} = {d+k\choose k}$. Therefore for any $i\in\{1,...,d\}$ we have
    \begin{equation}\label{eq:degree-sum-in-coordinate-i}
        \sum_{j=1}^N r_j^i = \frac{1}{d} \sum_{j=1}^N r_j = \frac{kN}{d+1}\in \BBN.
    \end{equation}
    (The \emph{degree shift} in formula \eqref{eq:degree-sum-in-coordinate-i} and the \emph{total degree shift} in formula \eqref{eq:total-degree-sum} will reappear in Eq.~\eqref{eq:wronskian-of-monoms-vandermonde} and many times onwards.)
\end{lem}
\begin{proof}
    First we rewrite, for the common degree $r\geq 1$:
    \begin{equation*}
        r{d+r-1\choose r} = r\frac{(d+r-1)!}{r!\,(d-1)!} = \frac{(d+r-1)!}{(r-1)!\,(d-1)!}\frac{d}{d} = d\frac{(d+r-1)!}{(r-1)!\,d!} = d{d+r-1\choose d}.
    \end{equation*}
    Use the last equality to take $d$ out of the sum,
    \begin{equation*}
        \begin{split}
            \sum_{r=0}^k r{d+r-1\choose r} &= 0+\sum_{r=1}^k r{d+r-1\choose r} = d\sum_{r=1}^k {d+r-1\choose d} = d {d+k\choose d+1},
        \end{split}
    \end{equation*}
    by the `hockey-stick lemma' in the last equality. The result simplifies to $d{d+k\choose k-1}$, which we rewrite by the inductive property of combinations ${n\choose k}={n-1\choose k-1}+{n-1\choose k}$ as follows, setting $n:=d+k+1$,
    \begin{equation*}
        {d+k\choose k-1} = {d+k+1\choose k} - {d+k\choose k} = \frac{d+k+1}{d+1} \frac{(d+k)!}{k!\,d!}-N = \frac{k}{d+1}N.
    \end{equation*}
    Combining all results above yields $\sum_{r=0}^k r{d+r-1\choose r} = \frac{kd}{d+1}N$ as desired.
\end{proof}

We are now ready to generalise Theorem~\ref{thm:mathRA} about the degree shift in the Wronskian determinant of monomials, now over dimension $d>1$, and whereby the Vandermonde determinant of the arguments' powers ($d$-tuples) shows up as the coefficient; see formula \eqref{eq:mathRA-wronskian-of-monoms} on p.~\pageref{eq:mathRA-wronskian-of-monoms}. Let us set up the notation. Let $\bkvec_1,...,\bkvec_N\in \BBQ^d$ be the power $d$-tuples; we put $\bx^{\bkvec_j}=(x^1)^{k_j^1}\cdots(x^d)^{k_j^d}$ as in Notation~\ref{not:multi-index}. We thus obtain $N$ arbitrary (generalised) monomials over base dimension $d$; whenever all $k_j^i\in \BBNo$, they are genuine monomials from $\Bbbk[x^1,...,x^d]$. For a given number $s\in \BBQ$ (or in $\BBNo$), the shift of a power multi-index $\bkvec$ occurs component-wise: let $\Bone=(1,...,1)$ so that $\bkvec+s\Bone:=(k^1+s,...,k^d+s)$. Recall that the complete generalised Wronskian determinant is given, equivalently, by the two formulas: $W_d^k=\Bone\wedge\partial_\bx\wedges\partial^k_{\bx...\bx}=\bigwedge_{j=1}^N \partial^{|\brvec_j|}/\partial \bx^{\brvec_j}$. 


\begin{theor}[generalised Vandermonde]\label{thm:general-vandermonde}
    Fix the base dimension $d\geqslant 1$ and degree $k\geqslant 1$; set $N={d+k\choose k}$. Let $\bx^{\bkvec_1},...,\bx^{\bkvec_N}$ be $N$ arbitrary (generalised) monomials (their powers can be rational numbers). Set $\bkvec:=\sum_{j=1}^N \bkvec_j$ with the component-wise summation: $k^i=\sum_{j=1}^Nk_j^i$. Then the complete generalised Wronskian determinant of these monomials in $d$ coordinates again is a monomial,
    \begin{equation}\label{eq:wronskian-of-monoms-vandermonde}
        W_d^k(\bx^{\bkvec_1},...,\bx^{\bkvec_N}) = \det \Van_d^k(\bk_1,...,\bk_N) \cdot \bx^{\bkvec - kN/(d+1)\Bone},
    \end{equation}
    and
    \begin{equation}\label{eq:det-triangular-equals-det-vandermonde}
        \det \Van_d^k(\bk_1,...,\bk_N) = \det \Delta_d^k(\bk_1,...,\bk_N),    
    \end{equation}
    where the matrix $\Delta_d^k(\bk_1,...,\bk_N)$ is given by
    \begin{equation}\label{eq:triangular-vandermonde-matrix-element}
        \Delta_d^k(\bk_1,...,\bk_N)^{\mathrm{row}~j}_{\mathrm{col}~n} = \prod_{i=1}^d \left\{ \prod_{\ell=0}^{r_j^i -1}(k_n^i-\ell) \right\}.
    \end{equation}
\end{theor}

\begin{define}\label{def:quasi-triangular-form}
    The matrix $\Delta_d^k$ is the \emph{quasi-triangular} form of the Vandermonde matrix.
\end{define}

\begin{note}\label{note:choice-of-term-quasi-triangular}
    The choice of the term, `quasi-triangular', is motivated by the fact that whenever $\bk_j=\br_j$ for all the lines of the generalised Wronskian determinant (so $j\in \{1,...,N\}$), the matrix $\Delta_d^k$ is triangular. Let us emphasise that some other choice of the $N$ monomials as arguments for $W_d^k$ \emph{may result in a non-triangular matrix} $\Delta_d^k$; still, its determinant is equal to the Vandermonde determinant of the power $d$-tuples $\bk_j$ of the monomial arguments.
\end{note}

\begin{rem}
    Suppose all $\bx^{\bkvec_j}\in \Bbbk[x]$ are actual monomials (i.e.\ $k_j^i>0$ for all $i,j$). If, after the degree shift, for some coordinate $x^i$ we have that the resulting degree is negative, $\sum_{j=1}^N k_j^i-kN/(d+1)<0$, then the Vandermonde determinant must vanish identically. (Indeed, any derivatives and products of monomials are monomials; it is impossible to produce a Laurent monomial.) See also Eq. \eqref{eq:differential-operator-acting-on-monomial} on p.~\pageref{eq:differential-operator-acting-on-monomial} and the surrounding discussion.
    
    In Proposition~\ref{prop:sufficient-vandermonde-zero} on p.~\pageref{prop:sufficient-vandermonde-zero} we list other cases when the Vandermonde determinant in front of the monomial equals zero.
\end{rem}

\begin{rem}
    Theorem~\ref{thm:general-vandermonde} reduces to the old Theorem~\ref{thm:mathRA} on p.~\pageref{thm:mathRA} when $d=1$. Indeed, for $d=1$ we have $N=k+1$, hence $kN/(d+1)=(N-1)\cdot N/2=N(N-1)/2$.
\end{rem}

\begin{proof}[Proof (of Theorem~\ref{thm:general-vandermonde})] 
    The proof follows the same approach as in Example~\ref{ex:vander-k1d2}. For any monomial $a(\bx)=c\bx^\brvec$, $c$ is the coefficient and $\bx^\brvec$ the term; we solve for both in Eq. \eqref{eq:wronskian-of-monoms-vandermonde}.

    The differential operator $\partial^{|\brvec_j|}/\partial\bx^{\brvec_j}$ in the $j$th row of the Wronskian acting on the monomial $\bx^{\bkvec_n}$ gives the term $\bx^{\bkvec_n-\brvec_j}$ and the coefficient in \eqref{eq:triangular-vandermonde-matrix-element}. As $k$ is the highest total differential order of the Wronskian, $|\brvec_j|\leqslant k$, factor $\bx^{\bkvec_n-k\Bone}$ from the $n$th colmn of the determinant, which leaves $\bx^{k\Bone-\brvec_j}$, the same term in the $j$th row; factor it, leaving only numbers in the entries of the determinant. The overall term is thus $\bx^\bnvec$, where the power $d$-tuple $\bnvec$ is given by
    \begin{equation}
        \bnvec =\sum_{n=1}^N (\bkvec_n-k\Bone) + \sum_{j=1}^N(k\Bone - \brvec_j) = \sum_{j=1}^N \bkvec_j - \frac{kN}{d+1}\Bone = \bkvec - \frac{kN}{d+1}\Bone,
    \end{equation}
    proving the statement in formula \eqref{eq:wronskian-of-monoms-vandermonde}.

    Expand each row in \eqref{eq:triangular-vandermonde-matrix-element}; the term of the highest degree in row $j$ and column $n$ is then $\bk_n^{\brvec_j}$, which is exactly the element in the Vandermonde determinant. As all lower-degree terms share an homogeneous coefficient in all rows, subtracting the preceding rows leaves only the top-degree term, proving that the two determinants are equal -- \eqref{eq:det-triangular-equals-det-vandermonde}. The proof is complete.
\end{proof}

\begin{ex}
    It is often much easier to compute the complete generalised Vandermonde determinant by using its quasi-triangular form. For instance, over dimension $d=2$ with degree $k=2$, and all standard monomials up to degree $k=2$ as arguments (namely, $1,x,y,x^2,xy,y^2$ corresponding to the power $d$-tuples $\brvec_1=(0,0)$, followed by $(1,0),...,(1,1)$ until $\brvec_6=(0,2)$, respectively), we have
    \begin{align}
        \det \Van_{d=2}^{k=2}(\brvec_1,...,\brvec_6) &= \det \begin{pmatrix}
            1&1&1&1&1&1\\
            0&1&0&2&1&0\\
            0&0&1&0&1&2\\
            0&1&0&4&1&0\\
            0&0&0&0&1&0\\
            0&0&1&0&1&4
        \end{pmatrix}=4,\\
    \intertext{whereas the determinant of the quasi-triangular matrix $\Delta_d^k(\br_1,...,\br_6)$ is}
        \det \Delta_{d=2}^{k=2}(\brvec_1,...,\brvec_6) &= \det \begin{pmatrix}
            1&1&1&1&1&1\\
            &1&0&2&1&0\\
            &&1&0&1&2\\
            &&&2&0&0\\
            &\bigzero&&&1&0\\
            &&&&&2
        \end{pmatrix}=4.
    \end{align}
    Indeed, in this example the determinant equals the product of entries on the main diagonal.
\end{ex}

By examining the mechanism by which the complete generalised Wronskian shifts the degree of each co\"ordinate (see Theorem~\ref{thm:general-vandermonde}), we give the following definitions.

\begin{define}[shifts and degrees]\label{def:wronsk-shift-and-monomial-degree}
    Let $a(\bx)\in \Bbbk[x^1,...,x^d]$ be a monomial, $a(\bx)=\bx^\bnvec=(x^1)^{n_1}\cdots (x^d)^{n_d}$. The \emph{total degree} of $a(\bx)$ is $\deg(a)=n_1+\cdots+n_d$; the \emph{degree of $a(\bx)$ in co\"ordinate $x^i$} is $\deg_i(a)=n_i$. The definition of total degree extends to an homogeneous polynomial, that is, to the sum of monomials if they all have the same total degree.

    The \emph{total degree shift} produced by the \emph{complete}\footnote{
        Breaking the completeness assumption voids this definition. Indeed, the degree shift will be different in each coordinate; the total degree shift will be reduced by the total differential orders of whatever differential operators are excluded from the Wronskian (cf.\ \cite{AVK25YerQTS13WrosnkBJP} for the definition and discussion of incomplete Wronskians).
        }
    generalised Wronskian $W_{d\geqslant 1}^{k\geqslant 1}$ of differential order $k$ over dimension $d$ is the integer $kd\cdot N/(d+1)$ and the \emph{degree shift in coordinate $x^i$} is the integer\footnote{Cf.\ Lemma~\ref{lem:combinatorial}: the shift $kN/(d+1)$ can be expressed as a binomial coefficient.} $kN/(d+1)$.
\end{define}

\begin{cor}[degree shift]\label{corr:degree-shift}
    With the same notation as in Theorem~\ref{thm:general-vandermonde}, the complete generalised Wronskian $W_{d\geqslant 1}^{k\geqslant 1}$ acts on monomials $\bx^{\bkvec_j}$ (with all degrees $k^i_j\in \BBNo$) in such a way that the degrees of monomials satisfy the two relations:
    \begin{equation}\label{eq:total-deg-bound}
        \deg W_d^k(\bx^{\bkvec_1},...,\bx^{\bkvec_N}) = 
        \begin{cases}
            \sum_{j=1}^N \bkvec_j - \frac{kd}{d+1}N & \text{if}\quad \det \Van_d^k(\bk_1,...,\bk_N) \neq 0, \\
            0 &\text{otherwise},
        \end{cases}
    \end{equation}
    and
    \begin{equation}\label{eq:deg-i-bound}
        \deg_i W_d^k(\bx^{\bkvec_1},...,\bx^{\bkvec_N}) = 
        \begin{cases}
            \sum_{j=1}^N \bkvec_j - \frac{kN}{d+1} & \text{if}\quad \det \Van_d^k(\bk_1,...,\bk_N) \neq 0, \\
            0 &\text{otherwise}.
        \end{cases}
    \end{equation}
    That is, the degree of the resulting monomial is shifted down by the degree shift.
\end{cor}

\begin{rem}\label{rem:degree-shift-equals-dividedpower}
    The degree shift $kN/(d+1)$ in a coordinate $x^i$, produced by the complete generalised Wronskian $W_{d\geqslant1}^{k\geqslant 1}$, is exactly equal to the sum of differential orders w.r.t\ $x^i$ (namely, $\sum_{j=1}^N r_j^i$) of the operators from the rows ($j\in \{1,...,N\}$) in the Wronskian determinant. Equivalently, this degree shift equals the sum of degrees in $x^i$ from the standard monomials (i.e.\ monomial $\bx^{\brvec_j}$ corresponding to the operator $\partial^{|\brvec_j|}/\partial \bx^{\brvec_j}$ in the $j$th row of the Wronskian). Note that, by construction, $\brvec_1=(0,...,0)$, $\brvec_2=(1,0,...,0)$, and so on until $\brvec_N=(0,...,0,k)$. Now it is readily seen why $W_d^k(\bx^{\brvec_1}/\brvec_1!,...,\bx^{\brvec_N}/\brvec_N!)=1$, a degree-zero monomial in $d$ variables.
\end{rem}

\begin{prop}\label{prop:sufficient-vandermonde-zero}
    Each of the following conditions alone is  sufficient for the complete generalised Vandermonde determinant of $d$-tuples $\bm_1,...,\bm_N\in \BBQ^d$ to be zero:
    \begin{enumerate}[label=({\it\roman*})]
        \item\label{prop8-i:duplicate columns} duplicate columns: $\bm_i=\bm_{i'}$ for $i'\neq i$;
        \item\label{prop8-ii:coinciding-degrees} all degrees in a coordinate coincide: $m_1^i=m_2^i=\cdots=m_N^i$ for some $i\in\{1,...,d\}$;
        \item\label{prop8-iii:prop-sufficient-vandermonde:deficiency} (\emph{`deficient degree'}): if the power $d$-tuples of monomials $\bm_1,...,\bm_N\in \BBN_{\geqslant 0}^d$ are such that their sum is less than the degree shift: $\sum_{j=1}^n m_j^i<kN/(d+1)$ for some $i$.
    \end{enumerate}
    In all cases the determinant vanishing is equivalent to the columns being linearly dependent.
\end{prop}
\begin{proof}
    (\textit{i}) Is immediate. (\textit{ii}) Follows by subtracting $m_1^i$ times the row $(1,...,1)$ of units from the $(i+1)$st row in the complete generalised Vandermonde determinant, obtaining an all-zeroes row. (\textit{iii}) Follows from Theorem~\ref{thm:general-vandermonde} as we have a deficiency of degree, which would produce a non-monomial -- impossible.
\end{proof}

\begin{rem}[multivariate generalisation of the Witt algebra]\label{rem:multivar-gen-of-witt-algebra}
    In Remark~\ref{rem:oct25-witt} on p.~\pageref{rem:oct25-witt} (see also Theorem~\ref{thm:mathRA}) we found the strong homotopy deformations of the Witt algebra in the one-dimensional case when $N$ is even; this was achieved by applying a degree shift to the complex Laurent monomials (with possibly half-integer degrees). Upon such a shift the canonical commutation relations were obtained: $[a_{i_1},...,a_{i_N}]_N=\Omega(i_1,...,i_N)\cdot a_{i_1+\cdots+i_N}$ with $\Omega = \det \Van_{d=1}^{0,1,...,N-1}$ totally antisymmetric w.r.t\ its arguments. 
    
    We now continue with the deformation approach, generalising the Witt algebra over dimension $d>1$. Take some generalised monomials $\bx^{\bkvec_j}=(x^1)^{k_j^1}\cdots (x^d)^{k_j^d}$ with the rational power $d$-tuples $k_j^i\in \BBQ$ and coordinates $\bx=(x^1,...,x^d)$ over base dimension $d$; by means of Theorem~\ref{thm:general-vandermonde}, calculate the $N$\nobreakdash-ary bracket $[\cdot,...,\cdot]_N$, $N={d+k\choose d}$, given by the complete generalised Wronskian determinant $W_d^k$ of generalised monomial arguments. The result is a (generalised) monomial; its degree is made of up the component-wise sum of arguments' degrees minus the homogeneous shift $kN/(d+1)$. Apply the shift $s\in \BBQ$ to all arguments: $a_{\bi_n}:=\bx^{\bivec_n+s\Bone}$, where $\bi_n\in \BBQ^d$ is the power $d$-tuple and $\Bone=(1,...,1)$. To recover the canonical commutation relations the shift must satisfy
    \begin{align}\label{eq:index-shift}
        Ns - \frac{kN}{d+1} &= s, & \text{hence} && s &= \frac{k}{d+1}\cdot \frac{N}{N-1},
    \end{align}
    which in general is not an integer.\footnote{
    We remark that this shift is the same as the homogenisation constant $y^*$ for homogenising the recurrence $y_n=y_{n-1}+y_{n-2}+\cdots+y_{n-N}-kN/(d+1)$. Namely, with $x_n=y_n-y^*$ we have $x_n=x_{n-1}+\cdots+x_{n-N}$.
    } This gives the commutation relations
    \begin{equation}\label{eq:witt-d-commutation-relations}
        [a_{\bi_1},...,a_{\bi_N}]_N=\Omega(\bi_1,...,\bi_N)\,a_{\bi_1+\cdots+\bi_N},
    \end{equation}
    where the structure constant $\Omega=\det\Van_d^k$ is totally antisymmetric w.r.t\ its arguments and the index sum is taken component-wise: $i_1^i+\cdots+i_N^i$ with $i\in\{1,...,d\}$.
\end{rem}

\section{Explicit construction of a class of polynomial algebras $\cA\subseteq\Bbbk[x^1,\ldots,x^d]$ with Wronskian brackets}\label{SecConstruct}
\noindent%
In this section we find all the finite-dimensional $N$\nobreakdash-ary strongly homotopy (SH) Lie polynomial subalgebras $\EuA\subseteq \Bbbk[\bx]$ containing the constant unit polynomial $1$ both in the algebra, $1\in \EuA$, and in the span of the image of the bracket, $1 \in \Span_\Bbbk W_d^k(\EuA,...,\EuA)$. 
We are interested in this condition because if $1\in \EuA$ is not in the (span of the) image of the bracket, the algebra $\EuA$ is not perfect (see Definition~\ref{def:perfect-lie-algebra} on p.~\pageref{def:perfect-lie-algebra}), hence it cannot be semi-simple; we shall give an equivalent condition for $1\in \Span_\Bbbk W_d^k(\EuA,...,\EuA)$.

We shall describe (a class of) all finite-dimensional algebras $\EuA$; we show that there are no others subject to the restrictions above. Among these finite-dimensional algebras, we characterise which are or cannot be perfect for each base dimension and differential order $k$; we give remedying examples of perfect algebras. 
Finally, we show that, subject to the assumptions above, no other finite-dimensional algebras exist, and we comment on why, even with the assumptions dropped, probably there exist no other (non-trivial, or perfect) finite-dimensional algebras.

\medskip
Recall (cf.\ Remark~\ref{rem:degree-shift-equals-dividedpower} on p.~\pageref{rem:degree-shift-equals-dividedpower}) that $W_d^k(\bx^{\brvec_1}/\brvec_1!,...,\bx^{\brvec_N}/\brvec_N!)=1$, where $\brvec_j$ is the power $d$-tuple of the differential operator $\partial^{\brvec_j}/\partial \bx^{\brvec_j}$ in the $j$th row of the complete generalised Wronskian determinant $W_d^k=\Bone\wedge \partial_\bx\wedge\cdots\wedge \partial^k_{\bx\hdots\bx}=\bigwedge_{j=1}^N \partial^{\brvec_j}/\partial \bx^{\brvec_j}$. Replacing any argument by a higher total-degree monomial makes the degree sum strictly greater than $kN/(d+1)$, which is the shift downwards by the Wronskian determinant; hence $1$ cannot be obtained on the new set of arguments. Let us formalise this claim and prove it (see Appendix~\ref{SecAppConsistency} on p.~\pageref{SecAppConsistency}).

\begin{lem}[consistency]\label{lem:consistency}
    Fix any dimension $d\geqslant 1$ and any differential order $k\geqslant 1$; set $N={d+k\choose k}$. Let $a_1(\bx),...,a_N(\bx)\in \Bbbk[\bx]$ be arbitrary \emph{monomials}. Then
    \begin{align}
        W_d^k(a_1(\bx),...,a_N(\bx))&=1 && \text{implies} & a_j(\bx)&= \bx^{\brvec_j}/\brvec_j! \quad\forall j\leqslant N
    \end{align}
    up to (positive-sign) reordering and normalisation constants of $a_j$'s such that their product equals $+1$.
\end{lem}

\medskip
We therefore immediately see that the requirement $1\in \Span_\Bbbk W_d^k(\EuA,...,\EuA)\subseteq \EuA$ is equivalent to $\EuA$ containing all polynomials up to degree $k$, i.e.\ $\Bbbk_k[\bx]\subseteq \EuA\subseteq \Bbbk[\bx]$. The case $\EuA=\Bbbk_k[\bx]$ is trivially closed as the resultant always lies in $\Bbbk$. Hence the following definition.

\begin{define}[Types of polynomial $N$\nobreakdash-ary SH-Lie algebras]\label{def:consistent-lonely-chubby-tall-lanky}
    Fix the dimension $d\geqslant 1$ and differential order $k\geqslant 1$; set $N={d+k\choose k}$ and consider the complete generalised Wronskian determinant $W_d^k$. An SH-Lie subalgebra of polynomials $\EuA\subseteq \Bbbk[x^1,...,x^d]$ with the $N$\nobreakdash-ary bracket given by $W_d^k$ is called \emph{trivial} if $\EuA\subseteq \Bbbk_k[\bx]$, and \emph{consistent} (with respect to $W_d^k$)  if it contains $\Bbbk_k[\bx]$ as a proper subset, i.e.\ $\Bbbk_k[\bx]\subsetneq \EuA\subseteq \Bbbk[\bx]$.

    In what follows we will usually take $\Bbbk_k[\bx]\subsetneq \EuA\subseteq \Bbbk_{k+1}[\bx]$. If $\EuA$ from $\Bbbk_{k+1}[\bx]$ contains (the span of) exactly one $(k+1)$th degree monomial or homogeneous polynomial, i.e.\ $\EuA=\Bbbk_k[\bx]\oplus \langle p\rangle$, we say that $\EuA$ is \emph{lonely}.\footnote{This class, $\EuA=\Bbbk_k[\bx]\oplus\langle p(\bx)\rangle$, will form the class of finite-dimensional algebras (see Theorem~\ref{thm:main-1-existence} on p.~\pageref{thm:main-1-existence}); it is this what permits us to write the equality sign $=$ in the definition: `Lonely is finite-dimensional.'} On the other hand, if $\EuA$ contains at least two linearly independent $(k+1)$th degree monomials or homogeneous polynomials $p\not\sim q\in \Bbbk_{k+1}[\bx]$, that is, $\EuA\supseteq \Bbbk_k[\bx]\oplus \langle p,q\rangle$, we say that the algebra $\EuA$ is \emph{chubby}.\footnote{We will show in Theorem~\ref{thm:main-2:non-existence} on p.~\pageref{thm:main-2:non-existence} that none such algebra is finite-dimensional: `Chubby is infinite-dimensional.'} 
    
    We say that the algebra $\EuA$ is \emph{tall} if it contains any polynomial of degree $>k+1$.
    If a tall algebra $\EuA$ contains only one monomial of degree $k+1$, then we say that $\EuA$ is \emph{lanky}; we will soon see that in this case, $\EuA=\Bbbk_k[\bx]\oplus\langle(x^i)^{k+1},...,(x^i)^{k+\ell}\rangle$ for some fixed $i\in\{1,...,d\}$ and $\ell\geqslant 2$ (see Lemma~\ref{lem:chubby-tall-lanky} below). If, on the contrary, a tall algebra contains multiple linearly independent monomials (or homogeneous polynomials) of degree $k+1$, then the tall algebra $\EuA$ is also chubby. 
\end{define}

Using the fact that the algebra $\EuA$ is closed under the Wronskian determinant, we will produce the $(k+1)$th degree monomials as resulting from the Wronskian determinant by using arguments from $\EuA$. By multilinearity and the degree-shift mechanism in Theorem~\ref{def:gen-vandermonde} on p.~\pageref{thm:general-vandermonde}, it suffices to consider $\EuA$ spanned by monomials. The following lemma explains our choice of terminology in Definition~\ref{def:consistent-lonely-chubby-tall-lanky} (see also the visualisation in Figure~\ref{fig:visaliation-of-algebras}).

\begin{lem}\label{lem:chubby-tall-lanky}
    Suppose $\EuA$ is a tall polynomial SH-Lie algebra, i.e.\ $\EuA$ contains a monomial\footnote{More generally, for a polynomial $p(\bx)$ in place of $a(\bx)$, extend the results of the lemma by decomposing $p(\bx)$ into its constituent monomials and applying the lemma to each monomial.}  $a(\bx)$ of degree $\deg(a)>k+1$, and suppose $\EuA$ is consistent: now $\EuA \supseteq \Bbbk_k[\bx]\oplus \langle a(\bx)\rangle$. Depending on $a(\bx)$ we have the following:
    \begin{enumerate}[label=({\it\roman*})]
        \item\label{item:lem-chubby} if $a(\bx)=(x^i)^{k+\ell}$ for some $i$ with $\ell\geq 2$, then $(x^i)^{k+\ell-1},...,(x^i)^{k+1}\in \EuA$,\hfill (\emph{lanky
        $\EuA$})
        \item\label{item:lem-lanky} otherwise (i.e.\ $x^ix^j|a(\bx)$ for some $j\neq i$) the algebra $\EuA$ is chubby.\hfill (\emph{chubby $\EuA$})
    \end{enumerate}
\end{lem}
We postpone the proof until we will have proven a structure Theorem~\ref{thm:structure-loneliness} (needed for the proof). See \S\ref{subSecNonExistence} for the proof on p.~\pageref{proof:lemma-chubby-tall-lanky}.
\begin{figure}[h!]
    \centering%
\mbox{
\begin{tikzpicture}[x=1.7cm,y=0.7cm]
    \def\yzero{0}
    \def\yk{2.2}
    \def\ykone{2.9}
    \def\yktwo{3.6}
    \def\ykthree{4.2} 
    \def\ykfour{4.8}

    \draw[->] (0,0) -- (0,5.3) node[above] {degree};

    \foreach \y/\label in {
        \yzero/{0},
        \yk/{k},
        \ykone/{k+1},
        \yktwo/{k+2},
        \ykfour/{k+\ell}
    }{
        \draw (-0.08,\y) -- (0.08,\y);
        \node[left] at (0,\y) {$\label~$};
    }

    \filldraw[fill=gray!25,draw=black]
        (0.2,\yzero) rectangle (1,\yk);
    \node at (0.6,{(\yzero+\yk)/2}) {\large $\Bbbk_k[\bx]$};

    
    \node[below] at (0.6,0) {\phantom{y}trivial $\subseteq$\phantom{y}};
\end{tikzpicture}\hspace{-0.8cm}%
\begin{tikzpicture}[x=1.7cm,y=0.7cm]
    \def\yzero{0}
    \def\yk{2.2}
    \def\ykone{2.9}
    \def\yktwo{3.6}
    \def\ykthree{4.2} 
    \def\ykfour{4.8}

    \draw[->] (0,0) -- (0,5.3) node[above] {degree};

    \foreach \y/\label in {
        \yzero/{0},
        \yk/{k},
        \ykone/{k+1},
        \yktwo/{k+2},
        \ykfour/{k+\ell}
    }{
        \draw (-0.08,\y) -- (0.08,\y);
        \node[left] at (0,\y) {$\label~$};
    }

    \filldraw[fill=gray!25,draw=black]
        (0.2,\yzero) rectangle (1,\yk);
    \node at (0.6,{(\yzero+\yk)/2}) {\large $\Bbbk_k[\bx]$};

    \fill (0.6,\ykone) circle (2.5pt);
    \node[above] at (0.6,\ykone) {$p(\bx)$};
    
    \node[below] at (0.6,0) {lonely $=$};
\end{tikzpicture}\hspace{0.05cm}%
\begin{tikzpicture}[x=1.7cm,y=0.7cm]
    \def\yzero{0}
    \def\yk{2.2}
    \def\ykone{2.9}
    \def\yktwo{3.6}
    \def\ykthree{4.2} 
    \def\ykfour{4.8}

    \draw[->] (0,0) -- (0,5.3) node[above] {degree};

    \foreach \y/\label in {
        \yzero/{0},
        \yk/{k},
        \ykone/{k+1},
        \yktwo/{k+2},
        \ykfour/{k+\ell}
    }{
        \draw (-0.08,\y) -- (0.08,\y);
        \node[left] at (0,\y) {$\label~$};
    }

    \filldraw[fill=gray!25,draw=black]
        (0.2,\yzero) rectangle (1,\yk);
    \node at (0.6,{(\yzero+\yk)/2}) {\large $\Bbbk_k[\bx]$};

    \fill (0.45,\ykone) circle (2.5pt);
    \node[above] at (0.35,\ykone) {$p(\bx)$};
    
    \fill (0.75,\ykone) circle (2.5pt);
    \node[above] at (0.85,\ykone) {$q(\bx)$};

    \node[below] at (0.6,0) {chubby $\supseteq$};

    \draw[->] (1.45,5.8) -- (0.45,5.8);

    \node[above] at (0.95,5.8) {\footnotesize{Lem.\,\ref{lem:chubby-tall-lanky}\ref{item:lem-chubby}}};
\end{tikzpicture}\hspace{-0.6cm}%
\begin{tikzpicture}[x=1.7cm,y=0.7cm]
    \def\yzero{0}
    \def\yk{2.2}
    \def\ykone{2.9}
    \def\yktwo{3.6}
    \def\ykthree{4.2} 
    \def\ykfour{4.8}

    \draw[->] (0,0) -- (0,5.3) node[above] {degree};

    \foreach \y/\label in {
        \yzero/{0},
        \yk/{k},
        \ykone/{k+1},
        \yktwo/{k+2},
        \ykfour/{k+\ell}
    }{
        \draw (-0.08,\y) -- (0.08,\y);
        \node[left] at (0,\y) {$\label~$};
    }

    \filldraw[fill=gray!25,draw=black]
        (0.2,\yzero) rectangle (1,\yk);
    \node at (0.6,{(\yzero+\yk)/2}) {\large $\Bbbk_k[\bx]$};

    \fill (0.6,\ykfour) circle (2.5pt);
    \node[above] at (0.6,\ykfour) {$p(\bx)$};
    
    \node[below] at (0.6,0) {\phantom{y}tall $\supseteq$\phantom{y}};
    \draw[->] (0.45,5.8) -- (1.1,5.8);

    \node[above] at (0.75,5.8) {\footnotesize{Lem.\,\ref{lem:chubby-tall-lanky}\ref{item:lem-lanky}}};
\end{tikzpicture}\hspace{-0.9cm}%
\begin{tikzpicture}[x=1.7cm,y=0.7cm]
    \def\yzero{0}
    \def\yk{2.2}
    \def\ykone{2.9}
    \def\yktwo{3.6}
    \def\ykthree{4.2} 
    \def\ykfour{4.8}

    \draw[->] (0,0) -- (0,5.3) node[above] {degree};

    \foreach \y/\label in {
        \yzero/{0},
        \yk/{k},
        \ykone/{k+1},
        \yktwo/{k+2},
        \ykfour/{k+\ell}
    }{
        \draw (-0.08,\y) -- (0.08,\y);
        \node[left] at (0,\y) {$\label~$};
    }

    \filldraw[fill=gray!25,draw=black]
        (0.2,\yzero) rectangle (1,\yk);
    \node at (0.6,{(\yzero+\yk)/2}) {\large $\Bbbk_k[\bx]$};

    \fill (0.6,\ykone) circle (2.5pt);
    \node[right] at (0.6,\ykone) {$(x^i)^{k+1}$};
    
    \fill (0.6,\yktwo) circle (2.5pt);
    \node[right] at (0.6,\yktwo) {$(x^i)^{k+2}$};
    
    \node at (0.6,\ykthree+0.15) {$\vdots$};
    
    \fill (0.6,\ykfour) circle (2.5pt);
    \node[right] at (0.6,\ykfour) {$(x^i)^{k+\ell}$};
    
    \node[below] at (0.6,0) {lanky $\supseteq$};
\end{tikzpicture}%
}
\vspace{-1.3cm}
    \caption{Visualisation of polynomial SH-Lie algebras from Definition~\ref{def:consistent-lonely-chubby-tall-lanky}. The set inclusions denote whether the subspace of polynomials from $\Bbbk[\bx]$ drawn is a subset of the algebra ($\subseteq$), equal to it ($=$), or contained in it ($\supseteq$).}
    \label{fig:visaliation-of-algebras}
\end{figure}
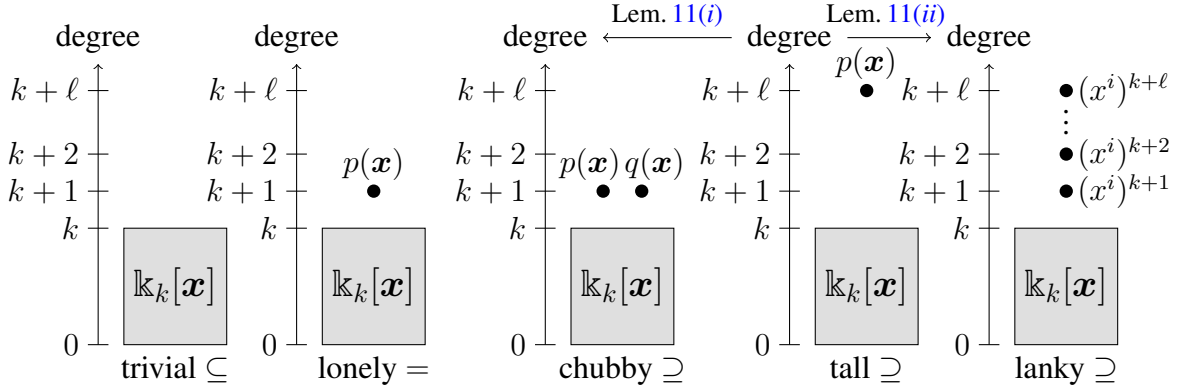

\subsection{Construction of finite-dimensional subalgebras and their structure constants}\label{subSecConstruction}
In this (sub)section we find all consistent (i.e.\ $\EuA\supsetneq \Bbbk_k[\bx]$ or, equivalently, $1\in \EuA$ such that also $1\in \Span_\Bbbk W_d^k(\EuA,...,\EuA)$) finite-dimensional polynomial subalgebras $\EuA\subsetneq \Bbbk[\bx]$ over any base dimension $d\geqslant 1$ and any differential order $k\geqslant 1$ in the complete generalised Wronskian determinant as the SH-Lie $N$\nobreakdash-ary bracket $[\cdot,...,\cdot]_N=W_{d\geqslant 1}^{k\geqslant 1}=\Bone\wedge \partial_\bx \wedges \partial^k_{\bx...\bx}$.

We now give a special case of the first main theorem: existence of a finite-dimensional class of polynomial algebras $\EuA$. Namely, suppose $\EuA$ is lonely and we restrict to the special case when $p$ of $\deg(p)=k+1$ is a \emph{monomial}; we denote the monomial by $a(\bx)$. To show that any homogeneous $p$ of $\deg(p)$ works too, we need a structure Theorem~\ref{thm:structure-loneliness}.

\begin{lem}[special case of our first main Theorem~\ref{thm:main-1-existence}: existence]
\label{lem:main-1-monom}
    Fix any dimension $d\geqslant 1$ and any order $k\geqslant 1$; set $N={d+k\choose k}$. Let $\EuA=\Bbbk_k[\bx]\oplus \langle a(\bx)\rangle$, where $a(\bx)$ is a \emph{monomial} of $\deg(a)=k+1$. Then $\EuA$ is \emph{closed} under the bracket $W_d^k=\Bone\wedge \partial_\bx \wedge\cdots\wedge \partial^k_{\bx\hdots \bx}$. In particular, $\EuA$ is a finite-dimensional subalgebra of $\Bbbk[\bx]$ with the Wronskian $W_d^k$ as the $N$\nobreakdash-ary bracket.
\end{lem}
\begin{proof}
    If all arguments in the Wronskian determinant are of order at most $k$, the result is a constant monomial; the sum of degrees in all $N$ monomials in $\Bbbk_k[\bx]$ is thus $kN/(d+1)$. Replacing the argument $\bx^{\brvec_j}/\brvec_j!\in \Bbbk_k[\bx]$ (see Notation~\ref{def:divpowermonomial-standardmonomial}) by the monomial $a(\bx)$ of degree $\deg(a)=k+1$ yields, via the Wronskian determinant $W_d^k$, the monomial of which the degree in coordinate $x^i$ is:
    \begin{align*}
        \deg_i W_d^k(1,\bx^{\brvec_2}/\brvec_2!,...,\underbrace{a(\bx)}_{\widehat{\bx^{\brvec_j}/\brvec_j!}},...,\bx^{\brvec_N}/\brvec_N!) &= \frac{kN}{d+1}-\deg_i(\bx^{\brvec_j}) + \deg_i a(\bx) - \frac{kN}{d+1} \\
        &= \deg_i a(\bx) - \deg_i (\bx^{\brvec_j}) \leq \deg_i a(\bx),
    \end{align*}
    where equality is attained only if $\deg_i(\bx^{\brvec_j})=0$. The total degree of the resulting monomial can therefore be $k+1$ only if $\deg_i(\bx^{\brvec_j})=0$ for every $i\in\{1,...,d\}$, which means that $j=1$ and $\bx^{\brvec_1}=1$; this yields the monomial $a(\bx)$ itself (up to a constant). In particular, no other monomials of degrees $k+1$ (or higher) can be obtained. The subspace $\EuA$ is thus closed under $W_d^k$. This completes the proof.
\end{proof}

\begin{ex}\label{ex:examples-of-findim-algebras-k1}
    Let the base dimension be $d=3$ and the differential order be $k=1$; here $\Bbbk_k[\bx]=\langle 1,x,y,z\rangle$. The polynomial subspaces
    \begin{align}
        \EuA &=\langle1,x,y,z,x^2\rangle, &
        &\text{or},&
        \EuA &=\langle1,x,y,z,xy\rangle, &
        &...,&
        &\text{or},&
        \EuA &=\langle1,x,y,z,z^2\rangle,
    \end{align}
    are all closed under the Wronskian determinant $W_{d=3}^{k=1}=\Bone\wedge \partial/\partial x\wedge \partial/\partial y\wedge \partial/\partial z$. For instance, setting $\EuA=\langle 1,x,y,z,xy\rangle$, we have
    \begin{align*}
        W_{d=3}^{k=1}(x,y,z,xy)=\det\begin{vmatrix}
            x&y&z&xy\\ 1&0&0&y\\ &1&0&x\\ &&1&0
        \end{vmatrix} &= xy,
        \intertext{calculated explicitly, and}
        W_{d=3}^{k=1}(1,y,z,xy)=\det\begin{vmatrix}
            1&1&1&1\\ &0&0&1\\ &1&0&1\\ &&1&0
        \end{vmatrix}y &= y,
    \end{align*}
    calculated by formula~\eqref{eq:wronskian-of-monoms-vandermonde} in Theorem~\ref{thm:general-vandermonde} on p.~\pageref{thm:general-vandermonde}. Notice that the monomial $z$ is not in the image of the bracket, $z\notin \Span_\Bbbk W_d^k(\EuA,...,\EuA)$, as any choice of independent monomial arguments from $\EuA=\langle1,x,y,z,xy\rangle$ can have at most sum-of-degrees in $z$ be $1=kN/(d+1)$.
\end{ex}

\begin{ex}\label{ex:examples-of-findim-d=k=2}
    Take both the dimension and differential order to be $d=2=k$. Then we obtain finite-dimensional polynomial algebras $\EuA=\langle 1,x,y,x^2,xy,y^2\rangle \oplus \langle a(\bx)\rangle$ for any of the choices $a(\bx)=x^3$, $x^2y$, $xy^2$, $y^3$. For instance, with $a(\bx)=xy^2$:
    \begin{equation*}
        W_{d=2}^{k=2}(1,x,y,x^2,xy,xy^2)=4x,
    \end{equation*}
    where we replaced the argument $y^2$ by $xy^2$. Note that there is no way to obtain the monomial $x^2$ as the result of the Wronskian determinant with arguments from this $\EuA=\Bbbk_2[\bx]\oplus \langle xy^2\rangle$: the arguments' degree sum minus shift in $x^1=x$ is at most $\deg_x a(\bx)=1$. The algebra $\EuA$ is hence not perfect (see Definition~\ref{def:perfect-lie-algebra} on p.~\pageref{def:perfect-lie-algebra} and Proposition~\ref{prop:imperfection} on p.~\pageref{prop:imperfection}).
\end{ex}

\begin{ex}[appending Laurent monomials] 
    Let us extend the finite-dimensional ternary polynomial algebra $\EuA=\langle 1,x,y, xy\rangle$ over base dimension $d=2$ by the span of Laurent monomials: $\EuA'=\EuA\oplus \EuL$, where $\EuL$ is generated by pure Laurent monomials and monomials with degree at most 1:
    \begin{equation}
        \EuL = \langle\{ \bx^{\bnvec} : \bnvec \in \BBZ^d_{\leqslant 1}\}\rangle=\langle x/y, x/y^2,..., y/x,y/x^2,..., 1/x, 1/y, 1/{xy}, 1/{xy^2}, ...\rangle.
    \end{equation}
    It is clear that $\EuL$ alone is closed under the ternary bracket: the only way to get a monomial with degree $>1$ in a coordinate is for all arguments from $\EuL$ to have that coordinate to degree 1; this makes the second row in the complete generalised Vandermonde determinant equal to the first, hence the determinant vanishes (see Proposition~\ref{prop:sufficient-vandermonde-zero}\,\ref{prop8-ii:coinciding-degrees} on p.~\ref{prop:sufficient-vandermonde-zero}).
    
    Setting $W_{d=2}^{k=1}=\Bone\wedge \partial_x\wedge \partial_y$, 
    we have the following ternary commutation relations:
    \begin{align*}
         [1,x,xy]&=x, &[1,y,xy]&=-y, &[x,y,xy]&=-xy,
         \intertext{describing the finite-dimensional algebra itself, and}
         [x,y,{1}/{y}]&=2/{y}, & [xy,{1}/{x},{1}/{y}]&=3/{xy} &[1,x,/{y}] &= -{1}/{y^2}, 
         \intertext{describing the interaction with the purely-negative degrees, and}
         [x,xy,{1}/{y}]&=x/{y}, & [x/y,xy,1/y]&=2x/y^4 &[x,x/y,x/y^4] &= 0, 
    \end{align*}
    which shows that $\EuA'=\EuA\oplus \EuL$ is closed. Note also that $[y,xy,x/y]=-2x\in \EuA$.
\end{ex}

The proof of Lemma~\ref{lem:main-1-monom} is silent about the coefficient -- our generalised Vandermonde determinant. In particular, can we take an homogeneous polynomial $p$ of total degree $\deg(p)=k+1$? We now give a structure theorem: an explicit factorisation of the complete generalised Vandermonde determinant\footnote{The complete generalised Vandermonde determinant involved (even in the triangular form) is difficult to expand; it is possible in $k=1$. See Appendix~\ref{SecAppExpansion} for the direct proof.} relevant for our case. Indeed, from $\EuA=\Bbbk_k[\bx]\oplus \langle p\rangle$ the only choices of $N$ monomial arguments are the divided power monomials (all from $\Bbbk_k[\bx]$) or one such monomial replaced by $p$. 

\begin{theor}[\emph{structure of loneliness}: structure constants for lonely algebras {$\EuA=\Bbbk_k[\bx]\oplus\langle a\rangle$}]\label{thm:structure-loneliness}
    Fix any dimension $d\geqslant 1$ and any differential order $k\geqslant 1$, and set $N={d+k\choose k}$. For any fixed index $j\in \{1,...,N\}$ and any \underline{\emph{monomial}} $a(\bx)=(x^1)^{n_1}\cdots (x^d)^{n_d}$ we have
    \begin{equation}\label{eq:structure-equation}
        W_d^k\Bigl(1,\frac{\bx^{\brvec_2}}{\brvec_2!},...,\underbrace{a(\bx)}_{\widehat{\bx^{\brvec_j}/\brvec_j!}},...,\frac{\bx^{\brvec_N}}{\brvec_N!}\Bigr) = \frac{(-)^{k-r_j}}{(k-r_j)!} \prod_{i=1}^d \left\{ \prod_{\ell=0}^{r_j^i-1} (n_i-\ell) \right\}\prod_{\ell=r_j+1}^k (\deg a-\ell)\cdot  \frac{a(\bx)}{\bx^{\brvec_j}},
    \end{equation}
    where if $\bx^{\brvec_j} \nmid a(\bx)$ then the first product in the coefficient is zero.
\end{theor}

\begin{rem}\label{rem:structure-constant-as-differential operator}
    With the same notation as in Theorem~\ref{thm:structure-loneliness} above, we rewrite the structure constant in terms of a differential. Namely, by Eq.~\eqref{eq:differential-operator-acting-on-monomial} on p.~\pageref{eq:differential-operator-acting-on-monomial} the first product in \eqref{eq:structure-equation} equals the coefficient of the differential operator $\partial^{|\brvec_j|} / \partial \bx^{\brvec_j}$ acting on $a(\bx)$, that is
    \begin{equation}\label{eq:structure-equation-as-differential}
        W_d^k\Bigl(1,\frac{\bx^{\brvec_2}}{\brvec_2!},...,\underbrace{a(\bx)}_{\widehat{\bx^{\brvec_j}/\brvec_j!}},...,\frac{\bx^{\brvec_N}}{\brvec_N!}\Bigr) = \frac{(-)^{k-r_j}}{(k-r_j)!} \prod_{\ell=r_j+1}^k (\deg a-\ell)\cdot \frac{\partial^{|\brvec_j|}}{\partial \bx^{\brvec_j}} a(\bx),
    \end{equation}
    i.e.\ the missing term differentiates the term that replaced it.
\end{rem}

\begin{rem}
    This structure theorem permits easy computation of the Wronskian bracket for any arguments in a lonely polynomial algebra $\EuA$.
\end{rem}

\begin{proof}[Proof (of Theorem~\ref{thm:structure-loneliness})]
    By Theorem~\ref{thm:general-vandermonde} about complete generalised Vandermonde determinants on p.~\pageref{thm:general-vandermonde} we immediately obtain the monomial term $a(\bx)/\bx^{\brvec_j}$; recall that whenever $\bx^{\brvec_j}$ does \emph{not} divide the monomial $a(\bx)$, the overall constant equals zero, hence the division by higher degrees (which would result in negative degrees) is fictitious (see Eq.~\eqref{eq:differential-operator-acting-on-monomial} on p.~\pageref{eq:differential-operator-acting-on-monomial} and the surrounding commentary). It remains to work out that constant -- by simplifying the generalised Vandermonde determinant. Note that the determinant contains $\bn=(n_1,...,n_d)$ values only in one column; their products go up to at most total degree~$k$ in them. We view the determinant as a polynomial $g \in \Bbbk[n_1,...,n_d]$; the polynomial can have at most $k$ factors linear in $n_i$'s (if linear factors exist). 

    Split this constant into a part that depends on $\bn$ multiplied by an overall constant independent of the degrees of coordinates, $\bn$, of $a(\bx)$:
    \begin{equation}\label{eq:n-indep-part-n-dep-part}
        W_d^k\Bigl(1,\frac{\bx^{\brvec_2}}{\brvec_2!},...,\underbrace{a(\bx)}_{\widehat{\bx^{\brvec_j}/\brvec_j!}},...,\frac{\bx^{\brvec_N}}{\brvec_N!}\Bigr) = (\text{const. indep. of } \bn)\cdot (\text{number dep. on } \bn)  \cdot  \frac{a(\bx)}{\bx^{\brvec_j}}.
    \end{equation}
    We work out the $\bn$-dependent part first.
    
    Note that taking $a(\bx)\sim \bx^{\brvec_{j'}}$ for some $j'\neq j$ we will obtain the zero polynomial. Indeed, for any choice of monomial $a$ of degree $r_j<\deg(a)\leq k$, the monomial $a$ must hit another monomial, as the set of arguments is complete up to degree $k$. This gives $k-r_j$ factors of the form $(\deg a - \ell)$ for $\ell$ values $\ell \in \{r_j+1,..., k\}$, hence their product $\prod_{\ell=r_j+1}^k (\deg p-\ell)$ divides the monomial $g$; we let $f$ be the quotient obtained by division.

    As $\deg_{\bn = (n_1,...,n_d)} g =k$ and the product above is of degree $k-r_j$, we find that $\deg_\bn f= r_j$. Let $i\in\{1,..,d\}$ be arbitrary. Then for the choices $1,x^i,(x^i)^2,...,(x^i)^{r^i_j}$ of $a(\bx)$ we obtain a repeated argument. This gives the linear factors $n_i-\ell$ for $ \ell \in\{0,..., r_j^i\}$. We thus obtain the first product in~\eqref{eq:structure-equation}. Inspecting the degree we find $r_j^i+\cdots +r_j^d=r_j$, therefore we have found all the linear ($\bn$-dependent) factors (hence the $\bn$-dependent part in Eq.~\eqref{eq:n-indep-part-n-dep-part}).

    All that is left is to determine the overall constant (independent of $\bn$). For this we use that the choice $p(\bx)=\bx^{\brvec_j}$ gives $W_d^k(...,p(\bx)=\bx^{\brvec_j},...)=\brvec_j!$. The first product in \eqref{eq:structure-equation} yields $\brvec_j!$, whereas the second yields $(-)^{k-r_j}(k-r_j)!$. This completes the proof.
\end{proof} 

We now state our first main theorem (recall the special case in Lemma~\ref{lem:main-1-monom}).

\begin{theor}[main theorem~\textbf{1}: lonely algebras are finite-dimensional]\label{thm:main-1-existence}
    Fix any dimension $d\geqslant 1$ and any order $k\geqslant 1$; set $N={d+k\choose k}$. Let the polynomial subspace $\EuA$ be $\EuA=\Bbbk_k[\bx]\oplus \langle p\rangle$, where $p$ is any homogeneous polynomial of total degree $\deg(p)=k+1$. Then the subspace $\EuA$ is \underline{\emph{closed}} under the bracket given by $W_d^k=\Bone\wedge \partial_\bx \wedge\cdots\wedge \partial^k_{\bx\hdots \bx}$. In conclusion, $\EuA$ is a \underline{\emph{finite-dimensional}} $N$\nobreakdash-ary SH-Lie subalgebra of $\Bbbk[\bx]$ with the complete generalised Wronskian determinant as the $N$\nobreakdash-ary bracket.
\end{theor}
\begin{proof}
    By Theorems~\ref{thm:general-vandermonde} and~\ref{thm:structure-loneliness}, it remains only to consider the bracket value $W_d^k(p(\bx),\frac{\bx^{\brvec_2}}{\brvec_2!},$ $...,$ $\frac{\bx^{\brvec_N}}{\brvec_N!})$. Split $p(\bx)$ into its constituent sum of monomials in $\Bbbk_{k+1}[\bx]$ (not necessarily belonging to $\EuA$); by multilinearity we consider each term separately. By Theorem~\ref{def:gen-vandermonde} (or~\ref{thm:structure-loneliness}) the resulting monomial from each monomial argument is proportional to the monomial itself; it remains to show that each resulting monomial is preceded by the same overall constant. From formula~\eqref{eq:structure-equation} with $\brvec_1=(0,...,0)$, so $r_1=0$, and as the degree of each constituent monomial in $p(\bx)$ is equal to $k+1$, we have that each monomial is preceded by the factor $(-)^k/k!\cdot \prod\left\{\prod_\varnothing\right\} \cdot (k+1-1)(k+1-2)\cdots (k+1-k)=(-)^k$, which is independent of the monomial. Hence the resulting polynomial is proportional to $p(\bx)$ itself, as required.
\end{proof}

\subsection{(Im)perfect $N$\nobreakdash-ary SH-Lie polynomial algebras}\label{subSecImperfect}
Any (semi-)simple Lie algebra is perfect. The definition of perfect Lie algebras extends to the $N$\nobreakdash-ary case (see Definition~\ref{def:perfect-lie-algebra} on p.~\pageref{def:perfect-lie-algebra}) -- the (span of the) $N$\nobreakdash-ary bracket is surjective: $\Span_{\Bbbk} W_d^k(\EuA,...,\EuA)=\EuA$. In this section we characterise which of the finite dimensional $N$\nobreakdash-ary SH-Lie algebras found in Lemma~\ref{lem:main-1-monom} on p.~\pageref{lem:main-1-monom} and the first main Theorem~\ref{thm:main-1-existence} are perfect: we show that over arbitrary base dimension $d\geqslant 1$ and differential order $k\geqslant 1$, modulo a list of exceptions, a \emph{monomial} as the top-degree polynomial (that is, a monomial $a(\bx)$ of degree $\deg(p)=k+1$ in the lonely algebra $\EuA=\Bbbk_k{\bx}\oplus \langle a(\bx)\rangle$) is \emph{not} enough for the algebra to be perfect. We give an example of a perfect finite-dimensional algebra with an homogeneous polynomial on top. Determining which homogeneous polynomial as the top-degree polynomial results in a perfect (lonely) $N$\nobreakdash-ary SH-Lie polynomial algebra remains an open problem.

\begin{prop}[monomial imperfection]\label{prop:imperfection}
    Let $\EuA$ be a consistent lonely $N$\nobreakdash-ary SH-Lie polynomial subalgebra, i.e.\ $\EuA=\Bbbk_k[\bx]\oplus \langle a(\bx)\rangle$, where $a(\bx)=(x^1)^{n_1}\cdots (x^d)^{n_d}$ is a \emph{monomial}\,\footnote{This assumption is essential. In Example~\ref{ex:perfection} homogeneous polynomials will alleviate this imperfection.} of degree $\deg (a)=n_1+\cdots+n_d=k+1$. Then the algebra $\EuA$ is \emph{not perfect}, i.e.\ the image of the bracket (given by the complete generalised Wronskian of differential order $k$ over base dimension $d$) does not span the algebra, $\Span_\Bbbk W_d^k(\EuA,...,\EuA)\neq\EuA$, unless either of the following cases is satisfied:
    \begin{align*}
        \text{(\textit{i}) }& d=1 \text{ (and any $k\geqslant 1$ value)}, & &\text{or} &
        \text{(\textit{ii}) }& d=2 \text{ and } k=1;
    \end{align*}
    i.e.\ every lonely (thus finite-dimensional) algebra in dimension one is perfect, whereas over multi-dimension there exists only one example of a perfect algebra with a monomial on top.
\end{prop}

Note that the exception (\textit{i}) corresponds to the class of perfect finite-dimensional algebras found in Example~\ref{ex:removed-factorial}, and (\textit{ii}) to Example~\ref{ex:dimension-two-closed-old} on p.~\pageref{ex:dimension-two-closed-old}.

\begin{rem}
    Over dimension $d=1$ with coordinate $x$ there is only one consistent finite-dimensional $N$\nobreakdash-ary polynomial SH-Lie algebra $\EuA=\langle 1,x/1!,x^2/2!,...,x^N/N!\rangle$ with the ordinary Wronskian determinant $W_{d=1}^{0,1,...,N-1}$ bracket of differential order $N-1$ and $N$ arguments; this algebra is perfect, $\Span_\Bbbk W_{d=1}^{0,1,...,N-1}(\EuA,...,\EuA)=\EuA$, and, according to the given Definition~\ref{def:consistent-lonely-chubby-tall-lanky}, lonely.
\end{rem}

\begin{proof}[Proof (of Proposition~\ref{prop:imperfection})]
    By Theorem~\ref{thm:general-vandermonde}, taking brackets of monomial arguments yields monomials; by multilinearity it then suffices to consider the span of brackets of monomial arguments. Supposing the algebra $\EuA$ is perfect, we must have the monomial $(x^i)^k$ in the image of the Wronskian determinant. By Theorem~\ref{thm:general-vandermonde}, the only way to obtain this monomial is through the choice of argument $a(\bx)$ such that $(x^i)^k=a(\bx)/x^\ell$ for some $\ell\in \{1,..., d\}$. Then $(x^i)^k | a(\bx)$ for every $i\in \{1,..., d\}$, which implies that the least common multiple $(x^1)^k\cdots (x^d)^k$  divides $p(\bx)$. It follows that $kd\leq k+1$, or, equivalently, $ k(d-1)\leq 1$. The only solutions of $k,d\in \BBN_{\geqslant 1}$ that satisfy the requirement are (\textit{i}) and (\textit{ii}). 
\end{proof}

\begin{rem}[shrinking of the algebra by iterating the bracket\footnote{The authors thank Maxim Kontsevich for this question, asked during the talk by the second author at the Mathematics Seminar, IH\'ES. In this remark we develop the question more precisely and answer it.}]\label{rem:shrinking-of-algebra}
    Let the base dimension be $d\geqslant 2$ and put the Wronskian $W_d^k$ as the $N$\nobreakdash-ary bracket.
    Take a \emph{monomial} $a(\bx)$ as the $(k+1)$th degree polynomial in a lonely algebra: $\EuA=\Bbbk_k[\bx]\oplus\langle a(\bx)\rangle$, hence $\dim_\Bbbk \EuA=N+1$. Suppose the SH-Lie algebra $\EuA$ is \emph{not perfect}. Upon applying the bracket, we will lose some monomials as arguments:
    \begin{align*}
        \Span_\Bbbk W_d^k (\EuA,...,\EuA) &\subsetneq \EuA,&  &\text{hence} & 
        \dim \Span_\Bbbk W_d^k (\EuA,...,\EuA) &\leqslant N.
    \end{align*}
    If at least two monomials are lost, we no longer have enough linearly-independent arguments to obtain a non-vanishing bracket. Supposing only one monomial argument is lost in the span, we now have only one set of $N$ linearly-independent arguments, hence re-applying the bracket yields only one monomial up to constant (possibly zero).

    Therefore only one or two iterated applications of the brackets are possible before the bracket of any arguments will vanish identically: the dimension decreases as
    \begin{align}\label{eq:descreasing-dimension}
        N+1\to N&\to 1\to 0& &\text{or} & N+1\to N&\to 0, & &\text{or as}& N+1\to &(<N) \to 0.
    \end{align}
    The problem is to determine for which choices of $a(\bx)$, which case in Eq.~\eqref{eq:descreasing-dimension} is obtained?

    Let us now answer this question by considering when the first two cases in~\eqref{eq:descreasing-dimension} can be satisfied.
    Suppose the degree of $x^i$ in $a(\bx)$ is less than $k$; it is then not possible to obtain the monomial $(x^i)^k$ as the result of the bracket. In the next iteration, any choice of $N$ independent monomials from $\Span_\Bbbk W_d^k(\EuA,...,\EuA)$ as arguments will then be deficient in $x^i$ (see Definition~\ref{def:defincient-exact-abundant-promising} on p.~\pageref{def:defincient-exact-abundant-promising} and Proposition~\ref{prop:sufficient-vandermonde-zero}\,\ref{prop8-iii:prop-sufficient-vandermonde:deficiency} on p.~\pageref{prop:sufficient-vandermonde-zero}), hence all brackets will vanish. The first case is thus \emph{impossible}.
    If in the first iteration only one monomial is lost, it must be $(x^i)^k$, as losing any other immediately implies that more are lost (increase the degree of any coordinate while decreasing another). The dimension must be $d=2$, as otherwise more than one monomial of the form $(x^i)^k$ will be lost; the differential order cannot exceed $k=2$ as for $k\geqslant 3$, setting $a(\bx)=x^ky$, we lose $y^2,...,y^k$.

    Therefore the cases in~\eqref{eq:descreasing-dimension} are satisfied only by:
    \begin{align}
        N+1&\to N\to 1\to 0 & &\text{impossible},\\
        N+1&\to N\to 0 & &(d=2=k)\colon\quad a(\bx)=x^2 y \text{ or } xy^2,\text{ and}\notag{}\\
           &           & &(d=2, k=1)\colon\quad a(\bx)=x^2 \text{ or } y^2,\label{eq:middle-case=N}\\
        N+1&\to (<N) \to 0 &&\text{every other choice of } a(\bx).\label{eq:decreasing-case-<N}
    \end{align}
    This solves the problem of decreasing imperfect algebras \emph{of monomials} by iterated brackets. Among the monomials in case~\eqref{eq:decreasing-case-<N}, it is a problem of combinators to determine to what value the dimension shrinks in each case: in each base variable the resulting monomial has (non-negative) degree not exceeding that of $a(\bx)=(x^1)^{n_1}\cdots (x^d)^{n_d}$; the number of such monomials is $\prod_{i=1}^d (n_i + 1)$. Case~\eqref{eq:middle-case=N} corresponds precisely to when this product equals $N$.
\end{rem}

\begin{ex}[perfection!]\label{ex:perfection}
    In Proposition~\ref{prop:imperfection} the assumption that $a(\bx)$ is a monomial was \emph{essential}.
    Let the base dimension be $d=3$ and the differential order be $k=1$. By Proposition~\ref{prop:imperfection} then, a 2nd degree monomial $a(\bx)$ as the top-degree polynomial is not enough to form a perfect algebra $\EuA=\langle1,x,y,z,a(\bx)\rangle$. We now illustrate that an homogeneous polynomial of degree 2 can give a perfect algebra: put $p(\bx)=xy+yz+zx$, which gives for $\EuA=\langle 1,x,y,z,xy+yz+zx\rangle$ the commutation relations $[x,y,z,p(\bx)]=p(\bx)$,
    \begin{align}\label{eq:commut-rels-perfect-algebra}
        [1,x,y,p(\bx)]&=x+y, & [1,x,z,p(\bx)]&=-x-z, &&\text{and} & [1,y,z,p(\bx)]&=y+z,
    \end{align}
    from which $\langle W_{d=3}^{k=1}(\EuA,\EuA,\EuA,\EuA)\rangle=\EuA$ follows as $x=((x+y)+(-x-y)+(y+z))/2$ using which we also obtain the remaining monomials $y$ and $z$.

    Notice that each r.\!-h.s.\ in \eqref{eq:commut-rels-perfect-algebra} contains exactly two monomial terms; the coordinate missing in the r.\!-h.s.\ was missing from the set of arguments in the l.\!-h.s. (cf.\ Proposition~\ref{prop:sufficient-vandermonde-zero} on p.~\pageref{prop:sufficient-vandermonde-zero}, condition \ref{prop8-iii:prop-sufficient-vandermonde:deficiency} on deficiency). For instance, the degree in $z$ among $1,x,y,xy$ sum up to $0<kN/(d+1)=1$. We will return to this point in the discussion (see Definition~\ref{def:defincient-exact-abundant-promising} on p.~\pageref{def:defincient-exact-abundant-promising}).
\end{ex}

\begin{open}\label{open:perfect-algebras}
    Let the base dimension $d\geqslant 2$ and differential order $k\geqslant 1$ be arbitrary. Put the complete generalised Wronskian $W_d^k=\Bone\wedge \partial_\bx\wedges \partial^k_{\bx...\bx}$ as the $N$\nobreakdash-ary bracket, $N={d+k\choose d}$, on the $N$\nobreakdash-ary SH-Lie polynomial algebra $\Bbbk[x^1,...,x^d]$. Does there exist an homogeneous polynomial $p(\bx)\in \Bbbk[\bx]$ of degree $k+1$ such that the $N$\nobreakdash-ary SH-Lie polynomial subalgebra $\EuA=\Bbbk_k[\bx]\oplus\langle p(\bx)\rangle$ is \emph{perfect} w.r.t\ the Wronskian, i.e.\ $\Span_\Bbbk W_d^k(\EuA,...,\EuA)=\EuA$?
\end{open}

\begin{comm}[to Open problem~\ref{open:perfect-algebras}]
    Fix $d\geqslant 2$ and $k\geqslant 1$ as in Problem~\ref{open:perfect-algebras}. Build the homogeneous polynomial $p(\bx)$ from $(k+1)$th degree monomials; there are a total of
    \begin{equation}\label{eq:number-monoms-top-degree}
        {d+(k+1)-1\choose k+1} = {d+k\choose k+1} = \frac{d}{k+1} N
    \end{equation}
    monomials of total degree $k+1$ (cf.\ Lemma~\ref{lem:combinatorial} on p.~\pageref{lem:combinatorial}, where the fraction in \eqref{eq:number-monoms-top-degree} is inverted). The dimension of the algebra $\EuA$ is $N+1$, where $N$ monomials come from $\Bbbk_k[\bx]$, of which one is the constant monomial; the polynomial $p(\bx)$ is always in the image (see Theorem~\ref{thm:structure-loneliness}). As we demanded that the algebra $\EuA$ is consistent, the unit constant polynomial is also in the image. Hence for the algebra to be perfect, all non-constant monomials up to degree $k$ must be in the span of the (homogeneous\footnote{The total degree of $q_j(\bx)$ is $\deg q_j(\bx)=k+1-r_j$, where $r_j$ is the total degree of the replaced monomial as argument.}) $N-1$ polynomials $q_2(\bx),...,q_N(\bx)$ obtained by
    \begin{equation}
        q_j(\bx)=W_d^k(1,\bx^{\brvec_2}/\brvec_2!,...,\widehat{\bx^{\brvec_j}/\brvec_j!}\mapsto p(\bx),...,\bx^{\brvec_N}/\brvec_N!);
    \end{equation}
    (see Theorem~\ref{thm:structure-loneliness} and Remark~\ref{rem:structure-constant-as-differential operator} on p.~\pageref{thm:structure-loneliness} for the computation and Notations~\ref{not:multi-index} and~\ref{not:jth-row-wronskian} on p.~\pageref{not:jth-row-wronskian} for the power $d$-tuple notation). The problem is thus two-fold:
    \begin{enumerate}[label=({\it\roman*})]
        \item\label{item:comment-open-perfect:combinatorial} (combinatorial part) calculate how many and which constituent monomials make up each polynomial $q_j(\bx)$ (depending on $j$, so on $r_j=|\brvec_j|$ and $\brvec_j=(r_j^1,...,r_j^d)$);
        \item\label{item:comment-open-perfect:linear-algebra} (linear algebra part) determine whether the $N-1$ polynomials obtained this way are linearly independent (hence their span is $(N-1)$-dimensional as required).
    \end{enumerate}
    Part~\ref{item:comment-open-perfect:linear-algebra} can be reduced to computing an $(N-1)\times (N-1)$ determinant (of numbers).
\end{comm}

\subsection{Non-existence of other (\emph{consistent}) finite-dimensional SH-Lie polynomial algebras}\label{subSecNonExistence}
In this section we look beyond the class of \emph{lonely} $N$\nobreakdash-ary strongly homotopy (SH) Lie polynomial algebras $\EuA=\Bbbk_k[\bx]\oplus\langle p(\bx)\rangle$ (where $p(\bx)$ is an homogeneous polynomial of degree $\deg p=k+1$) studied in \S\ref{subSecConstruction}. We recall that every lonely algebra is closed w.r.t\ the complete generalised Wronskian determinant (of differential order $k$ over dimension $d$) $W_d^k=\Bone\wedge \partial_\bx\wedges\partial^k_{\bx...\bx}$, and hence finite-dimensional (see the first main Theorem~\ref{thm:main-1-existence} on p.~\pageref{thm:main-1-existence} and Definition~\ref{def:consistent-lonely-chubby-tall-lanky} on p.~\pageref{def:consistent-lonely-chubby-tall-lanky} for the relevant definitions). We now show that a consistent (i.e.\ $\Bbbk_k[\bx]\subsetneq \EuA\subseteq \Bbbk[\bx]$) algebra $\EuA$ is finite-dimensional if and \emph{only if} $\EuA$ contains only a single homogeneous polynomial of degree $k+1$ (i.e\ $\EuA$ is lonely) and no other polynomial(s) of higher degree(s). In the end of this section we discuss whether we can drop the assumption that the algebra $\EuA$ is consistent (or, equivalently, by Lemma~\ref{lem:consistency} that the constant unit polynomial $1$ is both in the algebra $1\in \EuA$ and the image of the bracket: $1\in \Span_\Bbbk W_d^k(\EuA,...,\EuA)$); we formulate an open problem about this.

\medskip
Essentially, our goal is to produce monomials of higher degree from a collection of (poly- or) monomials in a consistent algebra using the complete generalised Wronskian determinant. To this end, let us consider algebras which are not precisely lonely; specifically, let us consider consistent algebras which contain a polynomial of too high a degree ($>k+1$) or too many (at least two) linearly independent polynomials of degree $k+1$; see Figure~\ref{fig:visaliation-of-algebras} on p.~\pageref{fig:visaliation-of-algebras} for an illustration.

Recall that Lemma~\ref{lem:chubby-tall-lanky} states that a tall algebra reduces to a chubby or a lanky algebra (see Definition~\ref{def:consistent-lonely-chubby-tall-lanky} on p.~\pageref{def:consistent-lonely-chubby-tall-lanky}) depending on the choice of top-degree polynomial $p(\bx)$. It suffices, by multilinearity of the Wronskian determinant, to let $p(\bx)$ be a monomial, $p(\bx)=a(\bx)$.\footnote{\label{foot:consideration-of-pure-monomials}
    This is because, for our purposes, it will be enough to consider $N-1$ or $N-2$ monomials from $\Bbbk_k[\bx]$ as arguments with the only remaining one or two arguments being polynomials: denote by and decompose the two polynomial arguments $p(\bx)$ and $q(\bx)$ into their constituent monomials. Now calculate the Wronskian of $N-2$ monomial arguments and the remaining two arguments being a pair of constituent monomials from $p$ and $q$, respectively. By Theorem~\ref{thm:general-vandermonde} in this case, the same monomial can be produced with the opposite sign (thus it would cancel in the expansion by multilinearity) only if the two constituent monomials are swapped as arguments, as shown below.
    
    Formally, decompose the two polynomials $p,q$ into their constituent monomials: $p(\bx)=c_1 a(\bx) + c_2 b(\bx)$ and $q(\bx)=d_1 a(\bx)+d_2 b(\bx)$, where all four constants are non-zero. Expand the Wronskian of arguments $p,q$ and $N-2$ monomials from $\Bbbk_k[\bx]$ by multilinearity. Denote by $(a,b)$ the monomial obtained from the choice of arguments $a,b$ and the same $N-2$ monomials. Then the result of the expanded Wronskian is
    \begin{equation*}
        c_1 d_1 (a,b) + c_2 d_1 (b,a) = (c_1d_2 - c_2d_1) (a,b).
    \end{equation*}
    If the expansion cancels, we obtain $c_1d_2=c_2d_1$, from which it quickly follows that $p$ and $q$ are not linearly independent.
    As $p(\bx)$ and $q(\bx)$ are, by assumption, linearly independent, they must contain at least one pair of overall non-proportional monomial arguments. Then it suffices to consider as arguments of the Wronskian a pair of constituent monomials and $N-2$ monomials of degree at most $k$, of which the Wronskian does not vanish identically after expansion; put these monomials as the new $p$ and $q$ themselves.
}

\medskip
We now \emph{re}-state and prove Lemma~\ref{lem:chubby-tall-lanky} (first stated on p.~\pageref{lem:chubby-tall-lanky}).

\begin{manuallemma}{\ref{lem:chubby-tall-lanky}}
    Suppose $\EuA$ is a tall polynomial SH-Lie algebra, i.e.\ $\EuA$ contains a monomial%
    \footnote{More generally, for a polynomial $p(\bx)$ in place of $a(\bx)$, extend the results of the lemma by decomposing $p(\bx)$ into its constituent monomials and applying the lemma to each monomial.}  
    $a(\bx)$ of degree $\deg(a)>k+1$, and suppose $\EuA$ is consistent: now $\EuA \supseteq \Bbbk_k[\bx]\oplus \langle a(\bx)\rangle$. Depending on $a(\bx)$ we have the following:
    \begin{enumerate}[label=({\it\roman*})]
        \item if $a(\bx)=(x^i)^{k+\ell}$ for some $i$ with $\ell\geq 2$, then $(x^i)^{k+\ell-1},...,(x^i)^{k+1}\in \EuA$,\hfill (\emph{lanky
        $\EuA$})
        \item otherwise (i.e.\ $x^ix^j|a(\bx)$ for some $j\neq i$) the algebra $\EuA$ is chubby.\hfill (\emph{chubby $\EuA$})
    \end{enumerate}
\end{manuallemma}

\begin{proof}[Proof of Lemma~\ref{lem:chubby-tall-lanky}]\label{proof:lemma-chubby-tall-lanky}
    (\textit{i}) As $\ell>1$, then by the structure Theorem~\ref{thm:structure-loneliness} on p.~\pageref{thm:structure-loneliness}, taking the bracket with $x^i$ as argument replaced by $a(\bx)$ yields the monomial $(x^i)^{k+\ell-1}$ of preceding degree times a non-zero constant. Repeating this descent $\ell\mapsto \ell-1$ until $(x^i)^{k+1}$ is obtained verifies the claim.

    (\textit{ii}) Similarly to (\textit{i}), but now there are at least two different ways to divide (at least by either $x^i$ or $x^j$), obtaining non-proportional monomials.
\end{proof}

Unlike we had only one monomial $a(\bx)$ in the structure Theorem~\ref{thm:structure-loneliness}, let us now replace two standard monomials in the set of arguments for the Wronskian $W_{d\geqslant 1}^{k\geqslant 1}$. Specifically, let us replace $1=\bx^{\brvec_1}$ and $x^1=\bx^{\brvec_2}$ by \emph{monomials} $q(\bx)$ and $p(\bx)$, respectively. The total degree of the resulting monomial from the Wronskian of the above arguments is at least $2k+1>k+1$ if the coefficient in front is non-zero. The task is to factorise the complete generalised Vandermonde determinant; we do this in Theorem~\ref{thm:golden-formula} on p.~\pageref{thm:golden-formula}. The examples below suggest the general form, which we prove:

\begin{ex}[$k=1$, $d=2$, $\widehat{x}\mapsto p$]\label{ex:k1d2-curly-bracket-formula}
     Let the base dimension be $d=2$ and consider the complete generalised Wronskian of differential order $k=1$: $W_{d=2}^{k=1}=\Bone\wedge \partial/\partial x\wedge \partial/\partial y$. Let $p(\bx)=x^{n_1}y^{n_2}$ and $q(\bx)=x^{m_1}y^{m_2}$ be arbitrary monomials. Then,
     \begin{align}
             W_{d=2}^{k=1}(q,p,y)&=\det\begin{vmatrix}
                 1&1&1\\ m_1&n_1&0\\ m_2&n_2&1
             \end{vmatrix} \frac{pq}{x}\notag{}
             = \{(n_2-1)m_1 - (m_2-1)n_1\}\frac{pq}{x}\\
             &= \{(n_1+n_2-1)m_1 - (m_1+m_2-1)n_1\}\frac{pq}{x},\notag{}
             \intertext{where we were able to add $n_1$ and $m_1$ in the round brackets as they cancel by skew-symmetry of the coefficient, giving}%
             &= \{(\deg p-1)m_1 - (\deg q -1 )n_1\}\frac{pq}{x}.\label{eq:curly-bracket-example}
     \end{align}
     Notice how the coefficient intertwines the degrees of the two monomials $p$ and $q$ by the others' degree in $x$ (i.e.\ $\deg(p)$ mixed with $m_1$), which is the replaced non-constant argument.
\end{ex}

\begin{ex}[$k=1$, any $d\geqslant 2$, $\widehat{x^j}\mapsto p$]\label{ex:arbitrary-replaced-monomial}
    Now let the base dimension $d\geqslant 2$ be arbitrary and set the differential order to be 1; this gives the first order Wronskian $W_{d\geqslant 2}^{k=1}=\Bone\wedge \partial/\partial x^1\wedges \partial/\partial x^d$. Replace $1$ as argument by $q(\bx)=(x^1)^{m_1}\cdots(x^d)^{m_d}$ and $x^j$ by $p=(x^1)^{n_1}\cdots(x^d)^{n_d}$; by direct expansion (see Appendix~\ref{SecAppExpansion} on p.~\pageref{SecAppExpansion}) we obtain
    \begin{equation}
        W_{d\geqslant 2}^{k=1}(1,x^1,...,x^{j-1},p,x^{j+1},...,x^d)=\{(\deg p-1)m_j-(\deg q-1)n_j\}\frac{pq}{x^j}.
    \end{equation}
    Notice (cf.\ Eq.~\ref{eq:differential-operator-acting-on-monomial} on p.~\pageref{eq:differential-operator-acting-on-monomial} and the surrounding discussion) that if $x^j$ does not divide the product $pq$, i.e.\ $n_1=m_1=0$, then the coefficient vanishes identically (so the division by $x^j$ is fictitious).
\end{ex}

\begin{ex}[$d=2$, $k=2,3,4,5$, $\widehat{x}\mapsto p$]\label{ex:comp-examples-golden-formula}
    Fix the base dimension $d=2$ with coordinates $x,y$ and vary the total differential order $k\in\{2,3,4,5\}$ (cf. the case $k=1$ in Example~\ref{ex:k1d2-curly-bracket-formula}).
    We now showcase computational results for the $k$ values above:\footnote{
        To compute the complete generalised Wronskian determinant symbolically, we write and run a script using the Python package SymPy; this Script~\ref{script:Wronskian-pq} is given in Appendix~\ref{SecAppPyScripts}. The computations were performed on the home machine of the first author and on the computational cluster `H\'abr\'ok' of the University of Groningen; see Appendix~\ref{SecAppRawOutputs} for the raw computational outputs.
    }
    \begin{subequations}\label{eqs:comp-gf-highk-k2345}
    \begin{align}
        W_{d=2}^{k=2}(q,p,y,x^2,xy,y^2)&=2\cdot (\deg p-2)(\deg q-2)\cdot \notag{}\\
                          &\qquad\cdot \{(\deg p-1)m_1 - (\deg q-1)n_1\}\frac{pq}{x} \label{eq:comp-gf-highk-k2} \\
        W_{d=2}^{k=3}(q,p,y,x^2,...,xy^2,y^3)&=2^4\,3\cdot \prod\nolimits_{\ell=2,3}(\deg p-\ell)(\deg q-\ell)\cdot \notag{}\\
                          &\qquad\cdot \{(\deg p-1)m_1 - (\deg q-1)n_1\}\frac{pq}{x} \label{eq:comp-gf-highk-k3} \\
        W_{d=2}^{k=4}(q,p,y,x^2,...,xy^3,y^4)&=2^{12}\,3^4\cdot \prod\nolimits_{\ell=2,3,4}(\deg p-\ell)(\deg q-\ell)\cdot \notag{}\\
                          &\qquad\cdot \{(\deg p-1)m_1 - (\deg q-1)n_1\}\frac{pq}{x} \label{eq:comp-gf-highk-k4} \\
        W_{d=2}^{k=5}(q,p,y,x^2,...,xy^4,y^5)&=2^{26}\,3^{10}\,5\cdot \prod\nolimits_{\ell=2,3,4,5}(\deg p-\ell)(\deg q-\ell)\cdot \notag{}\\
        &\qquad\cdot\{(\deg p-1)m_1 - (\deg q-1)n_1\}\frac{pq}{x} \label{eq:comp-gf-highk-k5}
    \end{align}
    \end{subequations}
    We emphasise that we have not normalised the standard monomials above (basically, not using $x^3y^2\mapsto x^3y^2/3!\,2!$). Let us calculate the overall normalisation constant per degree:
    \begin{equation}\label{eq:normalisation-consts-per-degree}
        \text{
        \begin{tabular}{c|c|cl}
            degree & monomials & \multicolumn{2}{c}{product of normalisation constants} \\ 
            2 & $x, y$ & $1!\,1!$&$=1$ \\
            2 & $x^2, xy,y^2$ & $2!\,2!$&$=2^2$ \\
            3 & $x^3,x^2y,xy^2,y^3$ & $3!\,2!\,2!\,3!$&$=2^4\,3^2$ \\
            4 & $x^4,x^3y,x^2y^2, xy^3,y^4$ &$4!\,3!\,2!\,2!\,3!\,4!$&$= 2^{10}\,3^4$ \\
            5 & $x^5,x^4y,x^3y^2,x^2y^3, xy^4,y^5$ &$5!\,4!\,3!\,2!\,2!\,3!\,4!\,5!$&$= 2^{16}\,3^6\,5^2$, \\
        \end{tabular}
        }
    \end{equation}
    Observe that dividing the product of normalisation degree from degree 1 up to $k$ by $k!\,(k-1)!$ the overall coefficients in Eqs~\eqref{eqs:comp-gf-highk-k2345} are obtained (for instance, $2^{2+4+10}\,3^{2+4}/4!\,3!=2^{12}\,3^{4}$ for $k=4$ and $2^{2+4+10+16}\,3^{2+4+6}\,5^2/5!\,4!=2^{26}\,3^{10}\,5$ for $k=5$).

    This results for $k\in\{2,3,4,5\}$ into the formula:
    \begin{multline}
        W_{d=2}^{k\in\{2,3,4,5\} }\Bigl(q,p,\frac{y}{1!},\frac{x^2}{2!},\frac{xy}{1!\,1!},...,\frac{y^k}{k!}\Bigr)=\frac{1}{k!\,(k-1)!}\prod_{\ell=2}^k (\deg p-\ell)(\deg q-\ell)\cdot \\
        \cdot \{(\deg p-1)m_1 - (\deg q-1)n_1\}\frac{pq}{x},
    \end{multline}
    which -- we show in Theorem~\ref{thm:golden-formula} -- is indeed the general form for any $k\geqslant 1$. (The empty product $\prod_{\ell=2}^1$ is interpreted as the unit: $\prod_\varnothing=1$.)
\end{ex}

\begin{theor}[Main formula]\label{thm:golden-formula}
    Take any dimension $d\geqslant 1$ and order $k\geqslant 1$; set $N={d+k\choose k}$. Let $p(\bx),q(\bx)\in \Bbbk_{k+1}[\bx]$ be arbitrary monomials: $p(\bx)=\bx^\bnvec = (x^1)^{n_1}\cdots (x^d)^{n_d}$ and $q=\bx^\bmvec = (x^1)^{m_1}\cdots (x^d)^{m_d}$. Calculating the complete generalised Wronskian (of differential order $k$ over dimension $d$) of all monomials up to degree $k$ as arguments, but with the constant polynomial $1=\bx^{\brvec_1}$ replaced by the monomial $q(\bx)$ and $x^1=\bx^{\brvec_2}$ by $p(\bx)$, gives
    \begin{multline}\label{eq:golden-formula}
        W_d^k\Bigl(\underbrace{q(\bx)}_{\widehat{1}},\underbrace{p(\bx)}_{\widehat{x^1}},\frac{\bx^{\brvec_2}}{\brvec_2!},...,\frac{\bx^{\brvec_N}}{\brvec_N!}\Bigr) = \frac{1}{k!(k-1)!} \prod_{\ell=2}^k (\deg p - \ell)
        (\deg q - \ell) \cdot \\
        \cdot \{(\deg p -1)m_1-(\deg q -1)n_1\} \cdot \frac{pq}{x^1},
    \end{multline}
    where, of course, if $x^1\nmid pq$, then the coefficient is zero by virtue of $m_1=n_1=0$.\footnote{Cf.\ Footnote~\ref{foot:curly-bracket} on p.~\pageref{foot:curly-bracket} for a simplification of notation of the curly bracket part of the main formula~\eqref{eq:golden-formula}.}
\end{theor}

    

\begin{problem}[vanishing mechanism in the Main formula~\eqref{thm:golden-formula}]\label{prob:vanishing-mechanism-golden-formula}
    Let $p(\bx)=(x^1)^{n_1}\cdots(x^d)^{n_d}$ and $q(\bx)=(x^1)^{m_1}\cdots(x^d)^{m_d}$ be arbitrary monomials over base dimension $d\geqslant 2$. Let the degree of the generalised Vandermonde determinant be $k\geqslant 1$. Consider all monomials up to degree $k$: $1,x^1,$$...,$$x^d,(x^1)^2,x^1x^2,$$...,$$(x^d)^k$; let their power $d$-tuples be $\brvec_1=(0,...,0)$, $\brvec_2=(1,0,...0)$, $\brvec_3=(0,1,0,...,0)$, ..., $\brvec_N=(0,...,0,k)$, and let $\bnvec=(n_1,...,n_d)$ and $\bmvec=(m_1,...,m_d)$ be the power $d$-tuples of $p$ and $q$, respectively. 
    
    Calculate their Wronskian as in Theorem~\ref{thm:golden-formula}; by formula \eqref{eq:golden-formula} it vanishes. \emph{What is the (algebraic) mechanism\footnote{Namely, are some (specific) rows or columns (clearly) linearly dependent as a result, or which terms cancel during the expansion?} for the complete generalised Vandermonde determinant of the power $d$-tuples, $\det \Van_d^k(\bm,\bn,\br_3,...,\br_N)=0$, to vanish} in each of the following cases:
    \begin{enumerate}[label=({\it\roman*})]
        \item\label{item:prob-vanishing:zero-degrees} if the degrees in $x^1$ of $p(\bx)$ and $q(\bx)$ are zero, $n_1=m_1=0$ (i.e.\ $x^1\nmid pq$);
        \item if the total degrees and the degrees in $x^1$ of $p(\bx)$ and $q(\bx)$ are equal, $\deg p=\deg q$ and $\deg_1 p = \deg_q q$?
    \end{enumerate}
    Refer to Theorem~\ref{thm:general-vandermonde} on p.~\ref{thm:general-vandermonde} about the complete generalised Vandermonde determinant appearing as the coefficient. For case~\ref{item:prob-vanishing:zero-degrees} refer to the deficiency condition~\ref{prop8-iii:prop-sufficient-vandermonde:deficiency} from Proposition~\ref{prop:sufficient-vandermonde-zero} on p.~\pageref{prop:sufficient-vandermonde-zero} for the Vandermonde determinant to vanish.
\end{problem}

\begin{rem}
    Direct expansion of the complete generalised Wronskian determinant in main formula~\ref{eq:golden-formula}, or, effectively, of the complete generalised Vandermonde determinant, is very difficult over base dimension $d\geqslant 2$ if the differential order is $k\geqslant 2$. The proof of Theorem~\ref{thm:golden-formula} by direct expansion in the case $k=1$ is given in Appendix~\ref{SecAppExpansion}.
\end{rem}

\begin{proof}[Proof (of Theorem~\ref{thm:golden-formula})]
    Recall that by Theorem~\ref{thm:general-vandermonde} the result of the bracket is the required monomial, whose coefficient is the complete generalised Vandermonde determinant.
    The proof strategy is to view this Vandermonde determinant as a polynomial in the degrees of monomials $p$ and $q$. A polynomial of degree (at most) $k$ can have (if they exist, i.e.\ all roots are in the base field) at most $k$ linear factors in the base variables. In this proof we find \emph{all} factors of the polynomial, hence we decompose the polynomial as a product of the linear factors. Of the $k$ factors, $k-1$ factors will be immediate and do not mix the degrees of the monomials $p$ and $q$, whereas the final $k$th factor (see the curly bracket in formula~\eqref{eq:golden-formula}) mixes the degrees and will require significant work to find. We then finally find the universal (independent of $p$ and $q$) constant in front of the linear factors.

    \medskip
    The case $k=1$ is proven in Appendix~\ref{SecAppExpansion}. Let $k\geq 2$, albeit the proof still holds for $k=1$. Calculate the resulting monomial by~\ref{thm:general-vandermonde}: the $pq/x$ term is obtained by adding the monomials' degrees, for which the coefficient is the generalised complete Vandermonde determinant. We view the result of the determinant as a polynomial in $\bm=(m_1,...,m_d)$; expansion over the first column yields a polynomial of (total) degree $k$. Therefore, this polynomial in $\bm$ has at most $k$ linear factors (in $\bm$).

    If the column of $\bm$'s is equal to another column, the determinant must be zero; this introduces a linear factor. As the set of columns of degrees at least two is complete, equality to another column of degree at least 2 is equivalent to $2\leq \deg(q)\leq k$; this introduces $k-1$ linear factors of the form $(\deg q -\ell)$ for $\ell=2,...,k$ hence their (l.c.m.)\ product~$\prod_{\ell=2}^k (\deg(q)-\ell)$. 

    \medskip
    Instead viewing the determinant as a polynomial in variables $\bn=(n_1,...,n_d)$ introduces, by the mechanism above applied to $p$, $k-1$ linear factors in $\bn$'s of the form $(\deg p -\ell)$ for $\ell\in\{2,...,k\}$. The factors involving the degrees in $q$ and the degrees in $p$ do not mix, hence we their product, $\prod_{\ell=2}^k (\deg p-\ell)\prod_{\ell=2}^k (\deg q-\ell)$, is a factor of the determinant. 
    
    Viewing the determinant as a polynomial either in $\bm$'s or $\bn$'s, it follows that only one other linear factor in $\bm$'s (or, equivalently, $\bn$'s) is possible; the factor may mix the individual components of $\bm$ and $\bn$. \emph{We now find this factor;} denote it by the function $f(\bn,\bm)$:
    \begin{equation}\label{eq:proof-of-golden-intermediate-factorisation}
        W_d^k\Bigl(q,p,\frac{\bx^{\brvec_2}}{\brvec_2!},...,\frac{\bx^{\brvec_N}}{\brvec_N!}\Bigr) = (\mathrm{const})\prod_{\ell=2}^k (\deg p-\ell)
        (\deg q -\ell)\cdot f(\bn,\bm) \cdot \frac{pq}{x^1},
    \end{equation}
    where the constant\footnote{This constant is not identically zero as there exists a choice of monomials, $q(\bx)=1$ and $p(\bx)=x$, which give a non-zero r.\!-h.s.\ in~\eqref{eq:proof-of-golden-intermediate-factorisation}, namely, the constant unit polynomial.} $(\mathrm{const})$ is independent of $\bn$ and $\bm$. The function $f(\bn,\bm)$ has degree one in $\bn$'s and also degree one in $\bm$'s.
    
    The factor $f(\bn,\bm)$ cannot equal a previously factor (i.e.\ of the form $(\deg(\cdot)-\ell)$) as the choice of monomials $p=q$, both of degree $k+1$, yields zero in~\eqref{eq:proof-of-golden-intermediate-factorisation}, which is not captured by the previous factors, which do not mix $\bn$'s and $\bm$'s; hence this root must be captured by the factor $f(\bn,\bm)$. The factor $f(\bn,\bm)$ of degree-one in each of the two $d$-tuples of arguments mixes $\bm$'s and $\bn$'s, hence written in full generality, is
    \begin{equation*}
        f(\bn,\bm)=\sum_{i=1}^d\lambda^{i,0}n_i+\sum_{j=1}^d \lambda^{0,j}m_j + \sum_{i,j} \lambda^{ij}n_im_j + C,
    \end{equation*}
    with all $\lambda$ factors and $C$ constants. Notation: $\bn=(n_1,...,n_d)$ and $\bm=(m_1,...,m_d)$. For convenience, the indices $i$ and $j$ (typically, $i<j$) will typically range from $1$ to $d$, whereas the indices $\mu$ and $\nu$ will incorporate also the $0$ value: $\mu,\nu\in \{0,1,...,d\}$.

    \smallskip
    As the determinant is anti-symmetric upon an exchange of $p$ and $q$, or, equivalently, the swap $n_i\rightleftarrows m_i$ for all $i\in\{1,...,d\}$, but all factors of $(\deg(\cdot)-\ell)$ are symmetric in $n_i\leftrightarrows m_i$, it follows that $f$ is skew w.r.t\ the two $d$-tuple arguments. Therefore $\lambda^{\mu\nu}=-\lambda^{\nu\mu}$ and $C=0$ for all indices $\mu,\nu\geq 0$; hence diagonal terms vanish: $\lambda^{ii}=0$. The function $f$ thus becomes
    \begin{equation*}
        f(\bn,\bm) = \sum_{i=1}^d \lambda^{i,0}(n_i-m_i) + \sum_{i<j} \lambda^{ij}(n_im_j - n_jm_i).
    \end{equation*}
    The choices of $d$-tuples $\bn=(0,...,0)$ and $\bm=(0,1,0,...,0)$ corresponding to the monomials $p=1$ and $q=y$, which in formula~\eqref{eq:proof-of-golden-intermediate-factorisation} give zero as the argument $y$ is repeated. But none of the terms under the product in~\eqref{eq:proof-of-golden-intermediate-factorisation} are zero, hence the vanishing condition must come from a root of the term $f(\bn,\bm)$:
    \begin{equation*}
        0=f(\bn=\mathbf{0},m_2=1)=\textstyle\sum_{i=1}^d (-)\lambda^{i,0} m_i = -\lambda^{2,0}\cdot 1,
    \end{equation*}
    which implies $\lambda^{2,0}=0$; similarly, all other zero-index constants vanish: $\lambda^{\mu,0}=0$ for indices $\mu\geq 2$ by choosing the $d$-tuple $\bm$ to be $m_{\mu}\neq 0$ and $m_{\nu}=0$ otherwise ($\nu\neq\mu$). This yields
    \begin{equation}\label{eq:proof:intermediate-form-for-mixed-term}
        f(\bn,\bm) = -\lambda^{0,1}(n_1-m_1)+\sum_{i<j}\lambda^{ij}(n_i m_j-n_j m_i)
    \end{equation}
    after applying $\lambda^{1,0}=-\lambda^{0,1}$.

    Let us now put the monomials $q=x$ and $p=y$ as arguments (hence $\bm=(1,0,...,0)$ and $\bn=(0,1,0,...,0)$); this choice yields a vanishing r.\!-h.s.\ in~\eqref{eq:proof-of-golden-intermediate-factorisation} as $y$ is a repeated argument. But again none of the factors in the products vanish, thus the vanishing condition is due to:
    \begin{equation*}
        0=f(n_2=1,m_1=1)= -\lambda^{0,1}(-)m_1+\lambda^{1,2}(-)n_2m_1 = \lambda^{0,1}-\lambda^{1,2},
    \end{equation*}
    which implies $\lambda^{1,2}=\lambda^{0,1}$; similarly, we obtain $\lambda^{1,\mu}=\lambda^{0,1}$ for every index $\mu\in\{2,...,d\}$ by putting as argument the monomial $p=x^{\mu}$. 
    Split the sum in the term~\eqref{eq:proof:intermediate-form-for-mixed-term} into two sums according to whether the summation variable $i$ is equal to one, that is, $i=1$ and $1<i<j$:
    \begin{equation*}
        \sum_{i<j}\lambda^{ij}(n_i m_j-n_j m_i) = \sum_{j=2}^d \lambda^{1,j} (n_1 m_j-n_j m_1) + \sum_{1<i<j} \lambda^{ij}(n_i m_j-n_j m_i),
    \end{equation*}
    which gives, after applying the substitution $\lambda^{1,j}=\lambda^{0,1}$ found above, the expression
    \begin{align}
        f(\bn,\bm)&=\lambda^{0,1}\Bigl\{-(n_1 - m_1) + \sum_{j=1}^d (n_1m_j - n_jm_1)\Bigr\} + \sum_{1<i<j} \lambda^{ij}(n_i m_j-n_j m_i)\notag{}\\
        &= \lambda^{0,1}\{(\deg p-1)m_1 - (\deg q -1)n_1 \} + \sum_{1<i<j} \lambda^{ij}(n_i m_j-n_j m_i),\label{eq:proof:golden-factor-curly-plus-sum}
    \end{align}
    where we used that $\deg(p)=n_1+\cdots+n_d$ and $\deg(q)=m_1+\dots+m_d$. Notice that the first summand is exactly the curly bracket in the main formula~\eqref{eq:golden-formula} up to an overall constant $\lambda^{0,1}$. We now show the sum comprising the second summand vanishes.

    \medskip
    Put as arguments in Eq.~\eqref{eq:golden-formula} the monomials $p(\bx)=y^{k+1}$ and $q(\bx)=y^{k+2}$; they are not divisible by the $x=x^1$ coordinate, namely, $n_1=0=m_1$; hence the $pq/x$ term is not a polynomial, and condition~\ref{prop8-iii:prop-sufficient-vandermonde:deficiency} from Proposition~\ref{prop:sufficient-vandermonde-zero} on p.~\pageref{prop:sufficient-vandermonde-zero} implies that the coefficient in front of $pq/x$ vanishes. But as $\deg p,\deg q>k$ none of the products in formula~\eqref{eq:proof-of-golden-intermediate-factorisation} are zero; this yields another root of the factor $f(\bn,\bm)$:
    \begin{equation*}
        0=f(\bn,\bm)=\sum_{i<i<j} \lambda^{ij} (n_i m_j - n_j m_i),
    \end{equation*}
    where the the curly bracket in~\eqref{eq:proof:golden-factor-curly-plus-sum} vanished by $n_1=0=m_1$. 
    We therefore obtain
    \begin{equation*}
        f(\bn,\bm)=\lambda^{0,1} \{(\deg p-1)m_1-(\deg q-1)n_1\},
    \end{equation*}
    where $\lambda^{0,1}$ is a universal constant independent of $\bn$ and $\bm$. We now proceed to find it.
    
    \smallskip
    Putting the arguments $p$ and $q$ to be the monomials they replaced, that is, $q=1$ and $p=x$, we know that the result of the Wronskian in Eq.~\eqref{eq:golden-formula} must yield unit (as product of units on the diagonal),
    hence obtaining overall that
    \begin{equation*}
        1=(\mathrm{const})\cdot \prod_{\ell=2}^k(1-\ell)\prod_{\ell=2}^k(-\ell)\cdot \{0+1\},
    \end{equation*}
    which gives $(\mathrm{const})=1/k!\,(k-1)!$.
    We thus recover the formula in~\eqref{eq:golden-formula}, which completes the proof.
\end{proof}


\begin{rem}\label{rem:reason-l2-formula}
    The reason $\ell$ starts in formula~\eqref{eq:golden-formula} from $\ell=2$ is that the first degree for which the arguments in~\eqref{eq:golden-formula} contain \emph{all} the monomials of that degree\footnote{
        This is equivalently referred to as a \emph{complete} degree; refer to the Definition~\ref{def:wronskian} on p.~\pageref{def:wronskian} of the \emph{complete} generalised Wronskian, which is complete in all degrees $0,1,...,k$ by the monomial--differential-operator correspondence (basically, $\partial^3/\partial x^2\partial y\leftrightarrows x^2y$).
    } 
    is 2. This is also why in formula~\eqref{eq:structure-equation} on p.~\pageref{eq:structure-equation} the second product runs from $\ell=r_j+1$ to $\ell=k$: the set of monomials of degree $r_j$ lacks the monomial $\bx^{\brvec_j}/\brvec_j!$, whereas all higher degrees are complete.
\end{rem}

Let us give a few computational examples illustrating this delayed start of $\ell$ at the first complete degree in the set of monomial arguments.

\begin{ex}[$d=2$, $k=2,3,4,5$, $\widehat{(x)^2}\mapsto p$]\label{ex:comp-d2-k2345-prepl-x2}
    Fix the base dimension $d=2$ and vary the total differential order $k\in\{2,3,4,5\}$ of the complete generalised Wronskian $W_d^k$ as in Example~\ref{ex:comp-examples-golden-formula} on p.~\pageref{ex:comp-examples-golden-formula}; adopt the notation $p(\bx)=x^{n_1}y^{n_2}$ and $q(\bx)=x^{m_1}y^{m_2}\in \Bbbk[\bx]$. As arguments, take the standard monomials $1,x,y,x^2,xy,...,xy^{k-1},y^k$, but replace $1$ by the monomial $q(\bx)$ and $(x)^2=xx$ by $p(\bx)$.\footnote{In contrast, $p$ replaced $x$ in Example~\ref{ex:comp-examples-golden-formula}.} Put, by definition,\footnote{\label{foot:curly-bracket}In contrast, in Example~\ref{ex:comp-examples-golden-formula} and the main formula~\eqref{eq:golden-formula} on pp.~\pageref{ex:comp-examples-golden-formula}~and~\pageref{eq:golden-formula}, resp., to simplify the notation we could have put
    \begin{equation}\label{eq:x-repl-gold-curly-between-pq}
        \{p, q\}_{\widehat{x}} := \{(\deg p-1)m_1 - (\deg q-1)n_1\}\frac{pq}{x}.
    \end{equation}
    There at least, unlike here, the formulae were manageable.
    }
    \begin{equation}\label{eq:xx-repl-curly-between-pq}
        \{p, q\}_{\widehat{xx}} := \{(\deg p-1)(\deg p -2)\,m_1(m_1-1) - (\deg q-1)(\deg q -2)\,n_1(n_1-1)\}\frac{pq}{xx},
    \end{equation}
    which we use to simplify the formulae below.
    Computations using Script~\ref{script:Wronskian-pq} in Appendix~\ref{SecAppPyScripts}, setting the optional argument \texttt{prepl=(x,x,)}, gave:\footnote{The raw outputs are given in Appendix~\ref{SecAppRawOutputs}.}
    \begin{subequations}\label{eqs:comp-p-repl-x2-highk-k2345}
    \begin{align}
        W_{d=2}^{k=2}(q,x,y,p,xy,y^2)&=-1\cdot \{p, q\}_{\widehat{xx}}, \label{eq:comp-p-repl-x2-highk-k2} \\
        W_{d=2}^{k=3}(q,x,y,p,xy,...,xy^2,y^3)&=-2^4\,3\cdot (\deg p-3)(\deg q-3)\cdot \{p, q\}_{\widehat{xx}} \label{eq:comp-p-repl-x2-highk-k3} \\
        W_{d=2}^{k=4}(q,x,y,p,xy,...,xy^3,y^4)&=-2^{11}\,3^5 \prod_{\ell=3,4}(\deg p-\ell)(\deg q-\ell)\cdot \{p, q\}_{\widehat{xx}}, \label{eq:comp-p-repl-x2-highk-k4} \\
        W_{d=2}^{k=5}(q,x,y,p,xy,...,xy^4,y^5)&=-2^{27}\,3^{10}\,5 \prod_{\ell=3,4,5}(\deg p-\ell)(\deg q-\ell)\cdot \{p, q\}_{\widehat{xx}} \label{eq:comp-p-repl-x2-highk-k5}
    \end{align}
    \end{subequations}
    Notice how the factor in the curly brackets $\{p,q\}_{\widehat{xx}}$ is the same in all equations~\eqref{eqs:comp-p-repl-x2-highk-k2345}, whereas the product ranges from $\ell=3$ to $\ell=k$; the same mechanism as explained in Remark~\ref{rem:reason-l2-formula}.

    We contrast the coefficient factor (which mixes the degrees of $p$ and $q$) in the curly brackets in formulas~\eqref{eqs:comp-p-repl-x2-highk-k2345} above against formulas~\eqref{eqs:comp-gf-highk-k2345} in Example~\ref{ex:comp-examples-golden-formula} on p.~\pageref{ex:comp-examples-golden-formula}:
    \begin{align*}
        \{p, q\}_{\widehat{x}}  &:= \{(\deg p-1)m_1 - (\deg q-1)n_1\}\frac{pq}{x},\qquad\text{(this is from formulas~\eqref{eqs:comp-gf-highk-k2345})}\\
        \{p, q\}_{\widehat{xx}} &:= \{(\deg p-1)(\deg p -2)\,m_1(m_1-1) - (\deg q-1)(\deg q -2)\,n_1(n_1-1)\}\frac{pq}{xx}.
    \end{align*}
    It is evident that the replacement $\widehat{x}\mapsto p$ versus $\widehat{xx}\mapsto p$ turns one linear factor in the degrees $n_1,n_2$ of $p$ into $(\deg p-2)$ and $(n_1-1)$ inside the curly bracket; recall that the sum of the factors' total degrees in $n_1$ and $n_2$ must be $k$ (see commentary within the proofs of Theorems~\ref{thm:structure-loneliness} and~\ref{thm:golden-formula}). 
    
    This suggests that replacing a higher-degree term by $p(\bx)$  (for instance, $\widehat{(x)^3},\widehat{(x)^4},...,\widehat{(x)^k}$, for which one can easily \emph{conjecture} the factor in the equivalent curly brackets,
    \begin{equation}\label{eq:xR-repl-gold-curly-between-pq}
        \{p,q\}_{\widehat{(x)^r}} \stackrel{?}{=} \Biggl\{ \prod_{\ell=1}^{r}(\deg p-\ell) \prod_{\ell=0}^{r-1}(m_i-\ell) - \prod_{\ell=1}^{r}(\deg q-\ell) \prod_{\ell=0}^{r-1}(n_i-\ell) \Biggr\} \frac{pq}{(x)^r},
    \end{equation}
    or $\widehat{xy}$, $\widehat{x^2y}$, $\widehat{x^2 y^2}$ and so on\footnote{We do not yet know what will be there.}) will result in fewer terms of the form $(\deg p - \ell)$, and the terms within the curly bracket will be of more complicated form; see Problem~\ref{open:factorisations-vandermonde} and the accompanying Commentary.
\end{ex}

\begin{open}
    The structure constants for our finite-dimensional $N$-ary Lie algebras can be written as a differential operator acting on the $(k+1)$th degree argument -- see Eq.~\eqref{eq:structure-equation-as-differential} on p.~\pageref{eq:structure-equation-as-differential}. Can the skew-symmetric operators in the curly brackets -- Eqs.~\eqref{eq:x-repl-gold-curly-between-pq}, \eqref{eq:xx-repl-curly-between-pq}, and~\eqref{eq:xR-repl-gold-curly-between-pq} -- be written as differential operators acting on the (poly- or) monomials $p$ and $q$?
\end{open}

\begin{rem}[Vandermonde factorisations take the simplest form]
    We have so far explained the nature of every factor in formula~\eqref{eq:structure-equation} from Theorem~\ref{thm:structure-loneliness} and in formula~\eqref{eq:golden-formula} in Theorem~\ref{thm:golden-formula}; this is the fullest and simplest possible factorisation of the generalised Vandermonde determinants which we can provide (cf.\ formula~\eqref{eq:vandermonde-ordinary} in Theorem--Definition~\ref{thmdef:ordinary-vandermonde} when the base dimension is $d=1$). 
    Now, for $d\geqslant 2$, all the factors which we actually present have appeared either due to:
    \begin{enumerate}[label=({\it\roman*})]
        \item the monomial $p$ coinciding with a monomial of degree $\ell$ from the range $\{r+1,...,k\}$, where $r$ is the degree of the monomial, which $p$ replaced as argument;
        \item the monomial $p$ coinciding with a monomial of degree less or equal to $r$;
        \item (specific only to Eq.~\eqref{eq:golden-formula}) the two monomials $p$ and $q$ are interrelated in such a way that the curly bracket in~\eqref{eq:golden-formula} vanishes.
    \end{enumerate}
    If three or more of the arguments were arbitrary monomials, we should expect a totally-skew factor mixing the degrees of the three arbitrary monomials. It is an open problem.
\end{rem}

\begin{open}[factorisations of the generalised Vandermonde determinant]\label{open:factorisations-vandermonde}
    Put the base dimension $d\geqslant 2$.
    Suppose that the number of standard monomials (from the set $1$, $x^1$, ..., $x^d$, $(x^1)^2$, $x^1x^2$, ..., $(x^d)^k$) replaced by the new (arbitrary) monomials $p(\bx)$ and $q(\bx)$ is increased to three or more \emph{monomials} (i.e.\ $p_1=p$, $p_2=q$, $p_3$, $p_4$ and so on). What is (or are) the simplest factorisation(s) of the complete generalised Vandermonde determinant of the arguments' power $d$-tuples (relative to the complete generalised Wronskian $W_{d\geqslant 1}^{k\geqslant 1}$)?

    Note that the answer depends on the choice of standard monomials to be replaced by the new ones.
\end{open}

\begin{comm}[to Open problem~\ref{open:factorisations-vandermonde}]
    In this paper, specifically Theorems~\ref{thm:structure-loneliness} and~\ref{thm:golden-formula}, we explored the cases when:
    \begin{enumerate}[label=({\it\roman*})]
        \item any single monomial from the standard set is replaced;
        \item two monomials are replaced, specifically $1$ and $x^1$ (equivalently by symmetry, $1$ and $x^i$ for any $i\in\{1,...,d\}$).
    \end{enumerate}
    Replacing other pairs of monomials from the standard set of degrees not exceeding $k\geqslant 1$ over base dimension $d\geqslant 1$ will result in other formulas and other factorisations of the complete generalised Vandermonde determinants: see Example~\ref{ex:comp-d2-k2345-prepl-x2}.

    F.~Brown in~\cite{FBrownVandermonde2024} expresses the generalised -- \textbf{NB!} `generalised' in a way still more general than ours!\footnote{Our \emph{complete} generalised Vandermonde determinants correspond to those in \S2.4 of \cite{FBrownVandermonde2024}. The further generalisation drops the `complete' from the complete generalised Vandermonde determinant; the Vandermonde determinant is now with respect to an \emph{in}complete Wronskian (cf. \cite{AVK25YerQTS13WrosnkBJP} for definitions and a discussion).} -- Vandermonde determinant (of arbitrary monomial arguments) as sums of smaller-size generalised Vandermonde determinants. In the sums the terms do not visibly share any common factor.
\end{comm}

We now state our second main theorem. For relevant definitions and details see Definitions~\ref{def:consistent-lonely-chubby-tall-lanky} and accompanying Lemma~\ref{lem:chubby-tall-lanky} on pp.~\pageref{def:consistent-lonely-chubby-tall-lanky} and~\pageref{lem:chubby-tall-lanky}, resp., about the different types of $N$\nobreakdash-ary polynomial SH-Lie algebras, and why every \emph{consistent} (cf.~\ref{lem:consistency}) algebra is either trivial, lonely, chubby, or lanky; see also Footnote~\ref{foot:consideration-of-pure-monomials} on p.~\pageref{foot:consideration-of-pure-monomials} for the reasoning why it is sufficient to consider algebras spanned by pure monomials. In the first main Theorem~\ref{thm:main-1-existence} we showed that every lonely algebra is finite-dimensional of dimension $N+1$; trivial algebras are trivially of dimension at most $N$. 
In the proof of the theorem we extensively use the main formula~\eqref{eq:golden-formula}.

\begin{theor}[main theorem-\textbf{2}: \emph{no other} consistent finite-dimensional algebras]\label{thm:main-2:non-existence}
    Take any base dimension $d\geqslant 1$ and differential order $k\geqslant 1$; set $N=\binom{d+k}{k}$. Put the \emph{\underline{com}p\underline{lete}} generalised Wronskian determinant $W_d^k
    $ as the $N$\nobreakdash-ary bracket. Suppose that an $N$\nobreakdash-ary polynomial SH-Lie algebra $\EuA$ is \emph{consistent}, but \emph{\underline{not lonel}y}; i.e.\ $\EuA$ is either of the two options:
    \begin{enumerate}[label=({\it\roman*})]
        \item $\EuA$ is chubby, that is, $\EuA\supseteq \Bbbk_k[\bx]\oplus\langle p(\bx),q(\bx)\rangle$, where $p\not\sim q$ are (poly- or) monomials of degrees $\deg(p)=k+1=\deg(q)$;
        \item $\EuA$ is lanky, that is, $\EuA$ contains the monomial $(x^i)^{k+\ell}$, where $\ell\geqslant 2$, and thus the tower of monomials in a single coordinate $x^i$: $\EuA\supseteq \Bbbk_k[\bx]\oplus \langle (x^i)^{k+1},...,(x^i)^{k+\ell}\rangle$.
    \end{enumerate}
    Then $\EuA$ \emph{\underline{cannot}} be finite-dimensional, so $\dim_\Bbbk (\EuA)=+\infty$.
\end{theor}

\begin{proof}
    We show that every (\textit{i}) chubby, (\textit{ii}) lanky algebra is infinite-dimensional by producing a sequence of monomials, constructed using the Wronskian determinant, which diverge in degree going to infinity. That is, the algebra contains this infinite sequence of independent monomials, hence is infinite dimensional. All cases proceed via induction.

    \medskip
    \noindent$\bullet$\quad\textbf{Case: chubby }(\textit{i})\textbf{.}\quad Let $p\not\sim q$ be monomials of degree $k+1$.; put $p(\bx)=\bx^\bnvec = (x^1)^{n_1}\cdots (x^d)^{n_d}$ and $q(\bx)=\bx^\bmvec = (x^1)^{m_1}\cdots (x^d)^{m_d}$.  As the monomials are not proportional to each other, there must exist a coordinate, in which their powers differ: $n_i\neq m_i$. Without loss, put $i=1$ (if needed, by rearrangement of the order of coordinates). We now separate the possible cases of $n_1$ and $m_1$: firstly, $m_1>1$ splits into $n_1=0$, $m_1>n_1\geqslant 1$, or $n_1>m_1$; in the last case, swap the monomials $p$ and $q$, such that $m_1$ is always bigger than $n_1$. Secondly, if $m_1=1$, we have $n_1=0$ or $n_1>1$; again, in the last case swap $p\leftrightarrows q$, obtaining $m_1>n_1$. Finally, if $m_1=0$, then $n_1>0$; swapping $p\leftrightarrows q$ we again obtain a previous case. Hence, without loss, an exhaustive list of cases is:
    \begin{align*}
        \text{(\textit{a}) } m_1>1 &\text{ and } n_1=0; &
        \text{(\textit{b}) } m_1>n_1&\geqslant 1; &
        \text{(\textit{c}) } m_1=1 &\text{ and } n_1=0.
    \end{align*}
    We tackle each separately. Denote $x:=x^1$. In all cases put $p_0=p$, $p_1=pq/x$, $p_2=p_1q/x$, ...,  $p_n=p_{n-1}q/x$, ...; note that the division does produce a monomial as $x|q$ in all cases ($m_1\geqslant 1$). As the total degree $\deg(p)=k+1$ and $\deg(p_1)=2(k+1)-1=2k+1$, we inductively find $\deg(p_n)=(n+1)k+1$. Indeed, as $n\to +\infty$, it follows that $\deg(p_n)\to +\infty$. Our goal is to produce each $p_n$ by using arguments from $\EuA$ and previous $p_0,p_1,...,p_{n-1}$.

    Firstly, note that whenever the degrees of the two free arguments in the Wronskian as in the main formula~\eqref{eq:golden-formula} are $\deg(p),\deg(q)>k$, the products of $(\deg(\cdot)-\ell)$ over $\ell\in\{2,3,...,k\}$ do not vanish. Put $c:=1/[k!\,(k-1)!]\cdot \prod_{\ell=2}^k (\deg(p)-\ell) (\deg(q)-\ell)\neq 0$ for every choice of $p$ and $q$ of degrees $\deg(p),\deg(q)>k$.

    \textbf{Subcase }(\textit{a})\textbf{.}\quad
    In this case $\deg_x(q)=m_1$ and, inductively, $\deg_x (p_n)=n(m_1-1)$. Then
    \begin{align*}
        W_d^k\Bigl(q,p,x^2,...,\frac{(x^d)^k}{k!}\Bigr) &= c\cdot \{(k+1-1)m_1-0\}\frac{pq}{x} = c\cdot k m_1 \frac{pq}{x},
        \intertext{where we note that the coefficient does not vanish, and}
        W_d^k\Bigl(q,p_n,x^2,...,\frac{(x^d)^k}{k!}\Bigr) &= c\cdot \{((n+1)k+1-1)m_1-(k+1-1)\,n(m_1-1)\}\frac{p_nq}{x} \\
        &= c\cdot k(m_1+n)\frac{p_nq}{x},
    \end{align*}
    where again the coefficient does not vanish. By induction that we thus obtain that each monomial $p_n$ belongs to the algebra $\langle p_0,p_1,p_2,...,p_n,...\rangle\subseteq \EuA$. As all $p_n$'s are linearly independent, we find that $\dim_\Bbbk \EuA = +\infty$.
    
    \textbf{Subcase }(\textit{b})\textbf{.}\quad
    Now $\deg_x(pq/x)=m_1+n_1-1\geqslant 2$ and, inductively, we see that the degree in $x=x^1$ of the $n$th monomial $p_n(\bx)$ is $\deg_x(p_n)=n(m_1-1)+n_1$. Then
    \begin{align*}
        W_d^k\Bigl(q,p,x^2,...,\frac{(x^d)^k}{k!}\Bigr) &= c\cdot \{(k+1-1)m_1-(k+1-1)n_1\}\frac{pq}{x} = c\cdot(m_1-n_1)\frac{pq}{x},
        \intertext{where the coefficient is not zero, and}
        W_d^k\Bigl(q,p_n,x^2,...,\frac{(x^d)^k}{k!}\Bigr) &= c\cdot \{((n+1)k+1-1)m_1-(k+1-1)(n(m_1-1)+n_1)\}\frac{p_n q}{x}\\
        &= c\cdot k(m_1-n_1+n)\frac{p_nq}{x},
    \end{align*}
    where again the coefficient does not vanish. We see that $\dim_\Bbbk \EuA = +\infty$.
    
    \textbf{Subcase }(\textit{c})\textbf{.}\quad
    In this case $\deg_x(p_n)=0$. Then
    \begin{align*}
        W_d^k\Bigl(q,p,x^2,...,\frac{(x^d)^k}{k!}\Bigr) &= c\cdot \{(k+1-1)m_1-0\}\frac{pq}{x}=c\cdot km_1 \frac{pq}{x},
        \intertext{where the coefficient does not vanish, and}
        W_d^k\Bigl(q,p_n,x^2,...,\frac{(x^d)^k}{k!}\Bigr) &= c\cdot \{((n+1)k+1-1)m_1-0\}\frac{p_nq}{x} = c\cdot (n+1)k\cdot\frac{p_nq}{x},
    \end{align*}
    where the coefficient is never zero. Again thus $\dim_\Bbbk \EuA=+\infty$.

    \medskip
    \noindent$\bullet$\quad\textbf{Case: lanky }(\textit{ii})\textbf{.}\quad
    Put, without loss, $q(\bx)=(x^1)^{k+1}$ and $p(\bx)=(x^1)^{k+\ell}$, where the degree $\ell\geqslant 2$. Then $\deg(pq/x)=k+1+k+\ell-1=2k+\ell$ and, inductively, as the monomials involve only one coordinate: $\deg p_n = \deg_x p_n = (n+1)k+\ell$. Then
    \begin{align*}
        W_d^k\Bigl(q,p,x^2,...,\frac{(x^d)^k}{k!}\Bigr) &= c \{(k+\ell -1)(k+1) - (k+1-1)(k+\ell)\}\frac{pq}{x} = c (\ell-1)\frac{pq}{x},
        \intertext{where $\ell-1>0$, hence the coefficient does not vanish; and}
        W_d^k\Bigl(q,p_n,x^2,...,\frac{(x^d)^k}{k!}\Bigr) &= c \{((n+1)k+\ell-1)(k+1) - (k+1-1)((n+1)k+\ell)\} \frac{p_nq}{x} \\
        &= c (nk+\ell -1)\frac{p_nq}{x},
    \end{align*}
    where the coefficient is never zero. We have found that $\langle p_0,p_1,...,p_n,...\rangle\subseteq \EuA$, so the dimension of the polynomial SH-Lie algebra $\EuA$ is $\dim_\Bbbk \EuA=+\infty$.

    \medskip
    All possible cases have been exhausted, whereby in each case $\dim_\Bbbk \EuA=+\infty$. The proof is thus complete.
\end{proof}

\begin{rem}\label{rem:only-way-to-keep-findim}
    From the proof of the second main Theorem~\ref{thm:main-2:non-existence} and Theorems~\ref{thm:general-vandermonde} and~\ref{thm:golden-formula} we see that to keep the algebra $\EuA$ finite dimensional, we must have a bound on the polynomial degrees (and their sum) of all the arguments in the bracket. To describe that bound, we recall that the sum of degrees minus the shift (see Definition~\ref{def:wronsk-shift-and-monomial-degree} and Corollary~\ref{corr:degree-shift} on p.~\pageref{corr:degree-shift}) cannot exceed the highest degree of a monomial in the set. And if the bound is exceeded, the only way for the algebra $\EuA$ to remain finite-dimensional is the vanishing of the coefficient of the resulting monomial.

    Our computational experiments (some 20--30 cases done by the first author) show that accidentally hitting the zero coefficient is very unlikely (unless it is an instance of repeated arguments or an instance of deficiency: the shift down of the degree is greater than the sum of available degrees; see Proposition~\ref{prop:sufficient-vandermonde-zero}\,\ref{prop8-iii:prop-sufficient-vandermonde:deficiency}). Moreover, we observe that three always is a set of arguments, reached through iterated commutation, such that the coefficient is not zero and the resulting degree exceeds that of any argument. (In fact, most choices work.)
\end{rem}

\begin{open}[\emph{in}-consistent finite-dimensional algebras]\label{open:existence-if-dropped-consistency}
    Relax the consistency assumption in the main Theorems~\ref{thm:main-1-existence} and~\ref{thm:main-2:non-existence}: suppose that the set of generators of the $N$\nobreakdash-ary polynomial SH-Lie algebra $\EuA$ has a gap somewhere in polynomial degrees not exceeding $k$. But still let us require that the algebra is \emph{perfect}: $\Span_\Bbbk W_d^k(\EuA,...,\EuA)=\EuA$.\footnote{By requesting that $\EuA$ is perfect we exclude trivial cases such as $\EuA=\Span_\Bbbk\langle 1,y,y^2,...,y^{100}\rangle$ over base dimension $d=2$, $\BBR^2\ni (x,y)$, and differential order $k=10$: all Wronskian brackets by $W_{d=2}^{k=10}$ are trivially zero, so the algebra $\EuA$, generated by these powers of $y$, is closed and finite-dimensional. We also exclude the trivial cases when $\dim(\EuA)=N$ and the result of the bracket of all $N$ generators again belongs to the algebra~$\EuA$.}

    Can $\EuA$ be finite-dimensional? (So far, no such example is known; the task is to either construct one for some choice of values $d$, $k$, or prove that no such example exists.)
\end{open}

\begin{conjecture}[to Open problem~\ref{open:existence-if-dropped-consistency}: `no']\label{conj:non-existence-of-inconsistent}
    We expect that the answer to the question in Open problem~\ref{open:existence-if-dropped-consistency} is \emph{no}: over dimension $d\geqslant 1$, \emph{in-consistent} and \emph{perfect} polynomial strongly homotopy Lie algebras with complete generalised Wronskian brackets may only be infinite-dimensional.
\end{conjecture}

\subsection*{Conclusion}
\phantom{.}

\medskip
\noindent We have found \emph{\underline{all} finite-dimensional} $N$\nobreakdash-ary \emph{consistent} polynomial SH-Lie algebras $\EuA$: that is, $\Bbbk_k[x^1,...,x^d]\subseteq \EuA \subsetneq \Bbbk[x^1,...,x^d]$ with the \emph{complete} generalised Wronskian determinant -- of differential order $k$ over base dimension $d$ -- as the $N$\nobreakdash-ary bracket, $N=\tbinom{d+k}{k}$. Each such algebra has dimension $(\leqslant)$$N$ or $N+1$: trivial, $\EuA\subseteq \Bbbk_k[\bx]$,  or lonely, $\EuA=\Bbbk_k[\bx]\oplus\langle p(\bx)\rangle$, algebras, respectively, where $p(\bx)\in \Bbbk_{k+1}[\bx]$ is any homogeneous polynomial of degree $k+1$.

\subsection*{Discussion}
\phantom{.}

\medskip
\noindent We leave to future research the investigation of very general properties of $N$\nobreakdash-ary SH-Lie polynomial algebras. Fix the base dimension $d\geqslant 1$ and the differential order $k\geqslant 1$ of the complete generalised Wronskian $W_d^k$; set $N={d+k\choose k}$. Consider an SH-Lie polynomial algebra $\EuA\subseteq \Bbbk[\bx]$ with $W_d^k$ as the $N$\nobreakdash-ary bracket.

Let us make some preliminary steps to phrase Definition~\ref{def:maximal-degree-sum} below. Write each polynomial in $\EuA$ as a linear combination of monomials. Firstly, suppose all such monomials do belong to the algebra. Define the \emph{maximal degree sum} in a coordinate $x^i$ by $N$ linearly independent monomial arguments, $D_i^{(N)}(\EuA)$, to be:
\begin{equation}\label{eq:max-achievable-sum-monomials}
    D_i^{(N)}(\EuA) = \mathop{\max_{a_1,...,a_N\in \EuA}}_{\text{lin.\,indep.\,monoms}} \sum_{j=1}^N \deg_i(a_j).
\end{equation}
(See Definition~\ref{def:wronsk-shift-and-monomial-degree} on p.~\pageref{def:wronsk-shift-and-monomial-degree} for the definition of degree in $x^i$.)

We extend the definition of the maximal degree sum to the case where polynomials in $\EuA$ are not always decomposed into monomials again belonging to $\EuA$. Here we mean mainly the $(k+1)$th degree homogeneous polynomial $p$, such as in a lonely algebra (see Definition~\ref{def:consistent-lonely-chubby-tall-lanky} for `lonely' on p.~\pageref{def:consistent-lonely-chubby-tall-lanky}). Take polynomials $p_j\in \EuA$ from the algebra and decompose each into monomials (not necessarily belonging to $\EuA$): $p_j=\sum_{\alpha_j} a_j^{\alpha_j}$, where each $a_j^{\alpha_j}$ is a monomial. Expanding the Wronskian $W_d^k(p_1,...,p_N)$ of the arguments $p_1,...,p_N$, the result will be a linear combination of Wronskians of the monomial arguments $\sum_{\alpha_1,...,\alpha_N} W_d^k(a_1^{\alpha_1},...,a_N^{\alpha_N})$. The \emph{maximal degree sum} in $x^i$ of $\EuA$ is defined by taking the maximum over the constituent monomials $a_j^{\alpha_j}$.

\begin{define}[maximal degree sum]\label{def:maximal-degree-sum}
    Fix the base dimension $d\geqslant 1$ and the differential order $k\geqslant 1$ of the complete generalised Wronskian determinant $W_d^k$; set $N={d+k\choose k}$. Let $\EuA\subseteq \Bbbk[\bx]$ be an $N$\nobreakdash-ary polynomial SH-Lie algebra. 
    Then the \emph{maximal degree sum}, $D_i^{(N)}$, in coordinate $x^i$  over $N$ (linearly independent) polynomials in $\EuA$ is:
    \begin{equation}\label{eq:max-degree-sum-polynomials}
        D_i^{(N)}(\EuA) = 
        \max_{P} \max_{A(P)}
        \sum_{j=1}^N \deg_i(a_j^{\alpha_j}),
    \end{equation}
    where $P$ consists of all subsets of polynomials $\{p_1,...,p_N\}\subseteq \EuA$ such that $p_1,...,p_N$ are linearly independent, and, after decomposing each polynomial into constituent monomials $p_j(\bx)=\sum_{\alpha_j} a_j^{\alpha_j}$ not necessarily belonging to $\EuA$, put $A(P)$ to be all $\{a_1^{\alpha_1},...,a_N^{\alpha_N}\}$ such that all the monomials are linearly independent. The function $D_i^{(N)}$ takes values in $\BBNo\cup \{+\infty\}$.
\end{define}

The properties of a polynomial SH-Lie algebra $\EuA$ are then characterised by the value of $D_i^{(N)}(\EuA)$ compared with the index-shift down by the complete generalised Wronskian (see Definition~\ref{def:wronsk-shift-and-monomial-degree}).

\begin{define}\label{def:defincient-exact-abundant-promising}
    Fix the base dimension $d\geqslant 1$ and the differential order $k\geqslant 1$ of the complete generalised Wronskian determinant $W_d^k$; set $N={d+k\choose d}$. Recall from Definition~\ref{def:wronsk-shift-and-monomial-degree} that the degree shift in coordinate $x^i$ by $W_d^k$ is $kN/(d+1)$. An $N$\nobreakdash-ary polynomial SH-Lie algebra $\EuA\subseteq \Bbbk[\bx]$ with the $N$\nobreakdash-ary bracket given by $W_d^k$ is called:
    \begin{enumerate}[label=({\it\roman*})]
        \item \emph{deficient} in coordinate $x^i$ if the maximal degree sum $D_i^{(N)}$ is strictly smaller than the index shift, that is, $D_i^{(N)}(\EuA)<kN/(d+1)$;
        \item \emph{exact} in coordinate $x^i$ if $D_i^{(N)}(\EuA)=kN/(d+1)$;
        \item \emph{abundant} in coordinate $x^i$ if $D_i^{(N)}(\EuA)>kN/(d+1)$; and
        \item \emph{promising} if the algebra $\EuA$ is abundant in every coordinate $x^1,...,x^d$.
    \end{enumerate}
    (The definition of promising is na\"ive and will likely be revised in the future research.)
\end{define}

This choice of terminology is revealed by the following lemma. Recall the deficiency mechanism for the vanishing of the complete generalised Vandermonde determinant of power $d$-tuples of (genuine) monomials: it is~\ref{prop8-iii:prop-sufficient-vandermonde:deficiency} from Proposition~\ref{prop:sufficient-vandermonde-zero} on p.~\pageref{prop:sufficient-vandermonde-zero}. 
The two statements follow by applying the degree shift by the Wronskian.

\begin{lem}[deficient vs exact algebra]\label{lem:deficient-vs-exact}
    Fix any base dimension $d\geqslant 1$ and differential order $k\geqslant 1$ of the complete generalised Wronskian determinant $W_d^k$; set $N={d+k\choose k}$. If an $N$\nobreakdash-ary polynomial SH-Lie algebra $\EuA\subseteq \Bbbk[\bx]$, with $W_d^k$ as the $N$\nobreakdash-ary bracket, is \emph{deficient} in at least one coordinate $x^i$, then \emph{every} choice of $N$ arguments from $\EuA$ for the Wronskian will yield zero. (Thus $\EuA$ is `trivially' closed under $W_d^k$.)

    \noindent$\bullet$\quad If $\EuA\subseteq \Bbbk[\bx]$ is exact in coordinate $x^i$ for some $i\in\{1,...,d\}$, then \emph{every} choice of arguments will result in a polynomial of degree zero in coordinate $x^i$. (Thus $\EuA$ cannot be perfect.)
\end{lem}

\begin{rem}\label{rem:number-exact-coods-asymptotics}
    The number of exact coordinates and the minimal number of arguments necessary to attain the maximum in~\eqref{eq:max-degree-sum-polynomials} determine the (maximal bounds for) the asymptotic behaviour of iterating the bracket of arguments from $\EuA$. Indeed, the maximal degree of all polynomial $p(\bx)\in \EuA$ in an exact coordinate is bounded by $kN/(d+1)$.

    It takes great effort to make such notions precise; relevant results will be announced separately. (Cf. \cite{GelfandKirillov} for the notion of an asymptotic dimension of a (SH-)Lie algebra under iterating the bracket.)
\end{rem}

\subsubsection*{Acknowledgements}
The first author thanks the Center for Information Technology of the University of Groningen for access to the High Performance Computing cluster, H\'abr\'ok. 
The second author thanks the organisers of the international workshop
`Combinatorics \& arithmetic for Physics: special days' -- CAP25 (19--21 November 2025 at the IH\'ES in Bures\/-\/sur\/-\/Yvette, France), where this line of research was presented, for a warm atmosphere during the meeting; the second author thanks M.\,Kontsevich, V.\,Retakh, and O.\,Gaber for helpful discussions and advice. 
The second author thanks also the organisers of the Prague Mathematical Physics seminar, where this work was presented on 7 May 2026, for partial financial support and hospitality.
This work has been partially supported by the Bernoulli Institute (Groningen, NL) via project~110135.

\newpage%
\appendix
\section{Re-proof of the main formula~\eqref{eq:golden-formula} by expansion: the lowest differential order case $k=1$}\label{SecAppExpansion}

\noindent Put $p(\bx)=\bx^\bnvec=(x^1)^{n_1}\cdots (x^d)^{n_d}$ and $q(\bx)=\bx^\bmvec=(x^1)^{m_1}\cdots (x^d)^{m_d}$ over the coordinates $\bx=(x^1,...,x^d)$. Let us compute $W_{d\geqslant 2}^{k=1}(p,x^2,x^3,...,x^d,q)$. By Theorem~\ref{thm:general-vandermonde} the result is $pq/x$ multiplied by the complete generalised Vandermonde determinant:
\begin{align*}
    \det \begin{vmatrix}
        1  &1&1&1&\cdots&1&1&1\\
        n_1&0&0&0& &&&m_1\\
        n_2&1&0&0& &&&m_2\\
        n_3& &1&0& &&&m_3\\
        \vdots& &&& &&&\vdots\\
        n_{d-1}& &&& &1&0&m_{d-1}\\
        n_d& &&& &&1&m_1
    \end{vmatrix}.
\end{align*}
The expansion by the first column, split into three parts, yields
\begin{multline}\label{eq:three-part-vandermonde}
    (-)^{d-1}m_1 - n_1\det \begin{vmatrix}
        1&1&\cdots&1&1\\
        1& &      & &m_2\\
         &1&      & &m_3\\
         & &\ddots& &\vdots\\
         & &      &1&m_d
    \end{vmatrix}
    +\\+ \sum_{j=2}^d (-)^j n_j \det \begin{vmatrix}
        1&1&1&\cdots &1&1&1&\cdots &1&1\\
        0& & &       & & & &       & &m_1\\
        1& & &       & & & &       & &m_2\\
         &1& &       & & & &       & &m_3\\
         & &1&       & & & &       & &m_4\\
         & & &\ddots & & & &       & &\vdots\\
         & & &       &1&0& &       & &m_{j-1}\\
         & & &       & &0&1&       & &m_{j+1}\\
         & & &       & &\vspace{-2.5ex}\underbrace{{ }}_{j\text{th}} & &\ddots & &\vdots\\
         & & &       & & & &       &1&m_d
    \end{vmatrix},
\end{multline}
where the determinant in the second term simplifies to
\begin{equation*}
    (-)^{d-2}m_2 - (-)^{d-3}m_3 + (-)^{d-4}m_4 -\cdots +(-)^{d}(-)^{d-d}(m_d-1) = (-)^{d-2}(m_2+\cdots+m_d-1)
\end{equation*}
by expansion over the first column $(1,1,0,...)$, repeated inductively until a $2\times 2$ determinant is reached.

In the third determinant in~\eqref{eq:three-part-vandermonde} cycle the first $j$ columns (i.e.\ until the $(1,...,0,0,...)^\top$ column), giving a sign $(-)^{j-1}$, and then cycle all rows except the first, giving a sign of $(-)^{d-2}$; the determinant becomes diagonal and equals $m_1$. When multiplied by $(-)^jn_j$, it yields the $j$th term $-(-)^{d-1}n_j m_1$ in the sum in~\eqref{eq:three-part-vandermonde}; the sum hence becomes $-(-)^{d-1}m_1(n_2+\cdots n_d)$.

Combining the results above turns~\eqref{eq:three-part-vandermonde} into
\begin{multline*}
    (-)^{d-1}\Bigl\{m_1 + n_1(m_2+\cdots+m_d-1)-m_1(n_2+\cdots+n_d) \Bigr\} =\\
    = (-)^{d-1} \{ (\deg q -1)n_1 - (\deg p - 1)m_1 \}.
\end{multline*}
This then gives, upon moving the monomial $q(\bx)$ from the last to the first argument,
\begin{align*}
    W_{d\geqslant 2}^{k=1}(q,p,x^2,x^3,...,x^d) &=(-)^d (-)^{d-1} \{ (\deg q -1)n_1 - (\deg p - 1)m_1 \}\frac{pq}{x^1} \\
    &= \{(\deg p - 1)m_1 - (\deg q -1)n_1\}\frac{pq}{x^1}.
\end{align*}
To recover the form in formula in Example~\ref{ex:arbitrary-replaced-monomial} (on p.~\pageref{ex:arbitrary-replaced-monomial}) apply the change of coordinates $x^1\leftrightarrows x^j$ and move the monomial $p(\bx)$ to the position at $x^j$. The sign changes from both actions cancel identically as they correspond to the same cyclic permutation. This completes the proof.
\hfill$\square$

\section{Proof of Lemma~\ref{lem:consistency}}\label{SecAppConsistency}
\begin{proof} 
    In each co\"ordinate $x^i$ we require $\sum_{j=1}^N \deg_i(a_j)=kN/(d+1)$ and summatively $\sum_{j=1}^N \deg(a_j)=kdN/(d+1)$. We take all constants out of the Wronskian, making each $a_j$ monic. Represent $a_j(\bx)=\bx^{\bkvec_j}=(x^i)^{k_j^1}\cdots (x^d)^{k_j^d}$, where $\bkvec_j=(k_j^1,...,k_j^d)\in \BBN_{\geqslant 0}^d$. Define a `norm' $||\bkvec_j||=k_j^1+\cdots+ k_j^d=\deg(a_j)$. The summative degree-sum conditions states that $\sum_{j=1}^N ||\bkvec_j||=kdN/(d+1)$. Set $S=\{\bkvec_1,...,\bkvec_N\}$ and define the number of monomials of degree $t$: $n_t:=\#\{\bkvec\in S: ||\bkvec||=t\}$. 

    We note that if all $||\bkvec||\leq k$, then there are only $N$ possible (independent) choices -- the $\bkvec$'s of the standard monomials (cf.\ Notation~\ref{def:divpowermonomial-standardmonomial} on p.~\pageref{def:divpowermonomial-standardmonomial}), i.e.\ $\bkvec_j=\brvec_j$. We call this set $S_0=\{\brvec_j:1\leqslant j \leqslant N \}$ and define $m_t=\#\{\bkvec\in S_0:||\bkvec||=t\}$.

    Denote the summative degree-sum $T_S=\sum_{\bkvec\in S}=\sum_{t\geq 0} tn_t$ and $T_{S_0}=\sum_{\brvec_j\in S_0}||\brvec_j||$; the numbers must satisfy $T_S=T_{S_0}=kdN/(d+1)$.
    We proceed to show that if any $||\bkvec_j||>k$, then $T_S>T_{S_0}$, giving a contradiction.

    (1.) If all $||\bkvec_j||>k$, then $T_S=\sum_{t\geqslant 0} tn_t> kN>kN\cdot d/(d+1)=T_{S_0
    }$, a contradiction.

    (2.) Let $t_1,...,t_\ell >k$ be all degrees $t>k$ for which $n_t>0$. Suppose for contradiction that such degrees exist, i.e.\, $\ell>0$. As $\#S=N$ we have
    \begin{equation}
        n_0+n_1+\cdots + n_k=N-(n_{t_1}+\cdots+n_{t_\ell})=:N-n.
    \end{equation}
    Take any $\bkvec\in S$ with $||\bkvec||=t_j$, $1\leqslant j\leqslant \ell$. Then some $n_s$ for $s\leqslant k$ must be strictly smaller than $m_s$. As $s\leqslant k$ we get $t_j-s\geqslant t_j-k$. We call $s=s_j^\nu$, where for each such $||\bkvec||=t_j$ we assign an index $1\leqslant \nu \leqslant n_{t_j}$.

    (3.) We now expand the sum
    \begin{equation}
    \begin{split}
        T_S &=\sum_{t\geq 0} tn_t = \sum_{t=0}^k tn_t + \sum_{j=1}^\ell t_j n_{t_j}
        = T_{S_0} - \sum_{j=1}^\ell \left\{\sum_{\nu=1}^{n_{t_j}} s_j^\nu \right\} + \sum_{j=1}^\ell t_j n_{t_j} \\
        &= T_{S_0} + \sum_{j=1}^\ell \left\{ t_j n_{t_j} - \sum_{\nu=1}^{n_{t_j}} s_j^\nu  \right\},
    \end{split}
    \end{equation}
    where for each $j$ value
    \begin{equation}
    \begin{split}
        tn_{t_j}- \sum_{\nu=1}^{n_{t_j}} s_j^\nu 
        &= t_j - s_j^1 + t_j - s_j^2 + \cdots + t_j - s_j^{n_{t_j}} \\
        &\geqslant t_j-k + \cdots + t_j-k = (t_j-k)n_{t_j}>0,
    \end{split}
    \end{equation}
    hence $T_S \geqslant T_{S_0} + \sum_{j=1}^\ell (t_j-k)n_{t_j}$ gives $T_S>T_{S_0}$, a contradiction. Therefore $\ell=0$.
\end{proof}


\section{Python scripts}\label{SecAppPyScripts}

\begin{lstlisting}[language=Python, caption={Computation of the Wronskian determinant over base dimension $d$ with coordinates $(x,y,z,...)$ of differential order $k$; monomial $q$ replaces the standard argument $1$ and $p$ replaces any chosen standard argument.},captionpos=t, label={script:Wronskian-pq}]
import sympy as sp
import itertools as its
import math


x, y, z, w, r, t, u, s  = sp.symbols('x y z w r t u s')
n1,n2,n3,n4,n5,n6,n7,n8 = sp.symbols('n1 n2 n3 n4 n5 n6 n7 n8')
m1,m2,m3,m4,m5,m6,m7,m8 = sp.symbols('m1 m2 m3 m4 m5 m6 m7 m8')
degp, degq = sp.symbols('degp degq')
p, q = sp.symbols('p q')


def difflist(L:list, var):
    """
    Differentiates the arguments in the list L by the differential operator
    composed of var: for instance, the second mixed dervative in x and y is
    given by var=(x,y,).
    """
    out = []
    for i in range(len(L)):
        if type(var)==tuple:out.append(sp.diff(L[i],*var))
        else:
            out.append(sp.diff(L[i],var))

    return out


def wronskian_any(k,ns,ms,xis=[x,y,z,w,r,t,u,s],prepl=(x,),excls=[],printout=True):
    """
    Computes the Wronskian-determinant multibracket of order k of
    [p, ... ,q], where p = (x^i)^{n_i} and q = (x^i)^{m_i} for n_i in ns
    and m_i in ms; the length of ns and ms determine the dimension d.
        By setting prepl to a variable combination other than (x,), the
    term p may replace prepl rather than x (1st position).
        For incomplete Wronskians, excls defines the list of excluded
    differentials.
    """
    pv = 1
    qv = 1
    if len(ns)!=len(ms):
        print("p,q dimension mismatch!")
        return None

    xis = xis[0:len(ns)]
    for i,ni in enumerate(ns):
        pv *= xis[i]**ni
    for i,mi in enumerate(ms):
        qv *= xis[i]**mi

    row = [pv]
    ders = []
    for ordr in range(1,k+1):
        for comb in its.combinations_with_replacement(xis, ordr):
            if comb not in excls: #exclusions in incomplete Wronskian
                ders.append(comb)

                if comb==(*prepl,):continue #replacement by p
                else:
                    row.append(math.prod(comb))
    row.append(qv)


    M = sp.Matrix([row]+[difflist(row,der) for der in ders])

    det_M = sp.factor(M.det())

    if printout:
        #sp.pprint(M)
        print(f"Wronskian k={k}, d={len(ns)}")
        print(f"p replaces term: {prepl}")
        print(f"excluded derivatives: {excls}")
        print(f"{row} = {det_M}")

    return det_M


# execution example
w = wronskian_any(2,[n1,n2],[m1,m2],prepl=(x,))
\end{lstlisting}

\begin{lstlisting}[language=Python, caption={Computation of the Wronskian of $N$ arbitrary arguments under \texttt{row} for any choice of differential order $k$ and base dimension $d$.},captionpos=t, label={script:Wronskian-any-arguments}]
import sympy as sp
import itertools as its
import math


x, y, z, w, r, t, u, s  = sp.symbols('x y z w r t u s')
n1,n2,n3,n4,n5,n6,n7,n8 = sp.symbols('n1 n2 n3 n4 n5 n6 n7 n8')
m1,m2,m3,m4,m5,m6,m7,m8 = sp.symbols('m1 m2 m3 m4 m5 m6 m7 m8')
l1,l2,l3,l4,l5,l6,l7,l8 = sp.symbols('l1 l2 l3 l4 l5 l6 l7 l8')


def difflist(L:list, var):
    out = []
    for i in range(len(L)):
        if type(var)==tuple:out.append(sp.diff(L[i],*var))
        else:
            out.append(sp.diff(L[i],var))

    return out


def wronskian_row(k,d,row,xis=[x,y,z,w,r,t,u,s],printout=True):
    """
    Calculates the Wronskian of the list of N arguments in row.
    """
    
    xis = xis[0:d]

    ders = []
    coeff = 1
    for ordr in range(1,k+1):
        for comb in its.combinations_with_replacement(xis, ordr):
            ders.append(comb)

            coeff *= math.factorial(comb.count(x))
            coeff *= math.factorial(comb.count(y))

    M = sp.Matrix([row]+[difflist(row,der) for der in ders])

    det_M = sp.factor(M.det())

    if printout:
        #sp.pprint(M)
        print(f"Wronskian k={k}, d={d}; coeff = {coeff}")
        print(f"{row} = {det_M}")

    return det_M


# execution example
row = [x**l1*y**l2,x**m1*y**m2,x**n1*y**n2]
wr = wronskian_row(1,2,row)
\end{lstlisting}

\medskip
\noindent$\boldsymbol{\copyright}$\qquad
The copyright for all newly designed software modules which are 
presented in this paper is retained by M.\,G.\,\cb{K}\=eni\cb{n}\v{s}; 
provisions of the~MIT free software license apply.

\section{Raw computational outputs}\label{SecAppRawOutputs}

\noindent
Set the base dimension $d=2$ and the differential order $k=2,3,4,5$. Compute the Wronskian determinant of arguments $p,y,x^2,xy,...,y^k,q$, where $p=x^{n_1}y^{n_2}$ and $q=x^{m_1}y^{m_2}$ are arbitrary monomials. The raw outputs of such computations by Script~\ref{script:Wronskian-pq}, executing the input commands,
\begin{lstlisting}[language=Python]
w = wronskian_any(2,[n1,n2],[m1,m2],prepl=(x,))
w = wronskian_any(3,[n1,n2],[m1,m2],prepl=(x,))
w = wronskian_any(4,[n1,n2],[m1,m2],prepl=(x,))
w = wronskian_any(5,[n1,n2],[m1,m2],prepl=(x,))
\end{lstlisting}
are given below:\footnote{In Python the symbol \texttt{*} denotes multiplication, whereas \texttt{**} denotes exponentiation (e.g.\ $x^2 = \text{\texttt{x**2}}$).}
\begin{lstlisting}
Wronskian k=2, d=2
p replaces term: (x,)
excluded derivatives: []
[x**n1*y**n2, y, x**2, x*y, y**2, x**m1*y**m2] = -2*x**m1*x**n1*y**m2*y**n2*(m1 + m2 - 2)*(n1 + n2 - 2)*(m1*n2 - m1 - m2*n1 + n1)/x


Wronskian k=3, d=2
p replaces term: (x,)
excluded derivatives: []
[x**n1*y**n2, y, x**2, x*y, y**2, x**3, x**2*y, x*y**2, y**3, x**m1*y**m2] = -48*x**m1*x**n1*y**m2*y**n2*(m1 + m2 - 3)*(m1 + m2 - 2)*(n1 + n2 - 3)*(n1 + n2 - 2)*(m1*n2 - m1 - m2*n1 + n1)/x


Wronskian k=4, d=2
p replaces term: (x,)
excluded derivatives: []
[x**n1*y**n2, y, x**2, x*y, y**2, x**3, x**2*y, x*y**2, y**3, x**4, x**3*y, x**2*y**2, x*y**3, y**4, x**m1*y**m2] = 331776*x**m1*x**n1*y**m2*y**n2*(m1 + m2 - 4)*(m1 + m2 - 3)*(m1 + m2 - 2)*(n1 + n2 - 4)*(n1 + n2 - 3)*(n1 + n2 - 2)*(m1*n2 - m1 - m2*n1 + n1)/x


Wronskian k=5, d=2
p replaces term: (x,)
excluded derivatives: []
[x**n1*y**n2, y, x**2, x*y, y**2, x**3, x**2*y, x*y**2, y**3, x**4, x**3*y, x**2*y**2, x*y**3, y**4, x**5, x**4*y, x**3*y**2, x**2*y**3, x*y**4, y**5, x**m1*y**m2] = 19813556551680*x**m1*x**n1*y**m2*y**n2*(m1 + m2 - 5)*(m1 + m2 - 4)*(m1 + m2 - 3)*(m1 + m2 - 2)*(n1 + n2 - 5)*(n1 + n2 - 4)*(n1 + n2 - 3)*(n1 + n2 - 2)*(m1*n2 - m1 - m2*n1 + n1)/x
\end{lstlisting}

The raw outputs of Script~\ref{script:Wronskian-pq}, setting the first optional argument to \texttt{prepl=(x,x,)}, namely, the monomial $p$ replaces the standard argument $(x)^2$, were
\begin{lstlisting}
Wronskian k=2, d=2
p replaces term: (x, x)
excluded derivatives: []
[x**n1*y**n2, x, y, x*y, y**2, x**m1*y**m2] = x**m1*x**n1*y**m2*y**n2*(2*m1**2*n1*n2 - 2*m1**2*n1 + m1**2*n2**2 - 3*m1**2*n2 + 2*m1**2 - 2*m1*m2*n1**2 + 2*m1*m2*n1 + 2*m1*n1**2 - 2*m1*n1*n2 - m1*n2**2 + 3*m1*n2 - 2*m1 - m2**2*n1**2 + m2**2*n1 + 3*m2*n1**2 - 3*m2*n1 - 2*n1**2 + 2*n1)/x**2

Wronskian k=3, d=2
p replaces term: (x, x)
excluded derivatives: []
[x**n1*y**n2, x, y, x*y, y**2, x**3, x**2*y, x*y**2, y**3, x**m1*y**m2] = 48*x**m1*x**n1*y**m2*y**n2*(m1 + m2 - 3)*(n1 + n2 - 3)*(2*m1**2*n1*n2 - 2*m1**2*n1 + m1**2*n2**2 - 3*m1**2*n2 + 2*m1**2 - 2*m1*m2*n1**2 + 2*m1*m2*n1 + 2*m1*n1**2 - 2*m1*n1*n2 - m1*n2**2 + 3*m1*n2 - 2*m1 - m2**2*n1**2 + m2**2*n1 + 3*m2*n1**2 - 3*m2*n1 - 2*n1**2 + 2*n1)/x**2


Wronskian k=4, d=2
p replaces term: (x, x)
excluded derivatives: []
[x**n1*y**n2, x, y, x*y, y**2, x**3, x**2*y, x*y**2, y**3, x**4, x**3*y, x**2*y**2, x*y**3, y**4, x**m1*y**m2] = -497664*x**m1*x**n1*y**m2*y**n2*(m1 + m2 - 4)*(m1 + m2 - 3)*(n1 + n2 - 4)*(n1 + n2 - 3)*(2*m1**2*n1*n2 - 2*m1**2*n1 + m1**2*n2**2 - 3*m1**2*n2 + 2*m1**2 - 2*m1*m2*n1**2 + 2*m1*m2*n1 + 2*m1*n1**2 - 2*m1*n1*n2 - m1*n2**2 + 3*m1*n2 - 2*m1 - m2**2*n1**2 + m2**2*n1 + 3*m2*n1**2 - 3*m2*n1 - 2*n1**2 + 2*n1)/x**2


Wronskian k=5, d=2
p replaces term: (x, x)
excluded derivatives: []
[x**n1*y**n2, x, y, x*y, y**2, x**3, x**2*y, x*y**2, y**3, x**4, x**3*y, x**2*y**2, x*y**3, y**4, x**5, x**4*y, x**3*y**2, x**2*y**3, x*y**4, y**5, x**m1*y**m2] = -39627113103360*x**m1*x**n1*y**m2*y**n2*(m1 + m2 - 5)*(m1 + m2 - 4)*(m1 + m2 - 3)*(n1 + n2 - 5)*(n1 + n2 - 4)*(n1 + n2 - 3)*(2*m1**2*n1*n2 - 2*m1**2*n1 + m1**2*n2**2 - 3*m1**2*n2 + 2*m1**2 - 2*m1*m2*n1**2 + 2*m1*m2*n1 + 2*m1*n1**2 - 2*m1*n1*n2 - m1*n2**2 + 3*m1*n2 - 2*m1 - m2**2*n1**2 + m2**2*n1 + 3*m2*n1**2 - 3*m2*n1 - 2*n1**2 + 2*n1)/x**2

\end{lstlisting}

\end{document}